\documentclass[a4paper,10pt]{article}

\usepackage{amsmath}  
\usepackage{amssymb}            
\usepackage{amsxtra}
\usepackage{amscd}
\usepackage[mathscr]{euscript}
\usepackage[dvipdfmx]{graphicx}
\usepackage{tikz}
\usetikzlibrary{cd,matrix,arrows,decorations.pathmorphing}
\usetikzlibrary{cd}
\usepackage{tikz-cd}

\everymath{\displaystyle} 

\newtheorem{thm}{Theorem}[subsection]
\newtheorem{prop}[thm]{Proposition}

\newtheorem{lem}[thm]{Lemma}
\newtheorem{dfn}[thm]{Definition}

\newtheorem{rem}[thm]{Remark}
\newtheorem{conj}[thm]{Conjecture}

\numberwithin{equation}{section}

\newcommand{\npmod}[1]{\!\!\pmod{#1}}
\newcommand{\nnpmod}[1]{\!\!\!\!\pmod{#1}}

\newenvironment{proof}{\par\noindent{\bf[Proof]}}%
                      {$\blacksquare$\noindent\par\vspace{0.5\baselineskip}}
                      {$\blacksquare$\par\noindent}

\font\b=cmr10 scaled \magstep4

\def\bigzerou{\smash{\lower0.7ex\hbox{\b 0}}}
\def\bigastl{\smash{\lower0.7ex\hbox{\b *}}}
\def\bigastu{\smash{\lower2.7ex\hbox{\b *}}}

\def\addots{\mathinner
    {\mkern1mu\raise1pt\hbox{.}\mkern2mu
    \raise4pt\hbox{.}\mkern2mu\raise7pt\vbox{\kern7pt\hbox{.}}\mkern1mu}}

\makeatletter
\newcommand{\numsubsection}{\@startsection%
  {subsection}%
  {2}%
  {0mm}%
  {\baselineskip}%
  {-0.2\parindent}%
{\normalfont\large\upshape\bfseries}}%
\def\@dotsep{1.5}
\def\@pnumwidth{1em}
\makeatother

\makeatletter
\newcommand{\numsubsubsection}{\@startsection%
  {subsubsection}%
  {2}%
  {0mm}%
  {\baselineskip}%
  {-0.2\parindent}%
{\normalfont\normalsize\upshape\bfseries}}%
\def\@dotsep{1.5}
\def\@pnumwidth{1em}
\makeatother



\title{On certain supercuspidal representations of symplectic groups
  associated with\\ tamely ramified extensions :\\
  the formal degree conjecture and\\
  the root number conjecture} 
\author{Koichi Takase
        \thanks{The author is partially supported by 
                    JSPS KAKENHI Grant Number JP 16K05053}}
\date{}

\begin{document}   


%
%

\maketitle


\section{Introduction}
\label{sec:introduction}

\numsubsection{}
\label{subsec:conjectures-for-split-semi-simple-group}
Let $F/\Bbb Q_p$ be a finite extension with $p\neq 2$ whose 
integer ring $O_F$ has unique maximal ideal $\frak{p}_F$ wich
is generated by $\varpi_F$. The residue class field 
$\Bbb F=O_F/\frak{p}_F$ is a finite field of $q$-elements. The Weil
  group of $F$ is denoted by $W_F$ which is a subgroup of the absolute
  Galois group $\text{\rm Gal}(\overline F/F)$ where 
$\overline F$ is a fixed algebraic closure of $F$ in which we will
  take the algebraic extensions of $F$.

Let $G$ be a connected semi-simple linear algebraic group defined over
$F$. For the sake of simplicity, we will assume that $G$ splits over
$F$. Then the $L$-group $^LG$ of $G$ is equal to the dual group 
$G\sphat$ of $G$. 
An admissible representation
$$
 \varphi:W_F\times SL_2(\Bbb C)\to\,^LG
$$
of the Weil-Deligne group of $F$ is called a discrete parameter of $G$
over $F$ if the 
centralizer $\mathcal{A}_{\varphi}=Z_{\,^LG}(\text{\rm Im}\varphi)$ of
the image of $\varphi$ in  $^LG$ is a finite
group. Let us denote by $\mathcal{D}_F(G)$ the $G\sphat$-conjugacy classes
of the discrete parameters of $G$ over $F$. The conjectural
parametrization of $\text{\rm Irr}_2(G)$ (resp. $\text{\rm Irr}_s(G)$),
the set of the equivalence 
classes of the irreducible admissible square-integrable (resp. 
supercuspidal) representations of $G$, by $\mathcal{D}_F(G)$ is 
(see \cite[p.483, Conj.7.1]{Gross-Reeder2010} for the details)

\begin{conj}
\label{conjecture:lamglamds-parameter-of-square-integrable-rep}
For every $\varphi\in\mathcal{D}_F(G)$, there exists a finite subset 
$\Pi_{\varphi}$ of $\text{\rm Irr}_2(G)$ such that
\begin{enumerate}
\item $\text{\rm Irr}_2(G)
       =\bigsqcup_{\varphi\in\mathcal{D}_F(G)}\Pi_{\varphi}$,
\item there exists a bijection of $\Pi_{\varphi}$ onto the equivalence
  classes $\mathcal{A}_{\varphi}\sphat$ 
  of the irreducible complex linear representations of 
  $\mathcal{A}_{\varphi}$,
\item $\Pi_{\varphi}\subset\text{\rm Irr}_s(G)$ if 
      $\varphi|_{SL_2(\Bbb C)}=1$.
\end{enumerate}
The finite set $\Pi_{\varphi}$ is called a $L$-packet of $\varphi$.
\end{conj}

According to this conjecture, any $\pi\in\text{\rm Irr}_2(G)$ is
determined by $\varphi\in\mathcal{D}_F(G)$ and 
$\sigma\in\mathcal{A}_{\varphi}\sphat$. So the formal degree of $\pi$
should be determined by these data. The formal degree conjecture due
to Hiraga-Ichino-Ikeda \cite{HiragaIchinoIkeda2008} is 
(with the formulation of \cite{Gross-Reeder2010}) 

\begin{conj}\label{conjecture:formal-degree-conjecture}
The formal degree $d_{\pi}$
of $\pi$ with respect to the absolute value of the Euler-Poincar\'e
measure (see {\rm \cite[$\S 3$]{Serre1971}} for the details) 
on $G(F)$ is equal to
$$
 \frac{\dim\sigma}
      {|\mathcal{A}_{\varphi}|}\cdot
 \left|
 \frac{\gamma(\varphi,\text{\rm Ad},\psi,d(x),0)}
      {\gamma(\varphi_0,\text{\rm Ad},\psi,d(x),0)}\right|.
$$
\end{conj}

Here 
$$
 \gamma(\varphi,\text{\rm Ad},\psi,d(x),s)
 =\varepsilon(\varphi,\text{\rm Ad},d(x),s)\cdot
  \frac{L(\varphi^{\vee},\text{\rm Ad},1-s)}
       {L(\varphi,\text{\rm Ad},s)}
$$
is the gamma-factor associated with the $\varphi$ combined with the
adjoint representation $\text{\rm Ad}$ of $G\sphat$ on 
its Lie algebra $\frak{g}\sphat$, and a continuous
additive character $\psi$ of $F$ such that 
$\{x\in F\mid\psi(xO_F)=1\}=O_F$ 
and the Haar measure $d(x)$ on the additive group $F$ such that 
$\int_{O_F}d(x)=1$. See \cite[pp.440-441]{Gross-Reeder2010} for the details.
$$
 \varphi_0:W_F\times SL_2(\Bbb C)
           \xrightarrow{\text{\rm proj.}}SL_2(\Bbb C)
           \to G\sphat
$$
is the principal parameter (see \cite[p.447]{Gross-Reeder2010} for the
definition). 

The formal degree conjecture concerns with the absolute value
of the epsilon-factor
$$
 \varepsilon(\varphi,\text{\rm Ad},d(x),s)
 =w(\varphi,\text{\rm Ad})\cdot q^{a(\varphi,\text{\rm Ad})(\frac 12-s)}
$$
where $a(\varphi,\text{\rm Ad})$ is the Artin-conductor and 
$w(\varphi,\text{\rm Ad})$ is the root number. 

In order to state the root number conjecture, we need some
notations. Let $T\subset G$ be a maximal torus split over $F$ with
respect to which the root datum 
$$
 (X(T),\Phi(T),X^{\vee}(T),\Phi^{\vee}(T))
$$
is defined. Then the dual group $G\sphat$ is, by the definition, the
connected reductive complex algebraic group with a maximal torus
$T\sphat$ with which its root datum is isomorphic to
$$
 (X^{\vee}(T),\Phi^{\vee}(T),X(T),\Phi(T)).
$$
Put $2\cdot\rho=\sum_{0<\alpha\in\Phi^{\vee}(T)}\alpha$, then 
$\epsilon=2\cdot\rho(-1)\in T$ is a central element of $G$. Now the
root number conjecture says that 

\begin{conj}{\rm \cite[p.493, Conj.8.3]{Gross-Reeder2010}}
\label{conjecture:root-number-conjecture}
$$
 \frac{w(\varphi,\text{\rm Ad})}
      {w(\varphi_0,\text{\rm Ad})}=\pi(\epsilon)
$$ 
where $\epsilon$ is the central element of $G$ defined above 
(see {\rm \cite[p.492, (65)]{Gross-Reeder2010}} for the details).
\end{conj}

Since $G$ is assumed to be splits over $F$, we have 
$w(\varphi_0,\text{\rm Ad})=1$ (see \cite[p.448]{Gross-Reeder2010}).

\numsubsection{}
In this paper, we will construct quite explicitly supercuspidal
representations of $G(F)=Sp_{2n}(F)$ associated with a tamely ramified
extension $K/F$ of degree $2n$ (Theorem 
\ref{th:supercuspidal-representation-of-sp(2n)}). Here $K$ is a
quadratic extension of over field $K_+$ of $F$. When $K/F$ is normal, 
we will also give  candidates of Langlands parameters of 
the supercuspidal representations (the section 
\ref{sec:kaleta-l-parameter}), and will verify the validity of the
formal degree conjecture (Theorem
\ref{th:formal-degree-conjecture-for-sp(2n)}) 
and the root number conjecture (Theorem 
\ref{th:root-number-conjecture-for-sp(2n)}) with them. Surprisingly
the root number conjecture is valid only if $K/F$ is not totally
ramified or $K/F$ is totally ramified and
$$
 \frac{q-1}2\cdot(n-1)\equiv 0\npmod{4}.
$$

Our supercuspidal representations, denoted by $\pi_{\beta,\theta}$,
are given by the compact induction 
$\text{\rm ind}_{G(O_F)}^{G(F)}\delta_{\beta,\theta}$ from irreducible
unitary representations $\delta_{\beta,\theta}$ of the hyperspecial
compact subgroup $G(O_F)=Sp_{2n}(O_F)$. Here $\pi_{\beta,\theta}$ and
$\delta_{\beta,\theta}$ are characterized each other by the conditions
\begin{enumerate}
\item $\delta_{\beta,\theta}$ factors through the canonical morphism 
      $G(O_F)\to G(O_F/\frak{p}_F^r)$ with $r\geq 2$, and 
      the multiplicity of $\delta_{\beta,\theta}$ in 
      $\pi_{\beta,\theta}|_{G(O_F)}$ is one, 
\item any irreducible unitary representation $\delta$ of $G(O_F)$
      which factors through the canonical morphism 
      $G(O_F)\to G(O_F/\frak{p}_F^r)$, and a constituent of 
      $\pi_{\beta,\theta}|_{G(O_F)}$, then
      $\delta=\delta_{\beta,\theta}$. 
\end{enumerate}
The parameters $\beta$ and $\theta$ are associated with the tamely
ramified extension $K/F$, that is, $O_K=O_F[\beta]$ and 
$\theta$ is a certain continuous unitary character of 
$$
 U_{K/K_+}=\{x\in K^{\times}\mid N_{K/K_+}(x)=\}
$$
(see the subsection 
\ref{subsec:symplectic-space-associated-with-tamely-ramified-ext} for
the precise definitions). 
We have the irreducible representation $\delta_{\beta,\theta}$ by the
general theory given by \cite{Takase2021}. 

The candidate of Langlands parameter is given by the method of Kaletha 
\cite{Kaletha2019}. Regard the compact group $U_{K/K_+}$ as the
group of $F$-rational points of an elliptic torus of $Sp_{2n}$. Then, by
the local Langlands correspondence of tori 
(see \cite{Yu2009}) and the Langlands-Schelstad procedure 
(\cite{LanglandsShelstad1987}) gives a group homomorphism $\varphi$ of
the Weil group $W_F$ of $F$ to the dual group 
$G\sphat=SO_{2n+1}(\Bbb C)$ of $Sp_{2n}$ over $F$. 

Although 
a general theory of construction of the supercuspidal representation
is given by \cite{Yu2001,Kaletha2019}, that of ours is based upon a
method of \cite{Shintani1968} which has an advantage of being more
direct and explicit. 

Note that the formal degree conjecture is proved by \cite{Schwein2021}
between the supercuspidal representations of \cite{Kaletha2019} and
the Langlands parameters of Kaletha. In this paper, supercuspidal
representations are constructed by a method different from that of 
\cite{Kaletha2019}, so it is of some interest.

\numsubsection{}
The section \ref{sec:supercuspidal-representation-of-sp(2n)} is devoted
to the construction of the supercuspidal representation 
$\pi_{\beta,\theta}$ of $Sp_{2n}(F)$. After recalling, in the subsection 
\ref{subsec:regular-irred-character-of-hyperspecial-compact-subgroup}, 
the general theory
of the regular irreducible representations of the finite group 
$G(O_F/\frak{p}_F^r)$ ($r\geq 2$) given by 
\cite{Takase2021}, we will define the irreducible unitary
representation $\delta_{\beta,\theta}$ of $Sp_{2n}(O_F)$ in the
subsection
\ref{subsec:symplectic-space-associated-with-tamely-ramified-ext}. The
construction of the supercuspidal representation $\pi_{\beta,\theta}$
is given in the subsection
\ref{subsec:construction-of-supercuspidal-representataion}. 

The candidate of Langlands parameter is given in the section 
\ref{sec:kaleta-l-parameter}. The local Langlands correspondence of
elliptic torus (Proposition
\ref{prop:local-langlands-correspondence-of-elliptic-tori}) and the
Langlands-Schelstad procedure (the subsection 
\ref{subsec:chi-datum}) are given quite explicitly. They give a 
candidate of Langlands parameter
$$
 \varphi:W_F\xrightarrow{\text{\rm canonical}}
         W_{K/F}
         \xrightarrow{\varphi_1\oplus\det\varphi_1}
         SO_{2n+1}(\Bbb C)
$$
where 
$\varphi_1=\text{\rm Ind}_{K^{\times}}^{W_{K/F}}\widetilde\vartheta$ is
the induced representation from a character $\widetilde\vartheta$ of
$K^{\times}$ to the relative Weil group 
$W_{K/F}=W_F/\overline{[W_K,W_K]}$. The character
$\widetilde\vartheta$ is defined by 
$\widetilde\vartheta(x)=\vartheta(x^{1-\tau})$ where 
$\text{\rm Gal}(K/K_+)=\langle\tau\rangle$ and
$\vartheta=c\cdot\theta$ with the character $c$ of $U_{K/K_+}$ which
is generated by the Langlands-Schelstad procedure. 

Using the explicit description of the parameter $\varphi$, we will
verify the formal degree conjecture in the section 
\ref{sec:formal-degree-conjecture}, and the root number conjecture in
the section \ref{sec:root-number-conjecture}. 

In section \ref{sec:case-of-sp(4)}, we will discuss the case of $n=2$
where we can define another ``natural" candidate for the Langlands
parameter of $\pi_{\beta,\theta}$. The representation space of 
$\text{\rm Ind}_{K^{\times}}^{W_{K/F}}\widetilde\theta$, with 
$\widetilde\theta(x)=\theta(x^{1-\tau})$, has $W_{K/F}$-quasi
invariant symplectic form. Then the candidate is given by
\begin{equation}
 W_F\xrightarrow{\text{\rm can.}}W_{K/F}
  \xrightarrow{\text{\rm Ind}_{K^{\times}}^{W_{K/F}}\widetilde\theta}
  GSp_4(\Bbb C)\xrightarrow{(\ast)}SO_5(\Bbb C)
\label{eq:another-parameter-sp(4)-introduction}
\end{equation}
where $(\ast)$ is the accidental surjection. With respect to this
parameter 
\begin{enumerate}
\item the formal degree conjecture is valid only if $K/F$ is
  unramified or totally ramified, and in this case
\item the root number cinjecture is valid only if
$$
 \theta(-1)=\begin{cases}
             1&:\text{\rm $K/F$ is unramified,}\\
            (-1)^{\frac{q-1}4}&:\text{\rm $K/F$ is totally ramifed.}
            \end{cases}
$$
\end{enumerate}
This means that the parameter 
\eqref{eq:another-parameter-sp(4)-introduction} 
is not the Langlands parameter of $\pi_{\beta,\theta}$, in general.

Several basic facts on the local factor associated with
representations of the Weil group are given in the appendix 
\ref{sec:local-factors}.

The quasi-invariant symmetric or symplectic form in the induced
representation on $W_{K/F}$ from the characters of $K^{\times}$ is
discussed in the appendix 
\ref{sec:symmetric-or-anti-symmetric-form-on-induced-rep-of-weil-groip}.

\section{Supercuspidal representations of $Sp_{2n}(F)$}
\label{sec:supercuspidal-representation-of-sp(2n)}

\subsection{Regular irreducible characters of hyperspecial compact
            subgroup}
\label{subsec:regular-irred-character-of-hyperspecial-compact-subgroup}
Let us recall the main results of \cite{Takase2021}.

Fix a continuous unitary additive character $\psi:F\to\Bbb C^1$ such
that
$$
 \{x\in F\mid\psi(xO_F)=1\}=O_F.
$$
Let $G=Sp_{2n}$ be the $O_F$-group scheme such that, for any
$O_F$-algebra 
\footnote{In this paper, an $O_F$-algebra means an unital commutative
  $O_F$-algebra.} 
$R$, the group of the $A$-valued point $G(A)$ is a subgroup of
$GL_{2n}(R)$ defined by
$$
 G(R)=\{g\in GL_{2n}(R)\mid gJ_n\,^tg=J_n\}
$$
where
$$
 J_n=\begin{bmatrix}
      0&I_n\\
     -I_n&0
     \end{bmatrix},
 \;\;\text{\rm where}\;\;
 I_n=\begin{bmatrix}
       &       &1\\
       &\addots& \\
      1&       &
     \end{bmatrix}.
$$
For a matrix $A\in M_{m,n}(R)$, put 
$^{\frak t}A=I_n\,^tAI_m\in M_{n,m}(R)$. 
Let $\frak{g}$ the Lie algebra scheme of $G$ which is a closed affine
$O_F$-subscheme of $\frak{gl}_n$ the Lie algebra scheme of $GL_n$
defined by
$$
 \frak{g}(R)=\{X\in\frak{gl}_{2n}(R)\mid XJ_n+J_n\,^tX=0\}
$$
for all $O_F$-algebra $R$. 
Let 
$$
 B:\frak{gl}_{2n}{\times}_{O_F}\frak{gl}_{2n}\to\Bbb A_{O_F}^1
$$
be the trace form on $\frak{gl}_{2n}$, that is $B(X,Y)=\text{\rm tr}(XY)$
for all $X,Y\in\frak{gl}_{2n}(R)$ with any $O_F$-algebra $R$. Since $G$
is smooth $O_F$-group scheme, we have a canonical isomorphism
$$
 \frak{g}(O_F)/\varpi^r\frak{g}(O_F)\,\tilde{\to}\,
 \frak{g}(O_F/\frak{p}^r)=\frak{g}(O_F){\otimes}_{O_F}O_F/\frak{p}^r
$$
(\cite[Chap.II, $\S 4$, Prop.4.8]{Demazure-Gabriel1970}) and the
 canonical group homomorphism $G(O_F)\to G(O_F/\frak{p}^r)$ is
 surjective, due to the formal smoothness 
\cite[p.111, Cor. 4.6]{Demazure-Gabriel1970}, whose kernel is denoted
by $K_r(O_F)$. 
For any $0<l<r$, let us denote by $K_l(O_F/\frak{p}^r)$ the kernel of
the canonical 
 group homomorphism $G(O_F/\frak{p}^r)\to G(O_F/\frak{p}^l)$ which is
 surjective.  

The following basic assumptions on $G$ are satisfied; 
\begin{itemize}
\item[I)] $B:\frak{g}(\Bbb F)\times\frak{g}(\Bbb F)\to\Bbb F$ is
  non-degenerate, 
\item[II)] for any integers $r=l+l^{\prime}$ with 
           $0<l^{\prime}\leq l$, we have a group isomorphism
$$
 \frak{g}(O_F/\frak{p}^{l^{\prime}})\,\tilde{\to}\,K_l(O_F/\frak{p}^r)
$$
           defined by 
$X\npmod{\frak{p}^{l^{\prime}}}\mapsto1+\varpi^lX\npmod{\frak{p}^r}$,
\item[III)] if $r=2l-1\geq 3$ is odd, then we have a mapping
$$
 \frak{g}(O_F)\to K_{l-1}(O_F/\frak{p}^r)
$$
defined by 
$X\mapsto(1+\varpi^{l-1}X+2^{-1}\varpi^{2l-2}X^2)\npmod{\frak{p}^r}$.
\end{itemize}
The condition I) implies that 
$B:\frak{g}(O_F/\frak{p}^l)\times\frak{g}(O_F/\frak{p}^l)
   \to O_F/\frak{p}^l$ 
is non-degenerate for all $l>0$, and so 
$B:\frak{g}(O_F)\times\frak{g}(O_F)\to O_F$ is also non-degenerate. 
By the condition II), $K_l(O_F/\frak{p}^r)$ is a commutative normal
subgroup of $G(O_F/\frak{p}^r)$, and its character is
$$
 \chi_{\beta}(1+\varpi^lX\npmod{\frak{p}^r})
 =\psi\left(\varpi^{-l^{\prime}}B(X,\beta)\right)
 \quad
 (X\npmod{\frak{p}^{l^{\prime}}}
  \in\frak{g}(O_F/\frak{p}^{l^{\prime}}))
$$
with 
$\beta\npmod{\frak{p}^{l^{\prime}}}
 \in\frak{g}(O_F/\frak{p}^{l^{\prime}})$.

Since 
any finite dimensional complex continuous representation of
the compact group $G(O_F)$ factors through the canonical group
homomorphism $G(O_F)\to G(O_F/\frak{p}^r)$ for some $0<r\in\Bbb Z$, we
want to know the irreducible complex representations of the finite
group $G(O_F/\frak{p}^r)$. 
Let us assume
that $r>1$ and put $r=l+l^{\prime}$ with the minimal integer $l$ such
that $0<l^{\prime}\leq l$, that is
$$
 l^{\prime}=\begin{cases}
             l&:\text{\rm if $r=2l$},\\
             l-1&:\text{\rm if $r=2l-1$}.
            \end{cases}
$$
Let $\delta$ be an irreducible complex representation of
$G(O_F/\frak{p}^r)$. The Clifford's theorem says that the restriction 
$\delta|_{K_l(O_F/\frak{p}^r)}$ is a sum of the 
$G(O_F/\frak{p}^r)$-conjugates of
characters of $K_l(O_F/\frak{p}^r)$:
\begin{equation}
 \delta|_{K_l(O_F/\frak{p}^r)}
 =\left(\bigoplus_{\dot\beta\in\Omega}\chi_{\beta}\right)^m
\label{eq:decomposition-formula-of-delta}
\end{equation}
with an adjoint 
$G(O_F/\frak{p}^{l^{\prime}})$-orbit 
$\Omega\subset\frak{g}(O_F/\frak{p}^{l^{\prime}})$. In this way the
irreducible complex representations of $G(O_F/\frak{p}^r)$ correspond
to adjoint $G(O_F/\frak{p}^{l^{\prime}})$-orbits in 
$\frak{g}(O_F/\frak{p}^{l^{\prime}})$. 

Fix an adjoint $G(O_F/\frak{p}^{l^{\prime}})$-orbit 
$\Omega\subset\frak{g}(O_F/\frak{p}^{l^{\prime}})$ and let us denote
by $\Omega\sphat$ the set of the equivalence classes of the
irreducible complex representations of $G(O_F/\frak{p}^{l^{\prime}})$
correspond to $\Omega$. Then \cite{Takase2021} gives a parametrization
of $\Omega\sphat$ as follows:

\begin{thm}\label{th:parametrization-of-omega-sphat-in-general}
Take a representative $\beta\npmod{\frak{p}^{l^{\prime}}}\in\Omega$
($\beta\in\frak{g}O_F)$) and assume that
\begin{enumerate}
\item the centralizer $G_{\beta}=Z_G(\beta)$ of
      $\beta\in\frak{g}(O_F)$ in $G$ is smooth over $O_F$,
\item the characteristic polynomial 
      $\chi_{\overline\beta}(t)=\det(t\cdot 1_{2n}-\overline\beta)$ of 
      $\overline\beta=\beta\pmod{\frak p}\in\frak{g}(\Bbb F)
        \subset\frak{gl}_{2n}(\Bbb F)$ is
      the minimal polynomial of $\overline\beta\in M_{2n}(\Bbb F)$.
\end{enumerate}
Then there exists a bijection $\theta\mapsto\delta_{\beta,\theta}$ of
the set 
$$
 \left\{\theta\in G_{\beta}(O_F/\frak{p}^r)\sphat\;\;\;
         \text{\rm s.t. $\theta=\chi_{\beta}$ on 
               $G_{\beta}(O_F/\frak{p}^r)\cap K_l(O_F/\frak{p}^r)$}
        \right\}
$$
onto $\Omega\sphat$.
\end{thm}

The correspondence $\theta\mapsto\delta_{\beta,\theta}$ is given by
the following procedure. 
The second condition in the theorem implies 
$$
 G_{\beta}(O_F/\frak{p}^r)
 =G(O_F/\frak{p}^r)\cap
  \left(O_F/\frak{p}^r\right)[\beta\npmod{\frak{p}^r}],
$$
in particular $G_{\beta}(O_F/\frak{p}^r)$ is commutative. So 
$G_{\beta}(O_F/\frak{p}^r)\sphat$ means the character group of 
$G_{\beta}(O_F/\frak{p}^r)$. 

$\Omega\sphat$ consists of the irreducible complex representations
whose restriction to $K_l(O_F/\frak{p}^r)$ contains the character
$\chi_{\beta}$. Then the Clifford's theory says the followings: put
\begin{align*}
 G(O_F/\frak{p}^r;\beta)
 &=\left\{g\in G(O_F/\frak{p}^r)\mid
           \chi_{\beta}(g^{-1}hg)=\chi_{\beta}(h)\;
             \forall h\in K_l(O_F/\frak{p}^r)\right\}\\
 &=\left\{g\in G(O_F/\frak{p}^r)\mid
           \text{\rm Ad}(g)\beta\equiv\beta
               \npmod{\frak{p}^{l^{\prime}}}\right\}
\end{align*}
and let us denote by 
$\text{\rm Irr}(G(O_F/\frak{p}^r;\beta),\chi_{\beta})$ the set of the
equivalence classes of the irreducible
 complex representations $\sigma$ of $G(O_F/\frak{p}^r;\beta)$ such
 that the restriction $\sigma|_{K_l(O_F/\frak{p}^r)}$ contains the
 character $\chi_{\beta}$. 
Then 
$\sigma\mapsto
 \text{\rm Ind}_{G(O_F/\frak{p}^r;\beta)}^{G(O_F/\frak{p}^r)}\sigma$
 gives a bijection of 
$\text{\rm Irr}(G(O_F/\frak{p}^r;\beta),\chi_{\beta})$ onto 
$\Omega\sphat$. 

Since $G_{\beta}$ is smooth over $O_F$, the canonical homomorphism 
$G_{\beta}(O_F/\frak{p}^r)\to G_{\beta}(O_F/\frak{p}^{l^{\prime}})$ is
 surjective. Hence we have
$$
 G(O_F/\frak{p}^r;\beta)
 =G_{\beta}(O_F/\frak{p}^r)\cdot K_{l^{\prime}}(O_F/\frak{p}^r).
$$
If $r=2l$ is even, then $l^{\prime}=l$ and, for any character 
$\theta\in G_{\beta}(O_F/\frak{p}^r)$ such that 
$\theta=\chi_{\beta}$ on 
$G_{\beta}(O_F/\frak{p}^r)\cap K_l(O_F/\frak{p}^r)$, the character 
$$
 \sigma_{\theta,\beta}(gh)=\theta(g)\cdot\chi_{\beta}(h)
 \quad
 (g\in G_{\beta}(O_F/\frak{p}^r), h\in K_l(O_F/\frak{p}^r))
$$
of $G(O_F/\frak{p}^r;\beta)$ is well-defined, and 
$\theta\mapsto\sigma_{\theta,\beta}$ is a surjection onto 
$\text{\rm Irr}(G(O_F/\frak{p}^r;\beta),\chi_{\beta})$. Hence
$$
 \theta\mapsto
 \delta_{\theta,\beta}
 =\text{\rm Ind}_{G(O_F/\frak{p}^r;\beta)}^{G(O_F/\frak{p}^r)}
   \sigma_{\theta,\beta}
$$
is the bijection of Theorem
\ref{th:parametrization-of-omega-sphat-in-general}.

If $r=2l-1$ is odd, then $l^{\prime}=l-1$. Let us denote by 
$\frak{g}_{\beta}=\text{\rm Lie}(G_{\beta})$ the Lie algebra
$O_F$-scheme of the smooth $O_F$-group scheme $G_{\beta}$. Then 
$$
 \Bbb V_{\beta}=\frak{g}(\Bbb F)/\frak{g}_{\beta}(\Bbb F)
$$
is a symplectic $\Bbb F$-space with a symplectic $\Bbb F$-form 
$$
 D_{\beta}(\dot X,\dot Y)=B([X,Y],\overline{\beta})\in\Bbb F
 \quad
 (X,Y\in\frak{g}(\Bbb F)).
$$
Let $H_{\beta}=\Bbb V_{\beta}\times\Bbb C^1$ be the Heisenberg group
associated with $(\Bbb V_{\beta},D_{\beta})$ and 
$(\sigma^{\beta},L^2(\Bbb W^{\prime}))$ the Schr\"odinger
representation of $H_{\beta}$ associated with a polarization 
$\Bbb V_{\beta}=\Bbb W^{\prime}\oplus\Bbb W$. More explicitly the
group operation of $H_{\beta}$ is defined by
$$
 (u,s)\cdot(v,t)=(u+v,st\cdot\widehat{\psi}(2^{-1}D_{\beta}u,v))
$$
where $\widehat{\psi}(\overline x)=\psi(\varpi^{-1}x)$ for 
$\overline x=x\npmod{\frak p}\in\Bbb F$, and the action of 
$h=(u,s)\in H_{\beta}$ on $f\in L^2(\Bbb W^{\prime})$ (a complex-valued
function on $\Bbb W^{\prime}$) is defined by
$$
 (\sigma^{\beta}(h)f)(w)
 =s\cdot\widehat{\chi}
         \left(2^{-1}D_{\beta}(u_-,u_+)+D_{\beta}(w,u_+)\right)\cdot
   f(w+u_-)
$$
where $u=u_-+u_+\in\Bbb V_{\beta}=\Bbb W^{\prime}\oplus\Bbb W$.

Take a character $\theta:G_{\beta}(O_F/\frak{p}^r)\to\Bbb C^{\times}$
such that 
$$
 \theta=\chi_{\beta}\;\;\text{\rm on}\;\;
 G_{\beta}(O_F/\frak{p}^r)\cap K_l(O_F/\frak{p}^r).
$$ 
Then an additive
character $\rho_{\theta}:\frak{g}_{\beta}(\Bbb F)\to\Bbb C^{\times}$
is defined by
$$
 \rho_{\theta}(X\nnpmod{\frak p})
 =\chi\left(-\varpi^{-l}B(X,\beta)\right)\cdot
  \theta\left(1+\varpi^{l-1}X+2^{-1}\varpi^{2l-2}X^2
              \nnpmod{\frak{p}^r}\right)
$$
with $X\in\frak{g}_{\beta}(O_F)$. 
Fix a $\Bbb F$-vector subspace $V\subset\frak{g}(\Bbb F)$ such that 
$\frak{g}(\Bbb F)=V\oplus\frak{g}_{\beta}(\Bbb F)$. Then an
irreducible representation 
$(\sigma^{\beta,\theta},L^2(\Bbb W^{\prime}))$ of
$K_{l-1}(O_F/\frak{p}^r)$ is defined by the following proposition:

\begin{prop}
\label{prop:another-expression-of-pi-beta-psi}
Take a $g=1+\varpi^{l-1}T\npmod{\frak{p}^r}\in K_{l-1}(O_F/\frak{p}^r)$ 
with $T\in\frak{gl}_n(O_F)$. Then we have 
$T\npmod{\frak{p}^{l-1}}\in\frak{g}(O_F/\frak{p}^{l-1})$ and
$$
 \sigma^{\beta,\rho}(g)
  =\tau\left(\varpi^{-l}B(T,\beta)
            -2^{-1}\varpi^{-1}B(T^2,\beta)\right)\cdot
   \rho_{\theta}(Y)\cdot\sigma^{\beta}(v,1)
$$
where $\overline T=[v]+Y\in\frak{g}(\Bbb F)$ with 
$v\in\Bbb V_{\beta}$ and 
$Y\in\frak{g}_{\beta}(\Bbb F)$. 
\end{prop}

Then main result shown in \cite{Takase2021}, under the assumptions of
Theorem \ref{th:parametrization-of-omega-sphat-in-general},
is that there exists a
group homomorphism (not unique) 
$$
 U:G_{\beta}(O_F/\frak{p}^r)\to  GL_{\Bbb C}(L^2(\Bbb W^{\prime}))
$$
such that 
\begin{enumerate}
\item $\sigma^{\beta,\theta}(h^{-1}gh)
       =U(h)^{-1}\circ\sigma^{\beta,\theta}(g)
                              \circ U(h)$ for all 
      $h\in G_{\beta}(O_F/\frak{p}^r)$ and $g\in K_{l-1}(O_F/\frak{p}^r)$,
      and 
\item $U(h)=1$ for all 
      $h\in G_{\beta}(O_F/\frak{p}^r)\cap K_{l-1}(O_F/\frak{p}^r)$.
\end{enumerate}

Now an irreducible representation 
$(\sigma_{\beta,\theta},L^2(\Bbb W^{\prime}))$ is defined by
$$
 \sigma_{\beta,\theta}(hg)
 =\theta(h)\cdot U(h)\circ\sigma^{\beta,\theta}(g)
$$
for 
$hg\in G(O_F/\frak{p}^r;\beta)
 =G_{\beta}(O_F/\frak{p}^r)\cdot K_{l-1}(O_F/\frak{p}^r)$ with 
$h\in G_{\beta}(O_F/\frak{p}^r)$ and $g\in K_{l-1}(O_F/\frak{p}^r)$, and 
$\theta\mapsto\sigma_{\beta,\theta}$ is a surjection onto 
$\text{\rm Irr}(G(O_F/\frak{p}^r;\beta),\chi_{\beta})$. Then 
$$
 \theta\mapsto
 \delta_{\beta,\theta}
 =\text{\rm Ind}_{G(O_F/\frak{p}^r;\beta)}^{G(O_F/\frak{p}^r)}
   \sigma_{\beta,\theta}
$$
is the bijection of Theorem
\ref{th:parametrization-of-omega-sphat-in-general}.

Because the connected $O_F$-group scheme $G=Sp_{2n}$ is reductive, 
that is, the fibers $G{\otimes}_{O_F}K$ ($K=F,\Bbb F$) are reductive
$K$-algebraic groups, the dimension of a maximal torus
in $G{\otimes}_{O_F}K$ is independent of $K$ which is denoted by 
$\text{\rm rank}(G)$. For any $\beta\in\frak{g}(O_F)$ we have
\begin{equation}
 \dim_K\frak{g}_{\beta}(K)=\dim\frak{g}_{\beta}{\otimes}_{O_F}K
 \geq\dim G_{\beta}{\otimes}_{O_F}K\geq\text{\rm rank}(G).
\label{eq:dimension-of-centrlizer-on-lie-alg-and-group}
\end{equation}
We say $\beta$ is {\it smoothly regular} over $K$ 
if $\dim_K\frak{g}_{\beta}(K)=\text{\rm rank}(G)$ 
(see \cite[(5.7)]{Springer1966}). 
In this case $G_{\beta}{\otimes}_{O_F}K$ is smooth over $K$. 

Let $G_{\beta}^o$ be the neutral component of $O_F$-group scheme
$G_{\beta}$ which is a group functor of the category of $O_F$-scheme 
(see $\S 3$ of Expos\'e $\text{\rm VI}_B$ in \cite{SGA-3}).
The following statements are equivalent;
\begin{enumerate}
\item $G_{\beta}^o$ is representable as an smooth open $O_F$-group subscheme of
      $G_{\beta}$, 
\item $G_{\beta}$ is smooth at the points of unit section,
\item each fibers $G_{\beta}{\otimes}_{O_F}K$ ($K=F,\Bbb F$) are smooth
      over $K$ and their dimensions are constant
\end{enumerate}
(see Th. 3.10 and Cor. 4.4 of \cite{SGA-3}). So if $\beta$ is
smoothly regular over $F$ and $\Bbb F$, then $G_{\beta}^o$ is smooth
open $O_F$-group subscheme of $G_{\beta}$. So we have 

\begin{prop}
\label{tprop:sufficient-condition-for-smooth-commutativeness-of-g-beta}
The centralizer $G_{\beta}=Z_G(\beta)$ of $\beta$ in $G$ is smooth
over $O_F$ if the following two conditions are fulfilled:
\begin{enumerate}
\item $\beta\in\frak{g}(O_F)$ is smoothly regular over $F$ and $\Bbb F$,
  and 
\item $G_{\beta}{\otimes}_{O_F}F$ and $G_{\beta}{\otimes}_{O_F}\Bbb F$ are
  connected. 
\end{enumerate}
\end{prop}

Let us assume the two conditions of the preceding proposition. 
Since we have canonical isomorphisms 
$$
 \frak{g}(\Bbb F)\,\tilde{\to}\,K_{m-1}(O_F/\frak{p}^m),
 \qquad
 \frak{g}_{\beta}(\Bbb F)\,\tilde{\to}\,
  G_{\beta}(O_F/\frak{p}^m)\cap K_{m-1}(O_F/\frak{p}^m)
$$
and the canonical morphism 
$G_{\beta}(O_F)\to G_{\beta}(O_F/\frak{p}^m)$ is surjective 
for any $m>1$, we have
$$
 |G(O_F/\frak{p}^m)|=|G(\Bbb F)|\cdot q^{(m-1)\dim G},
 \qquad
 |G_{\beta}(O_F/\frak{p}^m)|
 =|G_{\beta}(\Bbb F)|\cdot q^{(m-1)\text{\rm rank}\,G}
$$
for all $m>0$. Then we have
\begin{align*}
 \sharp\Omega\sphat=
 &\sharp\left\{\theta\in G_{\beta}(O_F/\frak{p}^r)\sphat\;\;\;
          \text{\rm s.t. $\theta=\psi_{\beta}$ on 
                $G_{\beta}(O_F/\frak{p}^r)
                  \cap K_l(O_F/\frak{p}^r)$}\right\}\\
 =&\left(G_{\beta}(O_F/\frak{p}^r):G_{\beta}(O_F/\frak{p}^r)
          \cap K_l(O_F/\frak{p}^r)\right)
 =|G_{\beta}(O_F/\frak{p}^l)|\\
 =&|G_{\beta}(\Bbb F)|\cdot q^{(l-1)\text{\rm rank}\,G}
 =\frac{|G(\Bbb F)|}
       {\sharp\overline\Omega}\cdot q^{(l-1)\text{\rm rank}\,G}
\end{align*}
where $\overline\Omega\subset\frak{g}(\Bbb F)$ is the image of 
$\Omega\subset\frak{g}(O_F/\frak{p}^{l^{\prime}})$ under the canonical
morphism $\frak{g}(O_F/\frak{p}^{l^{\prime}})\to\frak{g}(\Bbb F)$. 
On the other hand we have
$$
 \dim\sigma_{\beta,\theta}
 =\begin{cases}
   1&:\text{\rm $r$ is even},\\
   q^{\frac 12\dim_{\Bbb F}(\frak{g}(\Bbb F)/\frak{g}_{\beta}(\Bbb F))}
   =q^{(\dim G-\text{\rm rank}\,G)/2}
     &:\text{\rm $r$ is odd},
  \end{cases}
$$
so we have
\begin{align}
 \dim\delta_{\beta,\theta}
 &=\left(G(O_F/\frak{p}^r):G(O_F/\frak{p}^r;\beta)\right)
    \cdot\dim\sigma_{\beta,\theta} \nonumber\\
 &=\sharp\overline\Omega\cdot q^{(r-2)(\dim G-\text{\rm rank}\,G)/2}.
\label{eq:dimension-formula-of-delta-beta-theta}
\end{align}

In our case of $G=Sp_{2n}$, the
following two statements are equivalent for a $\beta\in\frak{g}(O_F)$: 
\begin{enumerate}
\item $\overline\beta\in\frak{g}(K)$ is smoothly regular over $K$,
\item the characteristic polynomial of 
      $\overline\beta\in\frak{g}(K)\subset\frak{gl}_{2n}(K)$ is equal
  to its minimal polynomial
\end{enumerate}
where $\overline\beta\in\frak{g}(K)$ is the image of
$\beta\in\frak{g}(O_F)$ by the canonical morphism
$\frak{g}(O_F)\to\frak{g}(K)$ with $K=F$ or $\Bbb F$. If further 
$\overline\beta\in\frak{g}(K)\subset\frak{gl}_{2n}(K)$ is nonsingular,
then $G_{\beta}{\otimes}_{O_F}K$ is connected. 

Now let $\Omega\subset\frak{g}(O_F/\frak{p}^{l^{\prime}})$ be a 
$G(O_F/\frak{p}^{l^{\prime}})$-adjoint orbit of 
$\beta\npmod{\frak{p}^{l^{\prime}}}
 \in\frak{g}(O_F/\frak{p}^{l^{\prime}})$ with $\beta\in\frak{g}(O_F)$ such
 that 
$\beta\npmod{\frak p}\in\frak{g}(\Bbb F)
 \subset\frak{gl}_{2n}(\Bbb F)$ is nonsingular and smoothly regular over 
$\Bbb F$. Then Theorem
 \ref{th:parametrization-of-omega-sphat-in-general} gives a
 parametrization of $\Omega\sphat$ by a subset of the character group 
$G_{\beta}(O_F/\frak{p}^r)$.

\begin{rem}
\label{remark:all-we-need-is-the-surejectivity-of-canonical-hom}
The assumption in Theorem
\ref{th:parametrization-of-omega-sphat-in-general} that the
centralizer $G_{\beta}$ to be smooth $O_F$-group scheme can be
replaced by the surjectivity of the canonical morphisms
$$
 G_{\beta}(O_F)\to G_{\beta}(O_F/\frak{p}^l),
 \quad
 \frak{g}_{\beta}(O_F)\to\frak{g}_{\beta}(O_F/\frak{p}^l),
$$
for all $l>0$. 
\end{rem}

\subsection{Symplectic spaces associated with tamely ramified
            extensions}
\label{subsec:symplectic-space-associated-with-tamely-ramified-ext}
Let $K_+/F$ be a tamely ramified field extension of degree $n>1$
and $K/K_+$ a quadratic field extension with 
$\text{\rm Gal}(K/K_+)=\langle\tau\rangle$. Let
$$
 e=e(K/F),
 \qquad
 f=f(K/F)
$$
be the ramification index and the inertial degree of $K/F$
respectively. Similarly put
$$
 e_+=e(K_+/F),
 \qquad
 f_+=f(K_+/F).
$$
Then we have $ef=2n$ and $e_+f_+=n$. There exists a $\omega\in O_K$
such that $\omega^{\tau}=-\omega$ and 
$O_K=O_{K_+}\oplus\omega\cdot O_{K_+}$. Then we have
$$
 \text{\rm ord}_K(\omega)=e(K/K_+)-1.
$$
Let $K_0/F$ be the maximal unramified subextension of $K/F$. Then 
$K_0/F$ is a cyclic Galois extension whose Galois group is generated
by the geometric Frobenius automorphism $\text{\rm Fr}$ which induces
the inverse of the Frobenius automorphism $[x\mapsto x^q]$ of the
residue field $\Bbb K_0$ over $\Bbb F$. Since 
$K/K_0$ is totally ramified, there exists a prime element 
$\varpi_K$ of $K$ such that $\varpi_K^e\in K_0$. Then 
$\{1,\varpi_K,\varpi_K^2,\cdots,\varpi_K^{e-1}\}$is an $O_{K_0}$-basis
of $O_K$. The following two propositions are proved by 
Shintani \cite[Lemma 4-7, Cor.1, Cor.2,pp.545-546]{Shintani1968}:

\begin{prop}\label{prop:shintani-lemma-on-generator-of-ok-over-of}
Put 
$\beta=\sum_{i=0}^{e-1}a_i\varpi_K^i\in O_K$ ($a_i\in O_{K_0}$). Then
$O_K=O_F[\beta]$ if and only if the following two conditions are
satisfied:
\begin{enumerate}
\item $a_0^{\text{\rm Fr}}\not\equiv a_0\npmod{\frak{p}_{K_0}}$ if
  $f>1$, 
\item $a_1\in O_{K_0}^{\times}$ if $e>1$.
\end{enumerate}
\end{prop}

\begin{prop}\label{prop:shintani-cor-of-irreducibility-chrara-poly}
Let $\chi_{\beta}(t)\in O_F[t]$ be the characteristic polynomial of
$\beta\in O_K\subset M_n(O_F)$ via the regular representation with
respect to an $O_F$-basis of $O_K$. If $O_K=O_F[\beta]$, then 
\begin{enumerate}
\item $\chi_{\beta}(t)\npmod{\frak{p}_F}\in\Bbb F[t]$ is the minimal
  polynomial of $\overline\beta\in M_n(\Bbb F)$, 
\item $\chi_{\beta}(t)\npmod{\frak{p}_F}=p(t)^e$ with an irreducible
  polynomial $p(t)\in\Bbb F[t]$,
\item if $e>1$, then $\chi_{\beta}(t)\npmod{\frak{p}_F^2}$ is
  irreducible over $O_F/\frak{p}_F^2$.
\end{enumerate}
\end{prop}

We can prove the following

\begin{prop}\label{prop:regular-element-of-gl(n)}
Take a $\beta\in M_n(O_F)$ whose the characteristic polynomial be
$$
 \chi_{\beta}(t)=t^n-a_nt^{n-1}-\cdots-a_2t-a_1.
$$
If $\chi_{\beta}(t)\npmod{\frak{p}_F}\in\Bbb F[t]$ is the minimal
polynomial of $\beta\npmod{\frak{p}_F}\in M_n(\Bbb F)$, then
\begin{enumerate}
\item $\{X\in M_n(O_F)\mid[X,\beta]=0\}=O_F[\beta]$,
\item for any $m>0$, put 
$\overline\beta=\npmod{\frak{p}_F^m}\in M_n(O_F/\frak{p}_F^m)$, then 
$$
 \left\{X\in M_n(O_F/\frak{p}_F^m)\mid
               [X,\overline\beta]=0\right\}
 =O_F/\frak{p}_F^m[\overline\beta],
$$
\item there exists a $in GL_n(O_F)$ such that
$$
 g\beta g^{-1}=\begin{bmatrix}
                0&0&\cdots&0&a_1\\
                1&0&\cdots&0&a_2\\
                 &1&\ddots&\vdots&\vdots\\
                 & &\ddots&0&a_{n-1}\\
                 & &      &1&a_n
               \end{bmatrix}.
$$
\end{enumerate}
\end{prop}

Then we have

\begin{prop}
\label{prop:existence-of-tau-symplectic-generator-of-ok-over-of}
There exists a $\beta\in O_K$ such that $O_K=O_F[\beta]$ and 
$\beta+\beta^{\tau}=0$ 
if and only if $K/K_+$ is unramified or $K/F$ is
totally ramified.
\end{prop}
\begin{proof}
Assume that there exists a $\beta\in O_K$ such that
$O_K=O_F[\beta]$ and $\beta+\beta^{\tau}=0$. Then $K=K_+(\beta^2)$. 
If $K/F$ is not totally ramified, we have $\beta\in O_K^{\times}$ by
Proposition \ref{prop:shintani-lemma-on-generator-of-ok-over-of}, and
hence $K/K_+$ is an unramified extension. 

Assume that $K/F$ is totally ramified. Then $K_0=F$ and 
$\varpi_K^e\in O_F$. Since the quadratic extension $K/K_+$ is
ramified, there exists a prime element $\beta$ of $K$ such that 
$\beta^2\in K_+$. Then $\beta=\varepsilon\cdot\varpi_K$ with
$\varepsilon\in O_K^{\times}$. 
Put $\varepsilon=\sum_{i=0}^{e-1}a_i\varpi_K^i$ with 
$a_i\in O_F$. Then 
$$
 \beta=a_{e-1}\varpi_K^e+\sum_{i=1}^{e-1}a_{i-1}\varpi_K^i
$$
with $a_0\in O_F^{\times}$. Now we have 
$\beta^{\tau}=-\beta$ and $O_K=O_F[\beta]$ by 
Proposition \ref{prop:shintani-lemma-on-generator-of-ok-over-of}. 

Assume that $K/K_+$ is unramified. Let $K_{+0}/F$ be the maximal
unramified subextension of $K_+/F$. Since $(K_{+0}:F)=f_+$ divides 
$(K_0:F)=f$, we have $K_{+0}\subset K_0$. We can chose $\varpi_K$ in
$K_+$ so that $\varpi_K^e\in K_{+0}$. For the residue fields, we have
$$
 (\Bbb K_0:\Bbb K_{+0})
 =\frac{(\Bbb K_0:\Bbb F)}
       {(\Bbb K_{+0}:\Bbb F)}=\frac f{f_+}=2.
$$
Put $\Bbb K_{+0}=\Bbb F[\overline\alpha]$ with 
$\alpha\in O_{K_{+0}}^{\times}$ such that 
$\overline\alpha\not\in\left(\Bbb K_{+0}^{\times}\right)^2$. Since 
$\Bbb K_0$ is the splitting filed of 
$f(X)=X^2-\overline\alpha\in\Bbb K_{+0}[X]$, there exists 
$\gamma\in O_{K_0}^{\times}$ such that 
$$
 f(\gamma)\equiv 0\npmod{\frak{p}_{K_0}},
 \quad
 f^{\prime}(\gamma)\not\equiv 0\npmod{\frak{p}_{K_0}}.
$$
Hence there exists $a\in O_{K_0}^{\times}$ such that $f(a)=0$ and 
$a\equiv\gamma\npmod{\frak{p}_{K_0}}$. Since
$$
 \Bbb K_0=\Bbb F[\overline\gamma]=\Bbb F[\overline a],
$$
we have $a^{\text{\rm Fr}}\not\equiv a\npmod{\frak{p}_{K_0}}$ and 
$a^{\tau}=-a$. Put $\beta=a(1+\varpi_K)\in O_K$, then 
$O_K=O_F[\beta]$ by Proposition
  \ref{prop:shintani-lemma-on-generator-of-ok-over-of} and 
$\beta^{\tau}=-\beta$.
\end{proof}

From now on let us assume that $K/K_+$ is unramified or $K/F$ is
totally ramified, and take a $\beta\in O_K$ such that 
$O_K=O_F[\beta]$ and $\beta^{\tau}+\beta=0$. Fix a prime element
$\varpi_{K_+}$ of $K_+$. Then a symplectic form on
$F$-vector space $K$ is defined by
$$
 D(x,y)
 =\frac 12T_{K/F}\left(\omega^{-1}\varpi_{K_+}^{1-e_+}x^{\tau}y\right)
 \quad
 (x,y\in K).
$$
For any $a\in K$, we have $D(xa,y)=D(x,y\cdot a^{\tau})$ for all 
$x,y\in K$. Inparticular
$$
 \beta\in\frak{sp}(K,D)
 =\{X\in\text{\rm End}_F(K)\mid D(xX,y)+D(x,yX)=0\;
                   \forall x,y\in K\}
$$
if we put $K\subset\text{\rm End}_F(K)$ by the regular
representation. 

Let $\{u_i\}_{1\leq i\leq n}$ be an $O_F$-basis of
$O_{K_+}$. Then 
$x\mapsto\left(T_{K_+/F}(u_1x),\cdots,T_{K_+/F}(u_nx)\right)$ gives an
isomorphism $\frak{p}_{K_+}^{1-e_+}\,\tilde{\to}\,O_F^n$ of 
$O_F$-module. Hence there exists an $O_F$-basis 
$\{u^{\ast}_i\}_{1\leq i\leq n}$ of $\frak{p}_{K_+}^{1-e_+}$ such that 
$T_{K_+/F}(u_iu^{\ast}_j)=\delta_{ij}$. 
Put $v_i=\omega\cdot\varpi_{K_+}^{e_+-1}\cdot u^{\ast}_i$ 
($1\leq i\leq n$). Then $\{u_1,\cdots,u_n,v_n,\cdots,v_1\}$ is a
$O_F$-basis of $O_K$ and a symplectic $F$-basis of $K$, that is
$$
 D(u_i,u_j)=D(v_i,v_j)=0,
 \quad
 D(u_i,v_j)=\delta_{ij}
 \quad
 (1\leq i,j\leq n).
$$
This means that our $O_F$-group scheme $G=Sp_{2n}$ is defined by the
symplectic $F$-space $(K,D)$ and the symplectic basis 
$\{u_i,v_j\}_{1\leq i,j\leq n}$. 

By Proposition \ref{prop:shintani-cor-of-irreducibility-chrara-poly},
the characteristic polynomial of 
$\overline\beta=\beta\npmod{\frak{p}_F}\in M_{2n}(\Bbb F)$ is equal to
its minimal polynomial. Then, by Proposition 
\ref{prop:regular-element-of-gl(n)}, we have
$$
 \{X\in M_{2n}(O_F)\mid[X,\beta]=0\}=O_F[\beta]=O_K
$$
and
$$
 \{X\in M_{2n}(O_F/\frak{p}_F^l)\mid[X,\overline\beta]=0\}
 =O_F/\frak{p}_F^l[\overline\beta]
 =O_K/\frak{p}_K^{el}
$$
for any $m>0$. Put 
$$
 U_{K/K_+}
 =\{\varepsilon\in O_K^{\times}\mid N_{K/K_+}(\varepsilon)=1\}.
$$
Then we have
$$
 G_{\beta}(O_F)=G(O_F)\cap O_K=U_{K/K_+}.
$$
We have also
$$
 \frak{g}_{\beta}(O_F)=\frak{g}(O_F)\cap O_K
 =\{X\in O_K\mid T_{K/K_+}(X)=0\}
$$
and
\begin{align*}
 G_{\beta}(O_F/\frak{p}_F^l)
 &=\{\overline\varepsilon\in\left(O_K/\frak{p}_K^{el}\right)^{\times}
    \mid N_{K/K_+}(\varepsilon)\equiv
    1\npmod{\frak{p}_{K_+}^{e_+l}}\},\\
 \frak{g}_{\beta}(O_F/\frak{p}_F^l)
 &=\{\overline X\in O_K/\frak{p}_K^{el}\mid
     T_{K/K_+}(X)\equiv 0\npmod{\frak{p}_{K_+}^{e_+l}}\}
\end{align*}
for all $l>0$. Then the canonical morphisms
$$
 G_{\beta}(O_F)\to G_{\beta}(O_F/\frak{p}_F^l),
 \qquad
 \frak{g}_{\beta}(O_F)\to\frak{g}_{\beta}(O_F/\frak{p}_F^l)
$$
are surjective for all $l>0$. In fact, 
Take a $\varepsilon\in O_K^{\times}$ such that 
$N_{K/K_+}(\varepsilon)\equiv 1\npmod{\frak{p}_{K_+}^{e_+l}}$. Because
$K/K_+$ is tamely ramified, we have 
$N_{K/K_+}(1+\frak{p}_K^{el})=1+\frak{p}_{K_+}^{e_+l}$. Hence there 
exists a $\eta\in 1+\frak{p}_K^{el}$ such that
$N_{K/K_+}(\eta)=\varepsilon$. 
Then $\alpha=\varepsilon\eta^{-1}\in O_K^{\times}$ such that 
$N_{K/K_+}(\alpha)=1$ and 
$\alpha=\varepsilon\npmod{\frak{p}_K^{el}}$. Take a 
$X\in O_K$ such that 
$T_{K/K_+}(X)\equiv 0\npmod{\frak{p}_{K_+}^{e_+l}}$. If we put 
$X=s+\omega t$ with $s,t\in O_{K_+}$, then 
$s\in\frak{p}_{K_+}^{e_+l}\subset\frak{p}_K^{el}$. Hence we have 
$\omega t\in\frak{g}_{\beta}(O_F)$ and 
$\omega t\equiv X\npmod{\frak{p}_K^{el}}$.

Due to Remark
\ref{remark:all-we-need-is-the-surejectivity-of-canonical-hom}, 
we can apply the general theory of subsection 
\ref{subsec:regular-irred-character-of-hyperspecial-compact-subgroup}
to our $\beta\in\frak{g}(O_F)$. Take an integer $r>1$ and put 
$r=l+l^{\prime}$ with minimal integer $l$ such that 
$0<l^{\prime}\leq l$. 
Let $\Omega\subset\frak{g}(O_F/\frak{p}_F^{l^{\prime}})$ be the 
adjoint $G(O_F/\frak{p}_F^{l^{\prime}})$-orbit of 
$\beta\npmod{\frak{p}_F^{l^{\prime}}}
 \in\frak{g}(O_F/\frak{p}_F^{l^{\prime}})$, and $\Omega\sphat$ the set
 of the equivalent classes of the irreducible representations of 
$G(O_F/\frak{p}_F^r)$ corresponding to $\Omega$ via Clifford's theory
 described in subsection
 \ref{subsec:regular-irred-character-of-hyperspecial-compact-subgroup}. 
Then we have a bijection $\theta\mapsto\delta_{\beta,\theta}$ 
of the continuous unitary character $\theta$ of $U_{K/K_+}$ such that
\begin{enumerate}
\item $\theta$ factors through the canonical morphism
         $U_{K/K_+}\to\left(O_K/\frak{p}_K^{er}\right)^{\times}$, 
\item for an $\alpha\in U_{K/K_+}$ such that 
      $\alpha\equiv 1+\varpi_F^lx\npmod{\frak{p}_K^{er}}$ with 
      $x\in O_K$ such that 
      $T_{K/K_+}(x)\equiv 0\npmod{\frak{p}_{K_+}^{e_+l^{\prime}}}$, we
  have 
$\theta(\alpha)
 =\psi\left(\varpi_F^{-l^{\prime}}T_{K/F}(x\beta)\right)$.
\end{enumerate}
onto $\Omega\sphat$. Here $\psi:F\to\Bbb C^{\times}$ is a continuous
unitary character of the additive group $F$ such that 
$\{x\in F\mid\psi(xO_F)=1\}=O_F$. Then we have

\begin{prop}\label{prop:dimension-of-delta-beta-theta}
$$
 \dim\delta_{\beta,\theta}
 =q^{n^2r}\cdot\prod_{k=1}^n\left(1-q^{-2k}\right)\times
  \begin{cases}
   \frac 12&:\text{\rm $K/F$ is totally ramified},\\
   \frac 1{1+q^{-f_+}}&:\text{\rm $K/K_+$ is unramified}.
  \end{cases}
$$
\end{prop}
\begin{proof}
For the dimension formula
\eqref{eq:dimension-formula-of-delta-beta-theta}, we have
$$
 \dim G=n(2n+1),
 \quad
 \text{\rm rank}\,G=n,
 \quad
 \sharp\overline\Omega=\frac{|G(\Bbb F)|}
                            {|G_{\beta}(\Bbb F)|}
$$
and
$$
 |G(\Bbb F)|=|Sp_{2n}(\Bbb F)|
 =q^{n(2n+1)}\cdot\prod_{k=1}^n\left(1-q^{-2k}\right).
$$
On the other hand $G_{\beta}(\Bbb F)$ is the kernel of
$$
 (\ast):\left(O_K/\frak{p}_K^e\right)^{\times}\to
        \left(O_{K_+}/\frak{p}_{K_+}^{e_+}\right)^{\times}
 \quad
 \left(\overline\varepsilon\mapsto
       \overline{N_{K/K_+}(\varepsilon)}\right).
$$
Since $K/K_+$ is tamely ramified quadratic extension, we have
$$
 1+\frak{p}_{K_+}^{e_+}=N_{K/K_+}(1+\frak{p}_K^e)
 \subset N_{K/K_+}(O_K^{\times})\subset O_{K_+}^{\times},
$$
and $(O_{K_+}^{\times}:N_{K/K_+}(O_K^{\times}))=e/e_+$, hence
\begin{align*}
 |G_{\beta}(\Bbb F)|
 &=\frac{\left|\left(O_K/\frak{p}_K^e\right)^{\times}\right|
         (O_{K_+}^{\times}:N_{K/K_+}(O_K^{\times}))}
        {\left|\left(O_{K_+}/\frak{p}_{K_+}^{e_+}\right)^{\times}
         \right|}
  =\frac e{e_+}\cdot q^n\cdot\frac{1-q^{-f}}
                                  {1-q^{-f_+}}\\
 &=q^n\times\begin{cases}
             2&:\text{\rm $K/F$ is totally ramified},\\
             1+q^{-f_+}&:\text{\rm $K/K_+$ is unramified}.
            \end{cases}
\end{align*}
\end{proof}

\subsection{Construction of supercuspidal representations}
\label{subsec:construction-of-supercuspidal-representataion}
We will keep the notations of the preceding subsection. 
The purpose of this subsection is to prove the following theorem:

\begin{thm}\label{th:supercuspidal-representation-of-sp(2n)}
If 
$l^{\prime}=\left\lfloor\frac r2\right\rfloor
 \geq\text{\rm Max}\{2,2(e-1)\}$, 
then the compactly induced representation 
$\pi_{\beta,\theta}
 =\text{\rm ind}_{G(O_F)}^{G(F)}\delta_{\beta,\theta}$ is an
 irreducible supercuspidal representation of $G(F)=Sp_{2n}(F)$ such
 that
\begin{enumerate}
\item the multiplicity of $\delta_{\beta,\theta}$ in 
      $\pi_{\beta,\theta}|_{G(O_F)}$ is one, 
\item $\delta_{\beta,\theta}$ is the unique irreducible unitary
  constituent of $\pi_{\beta,\theta}|_{G(O_F)}$ which factors through the
  canonical morphism $G(O_F)\to G(O_F/\frak{p}^r)$,
\item with respect to the Haar measure on $G(F)$ such that the volume
  of $G(O_F)$  is one, the formal degree of $\pi_{\beta,\theta}$ is
    equal to
$$
    \dim\delta_{\beta,\theta}
    =q^{n^2r}\cdot\prod_{k=1}^n\left(1-q^{-2k}\right)\times
    \begin{cases}
      \frac 12&:\text{\rm $K/F$ is totally ramified,}\\
      \frac 1{1+q^{-f_+}}&:\text{\rm $K/K_+$ is unramified}.
    \end{cases}
$$
\end{enumerate}
\end{thm}

The rest of this subsection is devoted to the proof. 

We have the Cartan decomposition
$$
 G(F)=\bigsqcup_{m\in\Bbb M}G(O_F)t(m)G(O_F)
$$
where
$$
 \Bbb M=\{(m_1,m_2,\cdots,m_n)\in\Bbb Z^n\mid
           m_1\geq m_2\geq\cdots\geq m_n\geq 0\}
$$
and 
$$
 t(m)=\begin{bmatrix}
       \varpi_F^m&0\\
       0&^{\frak t}\varpi_F^{-m}
      \end{bmatrix}
 \;\;\text{\rm with}\;\;
 \varpi_F^m=\begin{bmatrix}
             \varpi_F^{m_1}&      &              \\
                           &\ddots&              \\
                           &      &\varpi_F^{m_n}
            \end{bmatrix}
$$
for $m=(m_1,\cdots,m_n)\in\Bbb M$. 

For an integer $1\leq i\leq n$, let $L_i$ and $U_i$ be $O_F$-group
subscheme of $G=Sp_{2n}$ defined by
\begin{align*}
 L_i&=\left\{\begin{bmatrix}
              a& &             \\
               &g&             \\
               & &^{\frak t}a^{-1}
             \end{bmatrix}\biggm| a\in GL_i, g\in Sp_{2(n-i)}
        \right\},\\
 U_i&=\left\{\begin{bmatrix}
              1_i&   A   &   B   &   C       \\
                 &1_{n-i}&   0   &^{\frak t}B   \\
                 &       &1_{n-i}&-\,^{\frak t}A\\
                 &       &       &  1_i
             \end{bmatrix}\in Sp_{2n}\right\}
\end{align*}
so that $P_i=L_i\cdot U_i$ is a maximal parabolic subgroup of
$G=Sp_{2n}$ and $U_i$ (resp. $L_i$) is the unipotent (resp. Levi) part
of $P_i$. Put $U_i(\frak{p}_F^a)=U_i(O_F)\cap K_a(O_F)$ for a positive
integer $a$.

\begin{prop}\label{prop:compact-induction-is-admissible}
If $K/F$ is unramified or $r\geq 4$, then the compactly induced
representation 
$\pi_{\beta,\theta}
 =\text{\rm ind}_{G(O_F)}^{G(F)}\delta_{\beta,\theta}$ is an
admissible representation of $G(F)$.
\end{prop}
\begin{proof}
It is  enough to show that 
$\dim_{\Bbb C}
 \text{\rm Hom}_{K_a(O_F)}(\text{\bf 1},\pi_{\beta,\theta})
 <\infty$ for all $a>0$, where $\text{\bf 1}$ is the trivial
 one-dimensional representation of $K_a(O_F)$. We have
$$
 G(F)=\bigsqcup_{s\in\Bbb S}K_a(O_F)sG(O_F)
$$
where
$$
 \Bbb S=\{k\cdot t(m)\mid \dot k\in K_a(O_F)\backslash G(O_F), 
                          m\in\Bbb M\}.
$$
Then, by the restriction formula of induced representations and by the
Frobenius reciprocity, we have
\begin{align*}
 \text{\rm Hom}_{K_a(O_F)}(\text{\bf 1},\pi_{\beta,\theta})
 &=\bigoplus_{s\in\Bbb S}
   \text{\rm Hom}_{K_a(O_F)}\left(\text{\bf 1},
    \text{\rm ind}_{K_a(O_F)\cap sF(O_F)s^{-1}}^{G(F)}
     \delta_{\beta,\theta}^s\right)\\
 &=\bigoplus_{s\in\Bbb S}
    \text{\rm Hom}_{K_a(O_F)\cap sG(O_F)s^{-1}}
     (\text{\bf 1},\delta_{\beta,\theta}^s)\\
 &=\bigoplus_{s\in\Bbb S}
    \text{\rm Hom}_{s^{-1}K_a(O_F)s\cap G(O_F)}
     (\text{\bf 1},\delta_{\beta,\theta}).
\end{align*}
So it is enough to show that the number of $s\in\Bbb S$ such that 
$\text{\rm Hom}_{s^{-1}K_a(O_F)s\cap G(O_F)}
           (\text{\bf 1},\delta_{\beta,\theta})\neq 0$ is finite. 
Take such a $s=k\cdot t\in\Bbb S$ with $k\in G(O_F)$ and 
$t=t(m)$ ($m\in\Bbb M$). Suppose
$$
 \text{\rm Max}\{m_k-m_{k+1}\mid 1\leq k<n\}
 =m_i-m_{i+1}\geq a.
$$
Then we have $tU_i(\frak{p}_F^a)t^{-1}\subset K_a(O_F)$ and hence
$$
 U_i(\frak{p}_F^l)\subset U_i(O_F)
                  \subset s^{-1}K_a(O_F)s\cap G(O_F)
$$
and
$$
 \text{\rm Hom}_{U_i(\frak{p}_F^l)}
  (\text{\bf 1},\delta_{\beta,\theta})\supset
 \text{\rm Hom}_{s^{-1}K_a(O_F)s\cap G(O_F)}
  (\text{\bf 1},\delta_{\beta,\theta})\neq 0.
$$
This means, by \eqref{eq:decomposition-formula-of-delta}, that there
exists a $g\in G(O_F)$ such that 
$\chi_{\text{\rm Ad}(g)\beta}(h)=1$ for all $h\in U_i(\frak{p}_F^l)$,
that is 
$\psi\left(\varpi_F^{-l^{\prime}}
           \text{\rm tr}\left(g\beta g^{-1}X\right)\right)=1$ 
for all $X\in\text{\rm Lie}(U_i)(O_F)$. Hence we have 
$\overline{g\beta g^{-1}}
 \in\text{\rm Lie}(P_i)(O_F/\frak{p}_F^{l^{\prime}})$. 
Then the characteristic polynomial 
$\chi_{\beta}(t)\npmod{\frak{p}_F}\in\Bbb F[t]$ is reducible. Hence
$e>1$ by Proposition
\ref{prop:shintani-cor-of-irreducibility-chrara-poly}. Then 
$l^{\prime}=\left\lfloor\frac r2\right\rfloor\geq 2$ and 
$\chi_{\beta}(t)\npmod{\frak{p}_F^2}$ is reducible over 
$O_F/\frak{p}_F^2$ contradicting to Proposition 
\ref{prop:shintani-cor-of-irreducibility-chrara-poly}. Hence we have
$$
 \text{\rm Max}\{m_i-m_{i+1}\mid 1\leq i<n\}<a.
$$
Similar arguments using the parabolic subgroup $P_n$ shows that 
$2m<a$. This shows the required finiteness of $s\in\Bbb S$. 
\end{proof}

\begin{lem}\label{lemma:intertwiner-over-unipotent-part}w
\begin{enumerate}
\item If 
      $\text{\rm Hom}_{U_i(\frak{p}_F^{r-1})}
       (\text{\bf 1},\delta_{\beta,\theta})\neq 0$ for some 
      $1\leq i\leq n$, then $i\equiv 0\npmod{f}$ and $e>1$. If further 
      $i<n$, then $e\geq 3$.
\item If $\frac r2\geq 2$, then 
      $\text{\rm Hom}_{U_i(\frak{p}_F^{r-2})}
       (\text{\bf 1},\delta_{\beta,\theta})=0$ for all 
      $1\leq i\leq n$.
\end{enumerate}
\end{lem}
\begin{proof}
Assume that 
$\text{\rm Hom}_{U_i(\frak{p}_F^k)}
  (\text{\bf 1},\delta_{\beta,\theta}\neq 0$ with some 
$0<k\leq l^{\prime}$. Then $U_i(\frak{p}_F^{r-k})\subset K_l(O_F)$ and 
\eqref{eq:decomposition-formula-of-delta} implies that there exists a
$g\in G(O_F)$ such that $\chi_{\text{\rm Ad}(G)\beta}(h)=1$ for all 
$h\in U_i(\frak{p}_F^{r-k})$, that is 
$\psi\left(\varpi_F^{-k}\text{\rm tr}\left(g\beta g^{-1}X\right)
     \right)=1$ for all $X\in\text{\rm Lie}(U_i)(O_F)$. Then
$$
 g\beta g^{-1}\equiv\begin{bmatrix}
                     A&\ast&\ast\\
                     0& X  &\ast\\
                     0& 0  &-\,^{\frak t}A
                    \end{bmatrix}
              \npmod{\frak{p}_F^k}
$$
with $A\in\frak{gl}_i(O_F)$ and $X\in\frak{sp}_{2(n-i)}(O_F)$. So the
characteristic polynomial is
$$
 \chi_{\beta}(t)\equiv
 \det(t1_i-A)\det(t1_{2(n-i)}-X)\det(1_i+A)
 \npmod{\frak{p}_F^k}.
$$
If $k=1$, then the first statement of Proposition
\ref{prop:shintani-cor-of-irreducibility-chrara-poly} implies that 
$$
 i=\deg\det(t1_i-A)\equiv 0\npmod{f}\;\;\text{\rm and}\;\;
 e>1.
$$
If $l^{\prime}\geq 2$ and $k=2$, then 
$\chi_{\beta}(t)\npmod{\frak{p}_F^2}$ is reducible over
$O_F/\frak{p}_F^2$ contradicting to the third statement of Proposition 
\ref{prop:shintani-cor-of-irreducibility-chrara-poly}. 
\end{proof}

\begin{prop}\label{prop:minimal-k-type-of-induced-representation}
Assume that 
$l^{\prime}=\left\lfloor\frac r2\right\rfloor\geq
 \text{\rm Max}\{2,2(e-1)\}$. Then
\begin{enumerate}
\item $\dim_{\Bbb C}\text{\rm Hom}_{G(O_F)}
        \left(\delta_{\beta,\theta},\pi_{\beta,\theta}\right)=1$,
\item for any irreducible representation $(\delta,V_{\delta})$
  of $G(O_F)$ which factors through the canonical morphism 
  $G(O_F)\to G(O_F/\frak{p}_F^r)$, if 
$\text{\rm Hom}_{G(O_F)}(\delta,\pi_{\beta,\theta})\neq 0$, 
then $\delta=\delta_{\beta,\theta}$.
\end{enumerate}
\end{prop}
\begin{proof}
Let $(\delta,V_{\delta})$ be an irreducible unitary representation of
$G(O_F)$ which factors through the canonical morphism 
$G(O_F)\to G(O_F/\frak{p}_F^r)$. Then we have
\begin{align*}
 \text{\rm Hom}_{G(O_F)}(\delta,\pi_{\beta,\theta})
 &=\bigoplus_{m\in\Bbb M}
    \text{\rm Hom}_{G(O_F)}\left(\delta,
     \text{\rm ind}_{G(O_F)\cap t(m)G(O_F)t(m)^{-1}}^{G(O_F)}
            \delta_{\beta,\theta}^{t(m)}\right)\\
 &=\bigoplus_{m\in\Bbb M}
    \text{\rm Hom}_{G(O_F)\cap t(m)G(O_F)t(m)^{-1}}
     \left(\delta,\delta_{\beta,\theta}^{t(m)}\right)\\
 &=\bigoplus_{m\in\Bbb M}
    \text{\rm Hom}_{t(m)^{-1}G(O_F)t(m)\cap G(O_F)}
     \left(\delta^{t(m)^{-1}},\delta_{\beta,\theta}\right).
\end{align*}
Assume that 
$\text{\rm Hom}_{t(m)^{-1}G(O_F)t(m)\cap G(O_F)}
     \left(\delta^{t(m)^{-1}},\delta_{\beta,\theta}\right)\neq 0$ for
a $m=(m_1,\cdots,m_n)\in\Bbb M$. If
$$
 \text{\rm Max}\{m_k-m_{k+1}\mid 1\leq k<n\}=m_i-m_{i+1}\geq 2
$$
then we have 
$t(m)U_i(\frak{p}_F^{r-2})t(m)^{-1}\subset U_i(\frak{p}_F^r)$. Since 
$K_r(O_F)\subset\text{\rm Ker}(\beta)$, the restriction of 
$\delta^{t(m)^{-1}}$ to $U_i(\frak{p}_F^{r-2})$ is trivial. On the
  other hand, we have 
$$
 U_i(\frak{p}_F^{r-2})\subset 
 t(m)^{-1}U_i(\frak{p}_F^r)t(m)\cap U_i(O_F)\subset
 t(m)^{-1}G(O_F)t(m)\cap G(O_F).
$$
Now we have 
$\text{\rm Hom}_{U_i(\frak{p}_F^{r-2})}
  (\text{\bf 1},\delta_{\beta,\theta})\neq 0$ contradicting to the
second statement of Lemma
\ref{lemma:intertwiner-over-unipotent-part}. Hence we have 
$$
 \text{\rm Max}\{m_k-m_{k+1}\mid 1\leq k<n\}\leq 1.
$$
Similarly we have $2m_n\leq 1$, that is $m_n=0$. If there exists 
$1\leq i<n$ such that $m_i-m_{i+1}=1$. Then, with the similar
arguments as above, we have 
$\text{\rm Hom}_{U_i(\frak{p}_F^{r-1})}
  (\text{\bf 1},\delta_{\beta,\theta})\neq 0$. The first statement of
Lemma \ref{lemma:intertwiner-over-unipotent-part} implies that 
$i\equiv 0\npmod{f}$. Since $ef=2n$, this means $m_1<\frac e2$, hence
\begin{equation}
 4m_1\leq 2(e-1)\leq l^{\prime}.
\label{eq:4m1-is-less-than-l-prime}
\end{equation}
Since 
$t(m)K_{l+2m_1}(O_F)t(m)^{-1}\subset K_l(O_F)$ and hence 
$$
 K_{l+2m_1}(O_F)\subset t(m)^{-1}G(O_F)t(m)\cap G(O_F),
$$
we have
$$
 \text{\rm Hom}_{K_{l+2m_1}}
  \left(\delta^{t(m)^{-1}},\delta_{\beta,\theta}\right)
 \supset
 \text{\rm Hom}_{t(m)^{-1}G(O_F)t(m)\cap G(O_F)}
  \left(\delta^{t(m)^{-1}},\delta_{\beta,\theta}\right)\neq 0.
$$
Assume that $\delta$ corresponds, as explained in subsection 
\ref{subsec:regular-irred-character-of-hyperspecial-compact-subgroup}, 
to an adjoint $G(O_F/\frak{p}_F^{l^{\prime}})$-orbit
$\Omega^{\prime}\subset\frak{g}(O_F/\frak{p}_F^{l^{\prime}})$ of 
$\beta^{\prime}\npmod{\frak{p}_F^{l^{\prime}}}$
($\beta^{\prime}\in\frak{g}(O_F)$). Then there exists 
$k,h\in G(O_F)$ such that
$$
 \chi_{\text{\rm Ad}(k)\beta}(x)
 =\chi_{\text{\rm Ad}(h)\beta^{\prime}}(t(m)xt(m)^{-1})
$$
for all $x\in K_{l+2m_1}(O_F)$. This means
$$
 k\beta k^{-1}\equiv 
 t(m)^{-1}h\beta^{\prime}h^{-1}t(m)
 \npmod{\frak{p}_F^{l^{\prime}-2m_1}}.
$$
Then, because of \eqref{eq:4m1-is-less-than-l-prime}, the matrix 
$t(m)k\beta k^{-1}t(m)^{-1}$ belongs to 
$$
 h\beta^{\prime}h^{-1}
  +t(m)M_{2n}(\frak{p}_F^{l^{\prime}-2m_1})t(m)^{-1}
 \subset
 h\beta^{\prime}h^{-1}+M_{2n}(\frak{p}_F^{l^{\prime}-4m_1})
 \subset 
 M_{2n}(O_F).
$$
Since the characteristic polynomials of 
$t(m)k\beta k^{-1}t(m)^{-1}$ and $\beta$ are identical, there exists,
by the third statement of Proposition
\ref{prop:regular-element-of-gl(n)}, a $g\in GL_{2n}(O_F)$ such that 
$t(m)k\beta k^{-1}t(m)^{-1}=g\beta g^{-1}$. Then 
$g^{-1}t(m)k\in F[\beta]=K$ and 
$$
 N_{K/F}(g^{-1}t(m)k)=\det(g^{-1}t(m)k)\in O_F^{\times}.
$$
Hence $g^{-1}t(m)k\in O_K\subset M_{2n}(O_F)$ and 
$t(m)\in M_{2n}(O_F)$, that is $m=(0,\cdots,0)$. So we have proved 
$$
 \text{\rm Hom}_{G(O_F)}(\delta,\pi_{\beta,\theta})
 =\text{\rm Hom}_{G(O_F)}(\delta,\delta_{\beta,\theta})
$$
which clearly implies the statements of the proposition.
\end{proof}

The admissible representation 
$\pi_{\beta,\theta}
 =\text{\rm ind}_{G(O_F)}^{G(F)}\delta_{\beta,\theta}$ of $G(F)$ is
 irreducible. In fact, if there exists a $G(F)$-subspace
$0\lvertneqq W\lvertneqq
 \text{\rm ind}_{G(O_F)}^{G(F)}\delta_{\beta,\theta}$, we have
\begin{align*}
 0\neq\text{\rm Hom}_{G(F)}
    (W,\text{\rm ind}_{G(O_F)}^{G(F)}\delta_{\beta,\theta})
  &\subset\text{\rm Hom}_{G(O_F)}
    (W,\text{\rm Ind}_{G(O_F)}^{G(F)}\delta_{\beta,\theta})\\
  &=\text{\rm Hom}_{G(O_F)}(W,\delta)
\end{align*}
by Frobenius reciprocity. Hence 
$\delta\hookrightarrow W|_{G(O_F)}$. On the other hand, we have
\begin{align*}
 0&\neq
    \text{\rm Hom}_{G(F)}
      \left(\text{\rm ind}_{G(O_F)}^{G(F)}\delta_{\beta,\theta},
      \left(\text{\rm ind}_{G(O_F)}^{G(F)}\delta_{\beta,\theta}
             \right)/W\right)\\
 &=\text{\rm Hom}_{G(O_F)}\left(\delta_{\beta,\theta},
     \left(\text{\rm ind}_{G(O_F)}^{G(F)}\delta_{\beta,\theta}
           \right)/W\right),
\end{align*}
hence 
$\delta\hookrightarrow
 \left(\text{\rm ind}_{G(O_F)}^{G(F)}\delta_{\beta,\theta}\right)/W$. 
Now $\text{\rm ind}_{G(O_F)}^{G(F)}\delta_{\beta,\theta}$ is
semi-simple $G(O_F)$-module, we have
$$
 \dim_{\Bbb C}
 \text{\rm Hom}_{G(O_F)}(\delta_{\beta,\theta},
         \text{\rm ind}_{G(O_F)}^{G(F)}\delta_{\beta,\theta})
 \geq 2
$$
which contradicts to the first statement of Proposition 
\ref{prop:minimal-k-type-of-induced-representation}. 

Now $\pi_{\beta,\theta}$ is a
supercuspidal representation of $G(F)$ whose formal degree with
respect to the Haar measure $d_{G(F)}(x)$of $G(F)$ such that 
$\int_{G(O_F)}d_{G(F)}(x)=1$ is equal to $\dim\delta_{\beta,\theta}$. 
We have completed the proof of Theorem
\ref{th:supercuspidal-representation-of-sp(2n)}.

\section{Kaleta's $L$-parameter}
\label{sec:kaleta-l-parameter}

\subsection{Local Langlands correspondence of elliptic tori}
\label{subsec:local-langlands-correspondence-of-elliptic-tori}
Let $K_+/F$ be 
a finite extension, $K/K_+$ a quadratic extension with a non-trivial
element $\tau$ of $\text{\rm Gal}(K/K_+)$. Let us denote by $L$
an arbitrary Galois extension over $F$ containing $K$ for which let us 
denote by  
$$
 \text{\rm Emb}_F(K,L)
 =\left\{\sigma|_K\mid\sigma\in\text{\rm Gal}(L/F)\right\}
$$
the set of the embeddings over $F$ of $K$ into $L$.

Put $O_K=O_{K_+}\oplus\omega O_{K_+}$ with
$\omega^{\tau}+\omega=0$. Then $\text{\rm
  ord}_K(\omega)=e(K/K_+)-1$. Let us denote by $\Bbb V$ the
$\overline F$-algebra of the functions $v$ on 
$\text{\rm Emb}_F(K,\overline F)$ with values in $\overline F$ which
is endowed with a symplectic $\overline F$-form
$$
 D(u,v)=\frac 12\sum_{\gamma\in\text{\rm Emb}_F(K,\overline F)}
         \left(\omega^{-1}\varpi_{K_+}^{-d_+}\right)^{\gamma}
          u(\tau\gamma)\cdot v(\gamma)
$$
($u, v\in\Bbb V$) where $\mathcal{D}(K_+/F)=\frak{p}_{K_+}^{d_+}$ is
          the difference of $K_+/F$. The action of 
$\sigma\in\text{\rm Gal}(\overline F/F)$ on 
$v\in\Bbb V$ is defined by 
$v^{\sigma}(\gamma)=v(\gamma\sigma^{-1})^{\sigma}$. Then fixed point
subspace ${\Bbb V}^{\text{\rm Gal}(\overline F/L)}=\Bbb V(L)$ 
is the set of the
functions on $\text{\rm Emb}_F(K,L)$ with values in $L$, and 
${\Bbb V}^{\text{\rm Gal}(\overline F/F)}=\Bbb V(F)$ is identified
with $K$ via $v\mapsto v(\text{\bf 1}_K)$. 

The action of $\sigma\in\text{\rm Gal}(\overline F/F)$ on 
$g\in Sp(\Bbb V,D)$ is defined by 
$v\cdot g^{\sigma}=(v^{\sigma^{-1}}\cdot g)^{\sigma}$. Then 
the fixed point subgroup 
$Sp(\Bbb V,D)^{\text{\rm Gal}(\overline F/F)}$ is identified with 
$Sp(K,D)$ via $g\mapsto g|_K$.

Put $S=\text{\rm Res}_{K/F}\Bbb G_m$ which is identified with the
multiplicative group $\Bbb V^{\times}$. Then $S(F)$ is identified with
the multiplicative group $K^{\times}$. 

Let $T$ be a subtorus of $S$ wich is identified with 
the multiplicative subgroup of $\Bbb V^{\times}$
consisting of the functions $s$ on 
$\text{\rm Emb}_F(K,\overline F)$ to 
${\overline F}^{\times}$
such that $s(\tau\gamma)=s(\gamma)^{-1}$ for all 
$\gamma\in\text{\rm Emb}_F(K,\overline F)$. 
In other words $T$ is 
a maximal torus of $Sp(\Bbb V,D)$ by identifying $s\in T$ with 
$[v\mapsto v\cdot s]\in Sp(\Bbb V,D)$. 
The fixed point subgroup 
$T^{\text{\rm Gal}(\overline F/F)}=T(F)$ is identified with 
$$
 U_{K/K_+}
 =\{\varepsilon\in O_K^{\times}\mid N_{K/K_+}(\varepsilon)=1\}
 \quad
 \text{\rm by $s\mapsto s(\text{\bf 1}_K)$. }
$$
The group $X(S)$ of the characters over $\overline F$ of $S$ 
is a free $\Bbb Z$-module with $\Bbb Z$-basis 
$\{b_{\delta}\}_{\delta\in\text{\rm Emb}_F(K,\overline F)}$ where 
$b_{\delta}(s)=s(\delta)$ for $s\in S$. The dual torus 
$S\sphat=X(S){\otimes}_{\Bbb Z}\Bbb C^{\times}$ is identified with the
group of the functions $s$ on $\text{\rm Emb}_F(K,\overline F)$ with
values in $\Bbb C^{\times}$. 
The action of $\sigma\in W_F\subset\text{\rm Gal}(\overline F/F)$ on
$S$ induces the action on $X(S)$ such that 
$b_{\delta}^{\sigma}=b_{\delta\sigma}$, and hence the
action on $s\in S\sphat$ is defined by 
$s^{\sigma}(\gamma)=s(\gamma\sigma^{-1})$. 

Since we have a bijection $\dot\rho\mapsto\rho|_K$ of 
$W_K\backslash W_F$ onto $\text{\rm Emb}_F(K,\overline F)$, 
the $\overline F$-algebra $\Bbb V$ (resp. the torus $S$, $S\sphat$) is
identified 
with the set of the left$W_K$-invariant functions on $W_F$ with values
in $\Bbb \overline\overline F$ 
(resp. $\overline F^{\times}$, $\Bbb C^{\times}$). If we
denote by $\widetilde\tau\in W_{K_+}$ a pull back of 
$\tau\in\text{\rm Gal}(K/K_+)$ by the restriction mapping 
$W_{K_+}\to\text{\rm Gal}(K/K_+)$, the torus $T$ is identified with
the set of $s\in S$ such that $s(\widetilde\tau\rho)=s(\rho)^{-1}$ for
all $\rho\in W_F$. Note that
$$
 \widetilde\tau^2\npmod{\overline{[W_K,W_K]}}
 =\delta_K(\alpha_{K/K_+}(\tau,\tau))
$$
where $[\alpha_{K/K_+}]\in H^2(\text{\rm Gal}(K/K_+),K^{\times})$ is
the fundamental class which gives the isomorphism
$$
 \text{\rm Gal}(K/K_+)\,\tilde{\to}\,
 K_+^{\times}/N_{K/K_+}(K^{\times})
 \quad
 (\sigma\mapsto\alpha_{K/K_+}(\sigma,\tau)).
$$

The local Langlands correspondence for the torus $S$ is the isomorphism 
\begin{equation}
 H^1(W_F,S\sphat)\,\tilde{\to}\,
 \text{\rm Hom}(W_K,\Bbb C^{\times})
\label{eq:local-langlands-correspondence-for-torus}
\end{equation}
given by 
$[\alpha]\mapsto[\rho\mapsto\alpha(\rho)(\text{\bf 1}_K)]$. The
inverse mapping is defined as follows. Let
$$
 l:\text{\rm Emb}_F(K,\overline F)\to W_F
$$
be a section of the restriction mapping 
$W_F\to\text{\rm Emb}_F(K,\overline F)$, that is 
$l(\gamma)|_K=\gamma$ for all 
$\gamma\in\text{\rm Emb}_F(K,\overline F)$ and  
$l(\text{\bf 1}_K)=1$, and put
$$
 J(\gamma,\sigma)=l(\gamma)\sigma l(\gamma\sigma)^{-1}\in W_K
 \;\;\text{\rm for $\gamma\in\text{\rm Emb}_F(K,\overline F)$, 
           $\sigma\in W_F$.}
$$
Take a $\psi\in\text{\rm Hom}(W_K,\Bbb C^{\times})$ and define 
$\alpha\in Z^1(W_F,S\sphat)$ by
$$
 \alpha(\sigma)(\rho)
 =\alpha(\sigma\rho^{-1})(1)\cdot\alpha(\rho^{-1})(1)
 \;\;\text{\rm with}\;\;
 \alpha(\sigma)(1)
 =\psi\left(J(\text{\bf 1}_K,\sigma^{-1})^{-1}\right)
$$
for all $\sigma,\rho\in W_F$. Then $\psi\mapsto[\alpha]$ is the
inverse mapping of the isomorphism 
\eqref{eq:local-langlands-correspondence-for-torus}. 

If we restrict the isomorphism 
\eqref{eq:local-langlands-correspondence-for-torus} 
to continuous group homomorphisms, we have an isomorphism 
\begin{equation}
 H^1_{\text{\rm conti}}(W_F,S\sphat)\,\tilde{\to}\,
 \text{\rm Hom}_{\text{\rm conti}}(K^{\times},\Bbb C^{\times})
\label{eq:continuous-local-langlands-correspondence-for-torus}
\end{equation}
via \eqref{eq:local-langlands-correspondence-for-torus} combined with
the isomorphism of the local class filed theory
$$
 \delta_K:K^{\times}\,\tilde{\to}\,
          W_K/\overline{[W_K,W_K]}.
$$
The surjection $x\mapsto x^{1-\tau}$ of $K^{\times}$ onto $U_{K/K_+}$
gives  a canonical inclusion 
\begin{equation}
 \text{\rm Hom}_{\text{\rm conti.}}(U_{K/K_+},\Bbb C^{\times})\subset
 \text{\rm Hom}_{\text{\rm conti.}}(K^{\times},\Bbb C^{\times}).
\label{eq:canonical-inclusion-of-hom-group-of-elliptic-tori}
\end{equation}
The restriction from $S$ to $T$ gives a surjection $X(S)\to X(T)$
whose kernel is the subgroup of $X(S)$ generated by 
$\{b_{\delta}+b_{\tau\delta}\mid\delta\in\text{\rm Emb}_F(K,L)\}$. 
Then the dual torus $T\sphat=X(T){\otimes}_{\Bbb Z}\Bbb C^{\times}$ is
identified with the group of the functions $s$ on 
$\text{\rm Emb}_F(K,\overline F)$ with values in 
$\Bbb C^{\times}$ such that $s(\tau\gamma)=s(\gamma)^{-1}$ for all 
$\gamma\in\text{\rm Emb}_F(K,\overline F)$. As above $T\sphat$ is
identified with the set of the left $W_K$-invariant functions $s$ of
$W_F$ with values in $\Bbb C^{\times}$ such that 
$s(\widetilde\tau\rho)=s(\rho)^{-1}$ for all $\rho\in W_F$.

Then we have

\begin{prop}
\label{prop:local-langlands-correspondence-of-elliptic-tori}
\begin{enumerate}
\item The inclusion $T\sphat\subset S\sphat$ gives a canonical
      inclusion  
\begin{equation}
 H^1_{\text{\rm conti.}}(W_F,T\sphat)\subset  
 H^1_{\text{\rm conti.}}(W_F,S\sphat).
\label{eq:canonical-inclusion-of-cohomology-of-weil-group}
\end{equation}
\item The restriction of the isomorphism 
\eqref{eq:continuous-local-langlands-correspondence-for-torus} to
these included subgroups 
\eqref{eq:canonical-inclusion-of-cohomology-of-weil-group} and 
\eqref{eq:canonical-inclusion-of-hom-group-of-elliptic-tori}
gives the isomorphism
\begin{equation}
  H^1_{\text{\rm conti}}(W_F,T\sphat)\,\tilde{\to}\,
 \text{\rm Hom}_{\text{\rm conti}}(U_{K/K_+},\Bbb C^{\times}).
\label{eq:local-langlands-correspondence-for-symplectic-torus}
\end{equation}
\end{enumerate}
\end{prop}
\begin{proof}
See \cite{Yu2009} for the arguments with general tori. A direct proof
for our specific setting is as follows. 

1) Take a $\beta\in Z^1(W_F,T\sphat)\subset Z^1(W_F,S\sphat)$ such
that $\beta\in B^1(W_F,S\sphat)$, that is, there exists a 
$s\ S\sphat$ such that $\beta(\sigma)=s^{\beta-1}$ for all 
$\sigma\in W_F$. Chose a $\varepsilon\in\Bbb C^{\times}$ such that 
$\varepsilon^2=s(\text{\bf 1}_K)\cdot s(\tau)$. The relation 
$\beta(\sigma)(\tau)=\beta(\sigma)(\text{\bf 1}_K)^{-1}$ for all 
$\sigma\in W_F$ implies 
$$
 s(\sigma|)\cdot s(\widetilde\tau\sigma)
 =s(\text{\bf 1}_K)\cdot s(\tau)
 =\varepsilon^2
$$
for all $\sigma\in W_F$. Then 
$t=[\sigma\mapsto s(\sigma)\varepsilon^{-1}]$ is an element of
$T\sphat$ such that $t^{\sigma-1}=\beta(\sigma)$ for all 
$\sigma\in W_F$. 

2) Put
$$
 \text{\rm Emb}_F(K,\overline F)
 =\{\gamma_i, \tau\gamma_i\mid 1\leq i\leq n\}
$$
and let $l:\text{\rm Emb}_F(K,\overline F)\to W_F$ be a section of the
restriction mapping $W_F\to\text{\rm Emb}_F(K,\overline F)$ such that 
$l(\tau\gamma_i)=\widetilde\tau l(\gamma_i)$ ($1\leq i\leq n$). Take a 
$\theta
 \in\text{\rm Hom}_{\text{\rm conti.}}(K^{\times},\Bbb C^{\times})$
   which corresponds to $\alpha\in Z^1(W_F,S\sphat)$, that is
$$
 \alpha(\sigma)(\text{\bf 1}_K)
 =\theta(x)
$$
for $\sigma\in W_F$ with $x\in K^{\times}$ such that 
$J(\text{\bf 1}_K,\sigma^{-1})^{-1}\npmod{\overline{[W_K,W_K]}}
 =\delta_K(x)$. For any $\sigma\in W_F$, we have
$$
 \alpha(\sigma\widetilde\tau^{-1})
 =\begin{cases}
   \theta(x^{\tau})&:\sigma^{-1}|_K=\gamma_i,\\
   \theta(\alpha_{K/K_+}(\tau,\tau)\cdot x^{\tau})
                   &:\sigma^{-1}|_K=\tau\gamma_i.
  \end{cases}
$$
Since
$$
 \alpha(\sigma)(\rho)
 =\alpha(\sigma\rho^{-1})(\text{\bf 1}_K)\cdot
  \alpha(\rho^{-1})(\text{\bf 1}_K)^{-1}
$$
for all $\sigma,\rho\in W_F$ and 
$K_+^{\times}=N_{K/K_+}(K^{\times})\sqcup
              \alpha_{K/K_+}(\tau,\tau)N_{K/K_+}(K^{\times})$, we have
$\alpha\in Z^1(W_F,T\sphat)$ if and only if
$\theta(N_{K/K_+}(K^{\times})=1$, that is, there exists 
$c\in\text{\rm Hom}_{\text{\rm conti.}}(U_{K/K_+},\Bbb C^{\times})$ 
such that $\theta(x)=c(x^{1-\tau})$ ($x\in K^{\times}$).
\end{proof}

Put $^L\!T=W_F\ltimes T\sphat$. Then a cohomology class 
$[\alpha]\in H^1_{\text{\rm conti}}(W_F,T\sphat)$ defines a continuous
group homomorphism 
\begin{equation}
 \widetilde\alpha:W_F\to\,^LT
 \quad
 (\sigma\mapsto(\sigma,\alpha(\sigma)))
\label{eq:langlands-parameter-for-torus}
\end{equation}
and $[\alpha]\mapsto\widetilde\alpha$ induces a well-defined bijection
$$
 H^1_{\text{\rm conti}}(W_F,T\sphat)\,\tilde{\to}\,
 \text{\rm Hom}_{\text{\rm conti}}^{\ast}(W_F,^LT)
 /\text{\rm ``$T\sphat$-conjugate"}
$$
where $\text{\rm Hom}_{\text{\rm conti}}^{\ast}(W_F,^LT)$ denotes the
set of the continuous group homomorphisms $\psi$ of $W_F$ to $^LT$
such that 
$W_F\xrightarrow{\psi}\,^LT\xrightarrow{\text{\rm proj.}}W_F$ is the
identity map.

\subsection{$\chi$-datum}
\label{subsec:chi-datum}
In this subsection, let us assume that $K/F$ is a Galois
extension and put $\Gamma=\text{\rm Gal}(K/F)$. For a 
$\gamma\in\Gamma$ of order two, let us denote by $K_{\gamma}$ the
intermediate subfield of $K/F$ such that 
$\text{\rm Gal}(K/K_{\gamma})=\langle\gamma\rangle$. 

Let us denote by $SO_{2n+1}(\Bbb C)$ the complex special orthogonal
group with respect to the symmetric matrix
$$
 S=\begin{bmatrix}
    S_1&0\\
       &-2
   \end{bmatrix}
 \;\text{\rm with}\;
 S_1=\begin{bmatrix}
      0&1_n\\
      1_n&0
     \end{bmatrix}
$$
and put
$$
 \Bbb T\sphat
 =\left\{\begin{bmatrix}
          t& & \\
           &t^{-1}& \\
           & &1
         \end{bmatrix}\biggm| t=\begin{bmatrix}
                                t_1&      &   \\
                                   &\ddots&   \\
                                   &      &t_n
                                \end{bmatrix},\;
         t_i\in\Bbb C^{\times}\right\}
$$
a maximal torus of $SO_{2n+1}(\Bbb C)$. 
We have an isomorphism $T\sphat\,\tilde{\to}\,\Bbb T\sphat$ given by
$$
 s\mapsto\text{\rm diag}(s(\gamma_1),\cdots,s(\gamma_n),
                         s(\gamma_{n+1}),\cdots,s(\gamma_{2n}),1)
$$
where 
$\text{\rm Emb}_F(K,\overline F)=\{\gamma_i\}_{1\leq i\leq 2n}$
where $\gamma_1=\text{\bf 1}_K$ and $\gamma_{n+i}=\tau\gamma_i$ 
($1\leq i\leq n$). 
The action
of $W_F$ on $T\sphat$ induces the action on $\Bbb T\sphat$ which
factors through $\Gamma$. 

The Weyl group 
$W(\Bbb T\sphat)
 =N_{SO_{2n+1}(\Bbb C)}(\Bbb T\sphat)/\Bbb T\sphat$ on $\Bbb T\sphat$ is
 identified with a subgroup of the permutation group $S_{2n}$
 generated by
$$
 \begin{pmatrix}
  1&\cdots&n&n+1&\cdots&2n\\
  \sigma(1)&\cdots&\sigma(n)&n+\sigma(1)&\cdots&n+\sigma(n)
 \end{pmatrix}
 \;\text{\rm with $\sigma\in S_n$ and}
$$
$$
 \begin{pmatrix}
   1 &2&\cdots&n&n+1&n+2&\cdots&2n\\
  n+1&2&\cdots&n& 1 &n+2&\cdots&2n
 \end{pmatrix}.
$$
Then any $w\in W(\Bbb T\sphat)$ is represented by 
$$
 \widetilde w=\begin{bmatrix}
               [w]&0\\
           0&\det[w]
          \end{bmatrix}\in N_{SO_{2n+1}(\Bbb C)}(\Bbb T\sphat),
$$
where $[w]\in GL_{2n}(\Bbb Z)$ is the permutation matrix
corresponding to $w\in W(\Bbb T\sphat)\subset S_{2n}$. 

For any $\gamma\in\text{\rm Emb}_F(K,\overline F)=\Gamma$, 
let us denote by $a_{\gamma}$ an element of $X(T\sphat)$ such that 
$a_{\gamma}(s)=s(\gamma)$ for all $s\in T\sphat$. Then 
$$
 \Phi(T\sphat)
 =\left\{a_{\gamma}\cdot a_{\gamma^{\prime}}, a_{\gamma}\mid
   \gamma,\gamma^{\prime}\in\Gamma,\;\gamma\neq\gamma^{\prime}
        \right\}.
$$
is the set of the
roots of $SO_{2n+1}(\Bbb C)$ with respect to $T\sphat=\Bbb T\sphat$
with the simple roots
$$
 \Delta=\{\alpha_i=a_{\gamma_i}\cdot a_{\tau\gamma_{i+1}}, 
          \alpha_n=a_{\gamma_n}\mid 1\leq i<n\}.
$$
Let $\{X_{\alpha},X_{-\alpha},H_{\alpha}\}$ be the standard triple
associate with a simple root $\alpha\in\Delta$. Then 
$s_{\alpha}\in W(\Bbb T\sphat)$ is represented by
$$
 n(s_{\alpha})
 =\exp(X_{\alpha})\cdot\exp(-X_{-\alpha})\cdot\exp(X_{\alpha})
 \in N_{SO_{2n+1}(\Bbb C)}(\Bbb T\sphat)
$$
and $W(\Bbb T\sphat)$ is generated by
$S=\{s_{\alpha}\}_{\alpha\in\Delta}$. 
For any $w\in W(\Bbb T\sphat)$, let $w=s_1s_2\cdots s_r$ 
($s_i\in S$) be a reduced presentation and put
$$
 n(w)=n(s_1)n(s_2)\cdots n(s_r)
 \in N_{SO_{2n+1}(\Bbb C)}(\Bbb T\sphat).
$$
Then $r(w)=\widetilde w^{-1}n(w)\in\Bbb T\sphat$. 

The action of
$\sigma\in W_F$ on $X(T\sphat)$ induced from the action on $T\sphat$ is such
that $a_{\gamma}^{\sigma}=a_{\gamma\sigma}$ for all 
$\gamma\in\text{\rm Emb}_F(K,\overline F)$, and it determines an
element $w(\sigma)\in W(\Bbb T\sphat)$. Then
\cite{LanglandsShelstad1987} shows that the $2$-cocycle 
$t\in Z^2(W_F,\Bbb T\sphat)$ defined by
$$
 t(\sigma,\sigma^{\prime})
 =n(w(\sigma\sigma^{\prime}))^{-1}n(w(\sigma))\cdot
  n(w(\sigma^{\prime}))
 \quad
 (\sigma, \sigma^{\prime}\in W_F)
$$
is split by $r_p:W_F\to\Bbb T\sphat$ defined by $\chi$-data as follows.

For any $\lambda\in\Phi(T\sphat)$, put
$$
 \Gamma_{\lambda}
  =\{\sigma\in\Gamma\mid\lambda^{\sigma}=\lambda\},
 \quad
 \Gamma_{\pm\lambda}
  =\{\sigma\in\Gamma\mid\lambda^{\sigma}=\pm\lambda\}
$$
and put $F_{\lambda}=L^{\Gamma_{\lambda}}$, 
$F_{\pm\lambda}=L^{\Gamma_{\pm\lambda}}$. Then 
$(F_{\lambda}:F_{\pm\lambda})=\text{\rm $1$ or $2$}$ , 
and $\lambda$ is called symmetric if $(F_{\lambda}:F_{\pm\lambda})=2$.

The Galois group $\Gamma$ acts on $\Phi(T\sphat)$ and 
$$
 \Phi(T\sphat)/\Gamma=\{a_{\text{\bf 1}_K}a_{\gamma},
                        a_{\text{\bf 1}_K}\mid 1\neq\gamma\in\Gamma\}.
$$
If $\lambda=a_{\text{\bf 1}_K}a_{\gamma}$, then $\lambda$ is
symmetric if and only if $\gamma\neq\tau$. If further $\gamma^2\neq
1$, then $F_{\lambda}=K$ and $F_{\pm\lambda}=K_+$ and choose a
continuous character 
$\chi_{\lambda}:F_{\lambda}^{\times}=K^{\times}\to\Bbb C^{\times}$ 
such that 
$\chi_{\lambda}|_{F_{\pm\lambda}^{\times}}:K_+^{\times}\to\{\pm 1\}$
is the character of the quadratic extension $K/K_+$. 
We may assume that 
$\chi_{a_{\text{\bf 1}_K}a_{\gamma^{-1}}}
 =\chi_{a_{\text{\bf 1}_K}a_{\gamma}}^{-1}$. 

If $\gamma^2=1$, then
$F_{\lambda}=K_{\gamma}$ 
and $F_{\pm\lambda}=E=K_{\gamma}\cap K_+$ and choose a
continuous character 
$\chi_{\lambda}:
 F_{\lambda}^{\times}=K_{\gamma}^{\times}\to\Bbb C^{\times}$ 
such that 
$\chi_{\lambda}|_{F_{\pm\lambda}^{\times}}:E^{\times}\to\{\pm 1\}$
is the character of the quadratic extension $K_{\gamma}/E$. 

If $\lambda=a_{\text{\bf 1}_K}$ then $F_{\lambda}=K$ and 
$F_{\pm\lambda}=K_+$ and choose a
continuous character 
$\chi_{\lambda}:F_{\lambda}^{\times}=K^{\times}\to\Bbb C^{\times}$ 
such that 
$\chi_{\lambda}|_{F_{\pm\lambda}^{\times}}:K_+^{\times}\to\{\pm 1\}$
is the character of the quadratic extension $K/K_+$. 

These characters are parts of a system of $\chi$-data
$\chi_{\lambda}:F_{\lambda}\to\Bbb C^{\times}$ 
($\lambda\in\Phi(\Bbb T\sphat)$) 
such that
\begin{enumerate}
\item $\chi_{-\lambda}=\chi_{\lambda}^{-1}$ and 
      $\chi_{\lambda^{\sigma}}=\chi_{\lambda}(x^{\sigma^{-1}})$ for
      all $\sigma\in\Gamma$, and
\item $\chi_{\lambda}=1$ if $\lambda$ is not symmetric.
\end{enumerate}
With this $\chi$-data and the gauge 
$$
 p:\Phi(\Bbb T\sphat)\to\{\pm 1\}\;\text{\rm s.t.}\;
 p(\lambda)=\begin{cases}
             1&:\lambda>0,\\
            -1&:\lambda<0,
            \end{cases}
$$
the mechanism of \cite{LanglandsShelstad1987} gives 
a $r_p:W_F\to\Bbb T\sphat$ such that
$$
 t(\sigma,\sigma^{\prime})
 =r_p(\sigma)^{\sigma^{\prime}}r_p(\sigma\sigma^{\prime})^{-1}
  r_p(\sigma^{\prime})
 \;\;
 \text{\rm for all $\sigma, \sigma^{\prime}\in W_F$}
$$
and
\begin{align*}
 r_p(\sigma)=&\prod_{\gamma\in\Gamma,\gamma^2\neq 1}
         \prod_{0<\lambda\in
               \{a_{\text{\bf 1}_K} a_{\gamma}\}_{\Gamma}}
         \chi_{\lambda}(x)^{\check\lambda}
         \times
         \prod_{\stackrel{\scriptstyle \gamma\in\Gamma,\gamma^2=1}
                         {\gamma\neq 1,\tau}}
         \prod_{0<\lambda\in
               \{a_{\text{\bf 1}_K}a_{\gamma}\}_{\Gamma}}
         \chi_{\lambda}(N_{K/F_{\lambda}}(x))^{\check\lambda}\\
        &\times
        \prod_{0<\lambda\in\{a_{\text{\bf 1}_K}\}_{\Gamma}}
         \chi_{\lambda}(x)^{\check\lambda}
\end{align*}
if $\dot \sigma=(1,x)\in W_{K/F}=\Gamma\ltimes_{\alpha_{K/F}}K^{\times}$,
where $\{\alpha\}_{\Gamma}$ is the $\Gamma$-orbit of
$\alpha\in\Phi(\Bbb T\sphat)$ and $\check\lambda$ is the co-root of 
$\lambda$. Then we have a group homomorphism 
\begin{equation}
 ^LT=W_F\ltimes\Bbb T\sphat\to SO_{2n+1}(\Bbb C)
 \quad
 ((\sigma,s)\mapsto n(w(\sigma))r_p(\sigma)^{-1}s).
\label{eq:langlands-shelstad-homomorphism-of-torus-to-dual-group}
\end{equation}
If we put $r(\sigma)=r(w(\sigma))$ for $\sigma\in W_F$, we have
$$
 t(\sigma,\sigma^{\prime})
 =r(\sigma)^{\sigma^{\prime}}r(\sigma\sigma^{\prime})^{-1}
  r(\sigma^{\prime})
 \qquad
 (\sigma, \sigma^{\prime}\in W_F).
$$
Now 
$\chi_p(\sigma)=r(\sigma)\cdot r_p(\sigma)^{-1}$ ($\sigma\in W_F$)
 define an element of $Z^1(W_F,\Bbb T\sphat)$ and the group
 homomorphism
 \eqref{eq:langlands-shelstad-homomorphism-of-torus-to-dual-group} is
\begin{equation}
 ^LT=W_F\ltimes\Bbb T\sphat\to SO_{2n+1}(\Bbb C)
 \quad
 ((\sigma,s)\mapsto\widetilde w(\sigma)\chi_p(\sigma)\cdot s).
\label{eq:modified-langlands-shelstad-homomorphism-of-torus-to-dual-group}
\end{equation}
Let $c\in \text{\rm Hom}_{\text{\rm conti}}(U_{K/K_+},\Bbb C^{\times})$
be the character corresponding to the cohomology class 
$[\chi_p]\in H_{\text{\rm conti.}}^1(W_F,\Bbb T\sphat)$ by the local
Langlands correspondence of torus 
\eqref{eq:local-langlands-correspondence-for-symplectic-torus}.

\subsection{Explicit value of $c(-1)$}
\label{subsec:explicit-value-of-c(-1)}
From now on, we will assume that $K/F$ is a tamely ramified
Galois extension and put $\Gamma=\text{\rm Gal}(K/F)$. 

The structure of the Galois group $\text{\rm Gal}(K/F)$ is well
understood: 
\begin{equation}
 \text{\rm Gal}(K/F)=\langle\delta,\rho\rangle
\label{eq:generator-of-tame-galois-group}
\end{equation}
where $\text{\rm Gal}(K/K_0)=\langle\delta\rangle$ with the maximal
unramified subextension $K_0/F$ of $K/F$ and 
$\rho|_{K_0}\in\text{\rm Gal}(K_0/F)$ is the inverse of the Frobenius
automorphism. There exists a prime element $\varpi_K$ of $K$ such that
$\varpi_K^e\in K_0$. Then 
$\sigma\mapsto\varpi_K^{1-\sigma}\npmod{\frak{p}_K}$ is an
injective group homomorphism of $\text{\rm Gal}(K/K_0)$ into $\Bbb
K^{\times}$, and hence $e|q^f-1$. Put 
$\rho^f=\delta^m$ with $0\leq m<e$.
We have a relation $\rho^{-1}\delta\rho=\delta^q$ 
due to Iwasawa \cite{Iwasawa1955} and hence
$$
 \delta^m=\rho^{-1}\delta^m\rho=\delta^{qm}
$$
that is $m(q-1)\equiv 0\npmod{e}$. So we have
\begin{equation}
 \rho^{f(q-1)}=1
\label{eq:order-of-rho-divide-f(q-1)}
\end{equation}
Since $f$ divides $\text{\rm ord}(\rho)$, we have
$$
 \text{\rm ord}(\rho)=f\cdot\frac e{\text{\rm GCD}\{e,m\}}.
$$
The structure of the elements of order two in $\text{\rm Gal}(K/F)$
plays an important role in our arguments, and we have

\begin{prop}\label{prop:order-two-element-in-tamely-ramified-galois-group}
$H=\{\gamma\in\text{\rm Gal}(K/F)\mid\gamma^2=1\}
  \subset Z(\text{\rm Gal}(K/F))$ 
and
$$
 H=\begin{cases}
    \{1,\delta^{\frac e2}\}&:f=\text{\rm odd}\;\text{\rm or}\;
                             \left\{\begin{array}{l}
                                     e=\text{\rm even},\\
                                     m=\text{\rm odd}
                                    \end{array}\right.\\
    \{1,\rho^{\frac f2}\delta^{-\frac m2}\}
                           &:e=\text{\rm odd}\,,m=\text{\rm even}\\
    \{1,\rho^{\frac f2}\delta^{\frac{e-m}2}\}
                           &:e=\text{\rm odd}\,,m=\text{\rm odd}\\
    \{1,\delta^{\frac e2},\rho^{\frac f2}\delta^{-\frac m2},
        \rho^{\frac f2}\delta^{\frac{e-m}2}\}
               &:f=\text{\rm even}\,,e=\text{\rm even}\,,
                 m=\text{\rm even}.
  \end{cases}
$$
For $\gamma\in\text{\rm Gal}(K/F)$ of order two, the quadratic
extension $K/K_{\gamma}$ is ramified  if and only if 
$\gamma\in\text{\rm Gal}(K/K_0)$.
\end{prop}
\begin{proof}
Take a $1\neq\gamma\in\text{\rm Gal}(K/F)$ such that $\gamma^2=1$. 

If $\gamma\in\text{\rm Gal}(K/K_0)$, then $e$ is even and
$\gamma=\delta^{\frac e2}$ 
is the unique element of order $2$ of the normal subgroup 
$\text{\rm Gal}(K/K_0)$. So $\gamma\in Z(\text{\rm Gal}(K/F))$. In
this case $K_0\subset K_{\gamma}$ and $K/K_{\gamma}$ is ramified
extension. 

Assume that $\gamma\not\in\text{\rm Gal}(K/K_0)$. Then 
$\gamma|_{K_0}\in\text{\rm Gal}(K_0/F)$ is of order two 
(hence $f=2f^{\prime}$ is even), 
and $\gamma=\rho^{f^{\prime}}\delta^a$ with $0\leq a<e$. 
Then $K/K_{\gamma}$ is unramified
  extension, because if it was not the case we have $f(K_{\gamma}/F)=f(K/F)$
  and hence $K_0\subset K_{\gamma}$ which means 
$$
 \gamma\in\text{\rm Gal}(K/K_{\gamma})\subset\text{\rm Gal}(K/K_0),
$$
contradicting to the assumption 
$\gamma\not\in\text{\rm Gal}(K/K_0)$.
Then $f(K_{\gamma}/F)=f^{\prime}$ and $e(K_{\gamma}/F)=e$, and hence 
$e|q^{f^{\prime}}-1$. So we have
$$
 1=\gamma^2=\rho^f\rho^{-f^{\prime}}\delta^a\rho^{f^{\prime}}\delta^a
  =\delta^{m+aq^{f^{\prime}}+a}
  =\delta^{2a+m},
$$
hence $2a\equiv -m\npmod{e}$. Then 
$a\equiv -\frac m2\;\text{\rm or}\;\frac{e-m}2\npmod{e}$ 
if $e$ is even (hence $m$ is even), and
$$
 a\npmod{e}=\begin{cases}
             -\frac m2&:\text{\rm if $m$ is even},\\
             \frac{e-m}2&:\text{\rm if $m$ is odd}
            \end{cases}
$$
if $e$ is odd. We have $e|q^{f^{\prime}}-1$ hence
$$
 \delta\gamma
 =\rho^{f^{\prime}}\delta^{q^{f^{\prime}}+a}
 =\rho^{f^{\prime}}\delta^{1+a}
 =\gamma\delta.
$$
Now we have
\begin{equation}
 \rho^{f^{\prime}(q-1)}=1.
\label{eq:f-prime-(q-1)-kill-rho}
\end{equation}
In fact 
$\text{\rm Gal}(K_{\gamma}/F)
 =\langle\delta^{\prime},\rho^{\prime}\rangle$ with 
$\delta^{\prime}=\delta|_{K_{\gamma}}, \rho^{\prime}=\rho|_{K_{\gamma}}$. Then 
$(\rho^{\prime})^{f^{\prime}(q-1)}=1$, that is
$$
 \rho^{f^{\prime}(q-1)}
 \in\text{\rm Gal}(K/K_{\gamma})=\langle\gamma\rangle.
$$
If $\rho^{f^{\prime}(q-1)}\neq 1$, then 
$\rho^{f^{\prime}(q-1)}=\gamma=\rho^{f^{\prime}}\delta^a$, 
therefore 
$$
 \rho^{f^{\prime}q}=\rho^f\delta^a=\delta^{m+a}
 \in\text{\rm Gal}(K/K_0)
$$
and hence $f$ divides $f^{\prime}q$, contradicting to the assumption
that $q$ is odd. Now we have
$$
 \gamma\rho=\rho^{f^{\prime}+1}\delta^{qa}
 =\rho\gamma\cdot\delta^{a(q-1)}.
$$
For $a=-\frac m2$ or $a=\frac{e-m}2$, 
we have $a(q-1)\equiv 0\npmod{e}$ if and only if 
$$
 \frac{q-1}2\equiv 0\npmod{\frac e{\text{\rm GCD}\{e,m\}}}
$$
which is equivalent to 
$\rho^{f\cdot\frac{q-1}2}=\rho^{f^{\prime}(q-1)}=1$. 
Then \eqref{eq:f-prime-(q-1)-kill-rho} implies
$\gamma\rho=\rho\gamma$. Then we have $\gamma$ is an element of the
center of $\text{\rm Gal}(K/F)$. 
\end{proof}

Put $\widetilde c(x)=c(x^{1-\tau})$ for $x\in K^{\times}$. Then we
have
\begin{align*}
 \widetilde c(x)
 &=\chi_p(1,x)(\text{\bf 1}_K)\\
 &=\prod_{\gamma\in\Gamma,\gamma^2\neq 1}
    \chi_{a_{\text{\bf 1}_K}a_{\gamma}}(x)\times
   \prod_{\stackrel{\scriptstyle \gamma\in\Gamma,\gamma^2=1}
                         {\gamma\neq 1,\tau}}
    \chi_{a_{\text{\bf 1}_K}a_{\gamma}}(N_{K/K_{\gamma}}(x))\times
    \chi_{a_{\text{\bf 1}_K}}(x)^2.
\end{align*}
Since 
$\chi_{a_{\text{\bf 1}_K}a_{\gamma^{-1}}}
 =\chi_{a_{\text{\bf 1}_K}a_{\gamma}}^{-1}$ for $\gamma\in\Gamma$, we
have 
$$
 \widetilde c(x)=\chi_{a_{\text{\bf 1}_K}}(x^{1-\tau})
 \quad
 (x\in K^{\times})
$$
if $H=\{1,\tau\}$, and
$$
 \widetilde c(x)
 =\chi_{a_{\text{\bf 1}_K}a_{\delta^{\prime}}}
                  (N_{K/K_{\delta^{\prime}}}(x))\cdot
    \chi_{a_{\text{\bf 1}_K}a_{\tau\delta^{\prime}}}
                  (N_{K/K_{\tau\delta^{\prime}}}(x))\cdot
    \chi_{a_{\text{\bf 1}_K}}(x^{1-\tau})
 \quad
 (x\in K^{\times})
$$
if 
$H=\{1,\tau,\delta^{\prime}=\delta^{\frac e2},\tau\delta^{\prime}\}$. In
this case, since $f$ is even, $K/K_+$ is unramified so that 
$\tau\not\in\text{\rm Gal}(K/K_0)=\langle\delta\rangle$. 
We have 

\begin{prop}\label{prop:explicit-value-of-c(-1)}
If $|H|=2$, then
$$
 c(-1)=\begin{cases}
       (-1)^{\frac{q-1}2}
         &:\text{\rm if $K/K_+$ is ramified},\\
        1&:\text{\rm if $K/K_+$ is unramified,}.
       \end{cases}
$$
If $|H|=4$, then 
$$
 c(-1)=-(-1)^{\frac{q^{f_+}-1}2}.
$$
Note that $K/F$ is totally ramified if $K/K_+$ is ramified.
\end{prop}
\begin{proof}
If $|H|=2$, we have $c(x)=\chi_{a_{\text{\bf 1}_K}}(x)$ for 
$x\in U_{K/K_+}$ so that
$$
 c(-1)=\left(-1,K/K_+\right)
      =\begin{cases}
        1&:\text{\rm if $K/K_+$ is unramified,}\\
       (-1)^{\frac{q-1}2}
         &:\text{\rm if $K/K_+$ is ramified}.
       \end{cases}
$$
From now on, we will consider the case of $|H|=4$.
Put $H=\{1,\tau,\delta^{\prime},\tau\delta^{\prime}\}$ and let 
$E=K^H$ be the fixed subfield of $K/F$. Then we have
$$
 \text{\rm Gal}(K_{\delta^{\prime}}/E)
 =\langle\tau|_{K_{\delta^{\prime}}}\rangle,
 \quad
 \text{\rm Gal}(K_{\tau\delta^{\prime}}/E)
 =\langle\tau|_{K_{\tau\delta^{\prime}}}\rangle.
$$
Put $K_{\delta^{\prime}}=E(\eta)$ with $\eta^2\in E$, or equivalently 
$\eta^{\tau}=-\eta$. Then we have
\begin{align*}
 c(-1)&=\widetilde c(\eta)
  =\chi_{a_{\text{\bf 1}_K}a_{\delta^{\prime}}}
    (N_{K/K_{\delta^{\prime}}}(\eta))\cdot
  \chi_{a_{\text{\bf 1}_K}a_{\tau\delta^{\prime}}}
   (N_{K/K_{\tau\delta^{\prime}}}(\eta))\cdot
  \chi_{a_{\text{\bf 1}_K}}(\eta^{1-\tau})\\
 &=\left(\eta^2,K_{\delta^{\prime}}/E\right)\cdot
   \left(-\eta^2,K_{\tau\delta^{\prime}}/E\right)\cdot
   \left(-1,K/K_+\right).
\end{align*}
Since $K/K_+$ is unramified, we have $\left(-1,K/K_+\right)=1$. Since 
$N_{K_{\delta^{\prime}}/E}(\eta)=-\eta^2$, we have 
$\left(-\eta^2,K_{\delta^{\prime}}/E\right)=1$. Since
$K_{\delta^{\prime}}/E$ is unramified, we have 
$\left(-1,K_{\delta^{\prime}}/E\right)=1$. Hence 
$\left(\eta^2,K_{\delta^{\prime}}/E\right)=1$. By the standard formula
of the norm residue symbol, we have
\begin{equation}
 \left(\eta,K/K_{\delta^{\prime}}\right)
 =\left(-\eta^2,K/E\right)
 \in\text{\rm Gal}(K/K_{\delta^{\prime}})
 \subset\text{\rm Gal}(K/E)
\label{eq:standard-formula-of-norm-residue-in-k-over-e}
\end{equation}
since $N_{K_{\delta^{\prime}}/E}(\eta)=-\eta^2$, and
$$
 \left(-\eta^2,K/E\right)
 =(\left(-\eta^2,K_{\delta^{\prime}}/E\right),
   \left(-\eta^2,K_{\tau\delta^{\prime}}/E\right))
$$
in 
$\text{\rm Gal}(K/E)
 =\text{\rm Gal}(K_{\delta^{\prime}}/E)\times 
  \text{\rm Gal}(K_{\tau\delta^{\prime}}/E)$ by 
$\sigma=(\sigma|_{K_{\delta^{\prime}}},\sigma|_{K_{\tau\delta^{\prime}}})$. Note
  that we have $\left(-\eta^2,K_{\delta^{\prime}}/E\right)=1$. So if 
$\left(\eta,K/K_{\delta^{\prime}}\right)=1$, then 
$\left(-\eta^2,K_{\tau\delta^{\prime}}/E\right)=1$. If 
$\left(\eta,K/K_{\delta^{\prime}}\right)\neq 1$, then 
$\left(\eta,K/K_{\delta^{\prime}}\right)=\delta^{\prime}$, hence
$$
 \left(-\eta^2,K_{\tau\delta^{\prime}}/E\right)
 =\delta^{\prime}|_{K_{\tau\delta^{\prime}}}
 \neq 1.
$$
So we have 
$\left(-\eta^2,K_{\tau\delta^{\prime}}/E\right)
 =\left(\eta,K/K_{\delta^{\prime}}\right)$. The restriction mapping 
$\text{\rm Gal}(K/E)\to\text{\rm Gal}(K_+/E)$ sends 
$\left(-\eta^2,K/E\right)$ to $\left(-\eta^2,K_+/E\right)$. Since the
restriction mapping gives the isomorphism
$$
 \text{\rm Gal}(K/K_{\delta^{\prime}})\,\tilde{\to}\,
 \text{\rm Gal}(K_+/E).
$$
Hence \eqref{eq:standard-formula-of-norm-residue-in-k-over-e} shows 
$\left(\eta,K/K_{\delta^{\prime}}\right)
 =\left(-\eta^2,K_+/E\right)$. Since $K_+/E$ is a ramified quadratic
 extension and $\eta^2\in E$ is not square in $E$, we have
$$
 \left(-\eta^2,K_+/E\right)
 =\left(-1,K_+/E\right)\cdot\left(\eta^2,K_+/E\right)
 =(-1)^{\frac{q^{f_+}-1}2}\cdot(-1).
$$
\end{proof}

The following proposition will be used in the next two sections.

\begin{prop}\rm\label{prop:c-is-trivial-mod-p-square}
We can choose the $\chi$-data 
$\{\chi_{\lambda}\}_{\lambda\in\Phi(\Bbb T\sphat)}$ so that 
$c(x)=1$ for all $x\in U_{K/K_+}\cap(1+\frak{p}_K^2)$. 
\end{prop}
\begin{proof}
If $K/K_+$ is ramified, then $K/F$ is totally ramified and 
$c(x)=\chi_{a_{\text{\bf 1}_K}}(x)$ for $x\in U_{K/K_+}$. Since
$$
 \left(1+\frak{p}_{K_+},K/K_+\right)=1
 \;\;\text{\rm and}\;\;
 \left(1+\frak{p}_K^2\right)\cap K_+^{\times}
 =1+\frak{p}_{K_+},
$$
we can assume that $\chi_{a_{\text{\bf 1}_K}}$ is trivial on 
$1+\frak{p}_K^2$. Then 
$c(x)=1$ for all $x\in U_{K/K_+}\cap(1+\frak{p}_K^2)$.

Assume that $K/K_+$ is unramified. Since
$$
 \left(1+\frak{p}_{K_+},K/K_+\right)=1
 \;\;\text{\rm and}\;\;
 \left(1+\frak{p}_K\right)\cap K_+^{\times}
 =1+\frak{p}_{K_+},
$$
we can choose $\chi_{a_{\text{\bf 1}_K}}$ so that 
$\chi_{a_{\text{\bf 1}_K}}(1+\frak{p}_K)=1$. If further $|H|=4$, then 
$K_{\delta^{\prime}}/E$ is unramified, and $K_{\tau\delta^{\prime}}/E$ is
ramified. Since
$$
 (1+\frak{p}_{K_{\delta^{\prime}}})\cap E^{\times}
 =(1+\frak{p}_{K_{\tau\delta^{\prime}}}^2)\cap E
 =1+\frak{p}_E
$$
and
$(x,K_{\delta^{\prime}}/E)=(x,K_{\tau\delta^{\prime}}/E)=1$ for all 
$x\in 1+\frak{p}_E$, we can assume that
$$
 \chi_{a_{\text{\bf 1}_K}a_{\delta^{\prime}}}
             (1+\frak{p}_{K_{\delta^{\prime}}})=1,
 \quad
 \chi_{a_{\text{\bf 1}_K}a_{\tau\delta^{\prime}}}
             (1+\frak{p}_{K_{\tau\delta^{\prime}}})=1.
$$
Since $K/K_{\delta^{\prime}}$ is ramified and 
$K/K_{\tau\delta^{\prime}}$ is unramified, we have
$$
 N_{K/K_{\delta^{\prime}}}(1+\frak{p}_K^2)
 =1+\frak{p}_{K_{\delta^{\prime}}},
 \quad
 N_{K/K_{\tau\delta^{\prime}}}(1+\frak{p}_K^2)
 =1+\frak{p}_{K_{\tau\delta^{\prime}}}^2
$$
Hence $\widetilde c(x)=1$ for all $x\in 1+\frak{p}_K^2$. Because 
$K/K_+$ is unramified, we can prove by induction on $k$ that
$x\mapsto x^{1-\tau}$ is surjection of $1+\frak{p}_K^k$ onto 
$U_{K/K_+}\cap(1+\frak{p}_K^k)$. Then $c(x)=1$ for all 
$x\in U_{K/K_+}\cap(1+\frak{p}_K^2)$. 
\end{proof}

\subsection{$L$-parameters associated with characters of tame elliptic
            tori}
\label{subsec:l-parameter-associated-with-character-of-tame-elliptic-tori}
By local Langlands correspondence of tori described in Proposition 
\ref{prop:local-langlands-correspondence-of-elliptic-tori}, 
the continuous character $\theta$ of $U_{K/K_+}$ which parametrizes
the irreducible representation $\delta_{\beta,\theta}$ of 
$Sp_{2n}(O_F)$ determines the cohomology class 
$[\alpha]\in H^1_{\text{\rm conti}}(W_F,T\sphat)$. Then we have a group homomorphism
\begin{equation}
 \varphi:W_F\xrightarrow{\widetilde\alpha}
         {^LT}\xrightarrow{
 \text{\rm
   \eqref{eq:modified-langlands-shelstad-homomorphism-of-torus-to-dual-group}}}
         SO_{2n+1}(\Bbb C).
\label{eq:langlands-parameter-of-kaletha}
\end{equation}
The construction of $\varphi$ shows that 
$\varphi(\sigma)\in SO_{2n+1}(\Bbb C)$ is of the form
\begin{equation}
 \varphi(\sigma)=\begin{bmatrix}
                  \varphi_1(\sigma)&0\\
                  0&\det\varphi_1(\sigma)
                 \end{bmatrix}
 \;\text{\rm with}\;\varphi_1(\sigma)\in O(S_1,\Bbb C)
 \quad
 (S_1=\begin{bmatrix}
       0&1_n\\
       1_n&0
      \end{bmatrix})
\label{eq:structure-of-l-parameter-in-so(2n+1)}
\end{equation}
for $\sigma\in W_F$. The definition of 
\eqref{eq:modified-langlands-shelstad-homomorphism-of-torus-to-dual-group}
shows that 
\begin{align*}
 \text{\rm tr}\varphi_1(\sigma)
 &=\sum_{\gamma\in\text{\rm Emb}_F(K,\overline F),
          \gamma\sigma=\gamma}
    \chi_p(\sigma)(\gamma)\cdot\alpha(\sigma)(\gamma)\\
 &=\sum_{\dot\gamma\in W_K\backslash W_F, 
          \gamma\sigma\gamma^{-1}\in W_K}
    \chi_p(\sigma)(\gamma)\cdot\alpha(\sigma(\gamma)\\
 &=\sum_{\dot\gamma\in W_K\backslash W_F, 
          \gamma\sigma\gamma^{-1}\in W_K}
    \psi_c\cdot\psi_{\theta}(\gamma\sigma\gamma^{-1})
\end{align*}
for $\sigma\in W_F$.
Here $\psi_c$ (resp. $\psi_{\theta}$) is the element of 
$\text{\rm Hom}_{\text{\rm conti}}(W_K,\Bbb C)$ corresponding to 
$c$ (resp. $\theta$) by
$$
 \text{\rm Hom}_{\text{\rm conti}}(U_{K/K_+},\Bbb C^{\times})
 \xrightarrow
          {\eqref{eq:canonical-inclusion-of-hom-group-of-elliptic-tori}}
 \text{\rm Hom}_{\text{\rm conti}}(K^{\times},\Bbb C^{\times})
 \xrightarrow{\delta_K}
 \text{\rm Hom}_{\text{\rm conti.}}(W_K,\Bbb C^{\times}).
$$
This shows that 
$\varphi_1$ is the
induced representation of $W_F$ from the character
$\psi_c\cdot\psi_{\theta}$ of $W_K$. So $\varphi_1$ factors through the
canonical surjection 
$$
 W_F\to W_{K/F}=W_F/\overline{[W_K,W_K]}
$$ 
and, if we put $\vartheta=c\cdot\theta$ and 
$\widetilde\vartheta(x)=\vartheta(x^{1-\tau})$ ($x\in K^{\times}$),
we have
\begin{equation}
 \text{\rm tr}\varphi_1(\sigma,x)
 =\begin{cases}
   0&:\sigma\neq 1,\\
   \sum_{\gamma\in\text{\rm Gal}(K/F)}
     \widetilde{\vartheta}(x^{\gamma})
    &:\sigma=1
  \end{cases}
\label{eq:chracter-of-o(2n)-part-of-l-parameter}
\end{equation}
for 
$(\sigma,x)\in W_{K/F}
 =\text{\rm Gal}(K/F)\ltimes_{\alpha_{K/F}}K^{\times}$ with the
 fundamental class 
$[\alpha_{K/F}]\in H^2(\text{\rm Gal}(K/F),K^{\times})$.

The representation space $V_{\vartheta}$ of the induced representation 
$\text{\rm Ind}_{K^{\times}}^{W_{K/F}}\widetilde\vartheta$ is the
complex vector space of the $\Bbb C$-valued function $v$ on 
$\text{\rm Gal}(K/F)$ with the action of $(\sigma,x)\in W_{K/F}$
$$
 (x\cdot v)(\gamma)
 =\widetilde\vartheta(x^{\gamma})\cdot v(\gamma),
 \quad
 (\sigma\cdot v)(\gamma)
 =\widetilde\vartheta(\alpha_{K/F}(\sigma,\sigma^{-1}\gamma))
    \cdot v(\sigma^{-1}\gamma).
$$
A $\Bbb C$-basis $\{v_{\rho}\}_{\rho\in\text{\rm Gal}(K/F)}$ 
of $V_{\vartheta}$ is defined by
$$
 v_{\rho}(\gamma)=\begin{cases}
                     1&:\gamma=\rho,\\
                     0&:\gamma\neq\rho.
                    \end{cases}
$$
Then 
$$
 x\cdot v_{\rho}=\widetilde\vartheta(x^{\rho})\cdot v_{\rho},
 \quad
 \sigma\cdot v_{\rho}
 =\widetilde\vartheta(\alpha_{K/F}(\sigma,\rho))\cdot 
   v_{\sigma\rho}
$$
for $(\sigma,x)\in W_{K/F}$. The following proposition will be used to
analyze $\text{\rm Ind}_{K^{\times}}^{W_{K/F}}\widetilde\vartheta$ in
detail. 

\begin{prop}
\label{prop:conductor-and-conjugate-triviality-of-tilde-vartheta}
Assume $l\geq 2$, then 
\begin{enumerate}
\item 
$$
 \text{\rm Min}\left\{2\leq k\in\Bbb Z\biggm|
        \begin{array}{l}
         \widetilde\vartheta(\alpha)=1\\
         \forall\alpha\in 1+\frak{p}_K^k
        \end{array}\right\}
 =\begin{cases}
   e(r-1)+1&:\text{\rm $K/K_+$ is unramified,}\\
   e(r-1)  &:\text{\rm $K/K_+$ is ramified.}
  \end{cases}
$$
\item For an integer $k\geq 2$
$$
 \left\{\sigma\in\text{\rm Gal}(K/F)\biggm|
   \begin{array}{l}
    \widetilde\vartheta(x^{\sigma})=\widetilde\vartheta(x)\\
     \text{\rm for $\forall x\in 1+\frak{p}_K^k$}
   \end{array}\right\}
 =\begin{cases}
   \text{\rm Gal}(K/F)&:k>e(r-1),\\
   \text{\rm Gal}(K/K_0)&:k=e(r-1),\\
   \{1\}&:k<e(r-1).
  \end{cases}
$$
\end{enumerate}
\end{prop}
\begin{proof}
Note that $\vartheta(x)=\theta(x)$ for all 
$x\in U_{K/K_+}\cap(1+\frak{p}_K^2)$ 
(by Proposition \ref{prop:c-is-trivial-mod-p-square}) and 
$\theta(x)=1$ for all $x\in U_{K/K_+}\cap(1+\frak{p}_K^{er})$.
Take an integer $k$ such that $0\leq k\leq el^{\prime}$, and hence 
$2\leq el\leq er-k$. Then, for any $x\in O_K$, we have
$$
 (1+\varpi_F^r\varpi_K^{-k}x)^{1-\tau}
 \equiv
 1+\varpi_F^r(\varpi_K^{-k}x-\varpi_K^{-k\tau}x^{\tau})
  \npmod{\frak{p}_K^{er}}
$$
since $2(er-k)\geq er$. Hence, for 
$\alpha=1+\varpi_F\varpi_K^{-k}x\in 1+\frak{p}_K^{er-k}$ ($x\in O_K$), 
we have
\begin{equation}
 \widetilde\vartheta(\alpha)
  =\psi\left(T_{K/F}\left(
        (\varpi_K^{-k}x-\varpi_K^{-k\tau}x^{\tau})\beta\right)\right)
  =\psi\left(2T_{K/F}(\varpi_K^{-k}x\beta)\right). 
\label{eq:explicit-value-of-tilde-vartheta-alpha}
\end{equation}
1) The statement $\widetilde\vartheta(\alpha)=1$ for all 
$\alpha\in 1+\frak{p}_K^{er-k}$ is equivalent to the statement 
$T_{K/F}\left(\varpi_K^{-k}x\beta\right)\in O_F$ for
all $x\in O_K$, or to the statement 
$\varpi_K^{-k}(\beta)\in\mathcal{D}(K/F)^{-1}=\frak{p}_K^{1-e}$, 
and hence $\text{\rm ord}_K(\beta)\geq k-e-1$. Since
$$
 \text{\rm ord}_K(\beta)
 =\begin{cases}
   0&:\text{\rm $K/K_+$ is unramified,}\\
   1&:\text{\rm $K/K_+$ is ramified}
  \end{cases}
$$
the proof is completed.

\noindent
2) Because $K/F$ is tamely ramified, we have
\begin{align}
 V_t(K/F)
 &=\{\sigma\in\text{\rm Gal}(K/F)\mid
             \text{\rm ord}_K(x^{\sigma}-x)\geq t+1\;
              \forall x\in O_K\} \nonumber\\
 &=\begin{cases}
    \text{\rm Gal}(K/F)&:t<0,\\
    \text{\rm Gal}(K/K_0)&:0\leq t<1,\\
    \{1\}&:1\leq t.
   \end{cases}
\label{eq:ramification-group-of-tamely-ramified-extension}
\end{align}
Take a $\sigma\in\text{\rm Gal}(K/F)$. Then, by 
\eqref{eq:explicit-value-of-tilde-vartheta-alpha}, we have
$$
 \widetilde\vartheta(\alpha^{\sigma})
 =\psi\left(2T_{K/F}(\varpi_K^{-k\sigma}x^{\sigma}\beta)\right)
  =\psi\left(2T_{K/F}(\varpi_K^{-k}x\beta^{\sigma})\right).
$$
So the statement 
$\widetilde\vartheta(\alpha^{\sigma})=\widetilde\vartheta(\alpha)$ for
all $\alpha\in 1+\frak{p}_K^{er-k}$ is equivalent to the statement 
$\varpi_K^{-k}(\beta^{\sigma}-\beta)
 \in\mathcal{D}(K/F)^{-1}=\frak{p}_K^{1-e}$, or to the statement 
$$
 \text{\rm ord}_K(x^{\sigma}-x)\geq k-e+1
 \;\;\text{\rm for all $x\in O_K$}
$$
since $O_K=O_F[\beta]$, which is equivalent to 
$\sigma\in V_{k-e}$. Then
\eqref{eq:ramification-group-of-tamely-ramified-extension} completes
the proof.
\end{proof}

\begin{prop}
\label{prop:induced-representation-of-vartheta-is-irreducible}
The induced representation 
$\text{\rm Ind}_{K^{\times}}^{W_{K/F}}\widetilde\vartheta$ is
irreducible.
\end{prop}
\begin{proof}
Take a $0\neq T\in\text{\rm End}_{W_{K/F}}(V_{\vartheta})$. Since
$$
 Tv_{\rho}=T(\rho\cdot v_1)=\rho\cdot Tv_1
$$
for all $\rho\in\text{\rm Gal}(K/F)$, we have $Tv_1\neq 0$. If 
$(Tv_1)(\gamma)\neq 0$ for a $\gamma\in\text{\rm Gal}(K/F)$, then we
have
\begin{align*}
 \widetilde\vartheta(x^{\gamma})\cdot(Tv_1)(\gamma)
 &=(x\cdot Tv_1)(\gamma)=T(x\cdot v_1)(\gamma)\\
 &=(T(\widetilde\vartheta(x)\cdot v_1))(\gamma)
  =\widetilde\vartheta(x)\cdot(Tv_1)(\gamma),
\end{align*}
and hence 
$\widetilde\vartheta(x^{\gamma})=\widetilde\vartheta(x)$ 
for all $x\in K^{\times}$. Then $\gamma=1$ by Proposition 
\ref{prop:conductor-and-conjugate-triviality-of-tilde-vartheta}. This
means $Tv_1=c\cdot v_1$ with a $c\in\Bbb C^{\times}$. Then
$$
 Tv_{\rho}=\rho\cdot(Tv_1)=c\cdot v_{\rho}
$$
for all $\rho\in\text{\rm Gal}(K/F)$, and hence $T$ is a homothety.
\end{proof}

\begin{rem}\label{remark:condition-for-irreducibility-of-induced-rep}
The proof of Proposition
\ref{prop:induced-representation-of-vartheta-is-irreducible} shows
that the induced representation 
$\text{\rm Ind}_{K^{\times}}^{W_{K/F}}\widetilde\vartheta$ is
irreducible if $\widetilde\vartheta$ is a character of $K^{\times}$
such that 
$\widetilde\vartheta(x^{\sigma})=\widetilde\vartheta(x)$ for all 
$x\in K^{\times}$ with $\sigma\in\text{\rm Gal}(K/F)$ implies
$\sigma=1$. 
\end{rem}


\section{Formal degree conjecture}
\label{sec:formal-degree-conjecture}
In this section, we will assume that $K/F$ is a 
tamely ramified Galois extension of degree $2n$ and
put $\Gamma=\text{\rm Gal}(K/F)$, We 
will keep the notations of the preceding sections.


\subsection{$\gamma$-factor of adjoint representation}
\label{subsec:gamma-factor-of-adjoint-representation}
The admissible representation of the Weil-Deligne group 
$W_F\times SL_2(\Bbb C)$ to $SO_{2n+1}(\Bbb C)$ corresponding to the
triple $(\varphi,SO_{2n+1}(\Bbb C),0)$ as explained in the appendix 
\ref{subsec:weil-deligne-group} is 
\begin{equation}
 W_F\times SL_2(\Bbb C)
 \xrightarrow{\text{\rm projection}} W_F
 \xrightarrow{\varphi} SO_{2n+1}(\Bbb C)
\label{eq:candidate-of-langlands-arthur-parameter}
\end{equation}
whcih is also denoted by $\varphi$. The purpose of this subsection is
to determine the $\gamma$-factor 
$\gamma(\varphi,\text{\rm Ad},\psi,d(x),s)$ whose definition and the
basic properties are presented in the appendix 
\ref{subsec:weil-deligne-group}. Our result is

\begin{thm}\label{th:gamma-factor-of-l-parameter-sp(2n)}
$$
 \gamma(\varphi,\text{\rm Ad},\psi,d(x),0)
 =w(\text{\rm Ad}\circ\varphi)\cdot q^{n^2r}\times
   \begin{cases}
    1&:\text{\rm $K/K_+$ is ramified},\\
    \frac 2{1+q^{-f_+}}&:\text{\rm $K/K_+$ is unramified}
   \end{cases}
$$
where $\psi$ is a continuous unitary additive character of $F$ such
that
$$
 \{x\in F\mid\psi(xO_F)=1\}=O_F
$$
and $d(x)$ is the Haar measure on $F$ such that 
$\int_{O_F}d(x)=1$.
\end{thm}

The rest of this subsection is devoted to the proof of the theorem.

Let us use the notation of \eqref{eq:generator-of-tame-galois-group}
$$
 \Gamma=\text{\rm Gal}(K/F)=\langle\delta,\rho\rangle,
$$
that is, $\text{\rm Gal}(K/K_0)=\langle\delta\rangle$ with the maximal
unramified subextension $K_0/F$ of $K/F$ and 
$\rho|_{K_0}\in\text{\rm Gal}(K_0/F)$ is the inverse of the Frobenius
automorphism. Put
$$
 \rho\delta\rho^{-1}=\delta^l,
 \quad
 \rho^f=\delta^m
 \quad
 (0\leq i,m<e,\;  ql\equiv 1\npmod{e}).
$$
By the canonical surjection
$$
 W_F\to W_F/\overline{[W_K,W_K]}=W_{K/F}
        =\text{\rm Gal}(K/F){\ltimes}_{\alpha_{K/F}}K^{\times}
        \subset\text{\rm Gal}(K^{\text{\rm ab}}/F),
$$
$I_F=\text{\rm Gal}(F^{\text{\rm alg}}/F^{\text{\rm ur}})\subset W_F$
     is mapped onto 
$$
 \text{\rm Gal}(K/K_0){\ltimes}_{\alpha_{K/F}}O_K^{\times}
 =\text{\rm Gal}(K^{\text{\rm ab}}/F^{\text{\rm ur}}).
$$

The representation space $V_{\vartheta}$ of 
$\varphi_1=\text{\rm Ind}_{K^{\times}}^{W_{K/F}}\widetilde\vartheta$ has
a $W_{K/F}$-invariant non-degenerate symmetric form
$$
 S_1(u,v)
 =\sum_{\gamma\in\Gamma}
   \widetilde\vartheta\left(\alpha_{K/F}(\gamma,\tau)\right)^{-1}
    \cdot u(\gamma)v(\gamma\tau)
 \quad
 (u,v\in V_{\vartheta})
$$
which is unique up to constant multiple, by Proposition 
\ref{prop:existence-of-nu-invarinat-nu-symmetric-from}. Put
$$
 u_{\sigma}
 =\widetilde\vartheta\left(\alpha_{K/F}(\sigma,\tau)\right)
  \cdot v_{\sigma}
 \in V_{\vartheta}
$$
for $\sigma\in\Gamma$. Then we have
\begin{equation}
 \alpha\cdot v_{\beta}
 =\widetilde\vartheta\left(\alpha_{K/F}(\alpha,\beta)\right)
  \cdot v_{\alpha\beta},
 \quad
 \alpha\cdot u_{\beta}
 =\widetilde\vartheta\left(\alpha_{K/F}(\alpha,\beta)\right)^{-1}
  \cdot u_{\alpha\beta}
\label{eq:action-of-galois-element-on-canonical-vector-of-induced-rep}
\end{equation}
and
$$
 S_1(v_{\alpha},v_{\beta})=S_1(u_{\alpha},u_{\beta})=0,
 \quad
 S_1(v_{\alpha},u_{\alpha})
 =\begin{cases}
   1&:\alpha=\beta,\\
   0&:\alpha\neq\beta
  \end{cases}
$$
for $\alpha,\beta\in\Gamma$. Fixing a representatives 
$\mathcal{S}$ of $\Gamma/\langle\tau\rangle$, we will 
identify the orthogonal group $O(V,S_1)$ of the symmetric form $S_1$ 
with the matrix group $O(S_1,\Bbb C)$ of 
\eqref{eq:structure-of-l-parameter-in-so(2n+1)} by means of the 
$\Bbb C$-basis $\{v_{\sigma},u_{\sigma}\}_{\sigma\in\mathcal{S}}$ of 
$V_{\vartheta}$ which we will call the canonical basis associated with
$\mathcal{S}$. Then we have
$$
 \varphi(x)=\begin{bmatrix}
             [x]&        & \\
                &[x]^{-1}& \\
                &        &1
            \end{bmatrix}\in SO_{2n+1}(\Bbb C)
 \;\;\text{\rm with}\;\;
 [x]=\text{\rm diag}
      (\widetilde\vartheta(x^{\sigma}))_{\sigma\in\mathcal{S}}
$$
for $x\in K^{\times}\subset W_{K/F}$ 
so that the centralizer of 
$\varphi(O_K^{\times})$ in $SO_{2n+1}(\Bbb C)$ is
\begin{equation}
 Z_{SO_{2n+1}(\Bbb C)}(\varphi(O_K^{\times}))
 =\left\{\begin{bmatrix}
          a&      & \\
           &a^{-1}& \\
           &      &1
         \end{bmatrix}\biggm|
   \text{\rm $a$=diagonal$\in GL_n(\Bbb C)$}\right\}
\label{eq:centralizer-of-varphi-unit-in-so(2n+1)}
\end{equation}
and the space $\widehat{\frak g}^{O_K^{\times}}$ of the 
$\text{\rm Ad}\circ\varphi(O_K^{\times})$-fixed vectors in
$$
 \widehat{\frak g}=\frak{so}_{2n+1}(\Bbb C)
 =\{X\in\frak{gl}_{2n+1}(\Bbb C)\mid XS+S\,^tX=0\}
$$
is
\begin{equation}
 \widehat{\frak g}^{O_K^{\times}}
 =\left\{\begin{bmatrix}
          A&  & \\
           &-A& \\
           &  &0
         \end{bmatrix}\biggm|
    \text{\rm $A$=diagonal$\in\frak{gl}_n(\Bbb C)$}\right\}
\label{eq:fixed-vector-of-ad-varphi-unit-li-lie-alg-so(2n+1)}
\end{equation}
by Proposition 
\ref{prop:conductor-and-conjugate-triviality-of-tilde-vartheta}.


Let us denote by $\mathcal{A}_{\varphi}$ the centralizer of 
$\text{\rm Im}(\varphi)$ in $SO_{2n+1}(\Bbb C)$. We have

\begin{prop}\label{prop:l-factor-of-adjoint-representation}
$$
 L(\varphi,\text{\rm Ad},s)
 =\begin{cases}
   1&:\text{\rm $K/K_+$ is ramified},\\
   \frac 1{1+q^{-f_+s}}&:\text{\rm $K/K_+$ is unramified}
  \end{cases}
$$
and 
$$
 \mathcal{A}_{\varphi}
 =\left\{\begin{bmatrix}
          \pm 1_{2n}&0\\
              0     &1
         \end{bmatrix}\right\}.
$$
\end{prop}
\begin{proof}
Assume that $K/K_+$ is ramified. Then $K/F$ is totally ramified and 
$\text{\rm Gal}(K/F)=\langle\delta\rangle$ a cyclic group of order
$2n$ with $\tau=\delta^n$. Put 
$\mathcal{S}=\{\delta^i\}_{0\leq i<n}$ which is a representatives of 
$\Gamma/\langle\tau\rangle$. Since $K/F$ is a cyclic extension, we
  have $\alpha_{K/F}(\alpha,\beta)\in F^{\times}$ for all 
$\alpha,\beta\in\Gamma$, the canonical basis associated with
  $\mathcal{S}$ is
$$
 v_i=v_{\delta^{i-1}},
 \quad
 u_i=v_{\delta^{n+i-1}}
 \quad
 (1\leq i\leq n).
$$
Then
\eqref{eq:action-of-galois-element-on-canonical-vector-of-induced-rep}
shows 
\begin{equation}
 \varphi(\delta)=\begin{bmatrix}
                     0    &1& 0\\
                  1_{2n-1}&0& 0\\
                     0    &0&-1
                 \end{bmatrix}\in SO_{2n+1}(\Bbb C).
\label{eq:matrix-representation-of-varphi-delta-in-ramified-case}
\end{equation}
Then the centralizer $\mathcal{A}_{\varphi}$ of 
$\text{\rm Im}(\varphi)$ in $SO_{2n+1}(\Bbb C)$ is
$$
 \mathcal{A}_{\varphi}=\left\{\begin{bmatrix}
                               \pm 1_{2n}&0\\
                                    0    &1
                              \end{bmatrix}\right\},
$$
and the space $\widehat{\frak g}^{I_F}$ of the 
$\text{\rm Ad}\circ\varphi(I_F)$-fixed vectors in $\widehat{\frak g}$
is $\{0\}$ so that we have $L(\varphi,\text{\rm Ad},s)=1$. 

Now assume that $K/K_+$ is unramified. Then 
$\tau=\delta^a\rho^{f_+}$ with $0\leq a<e$ by Proposition 
\ref{prop:order-two-element-in-tamely-ramified-galois-group}. 
Put
$\mathcal{S}=\{\delta^i\rho^j\}_{0\leq i<e, 0\leq j<f_+}$ which is a
representatives of $\Gamma/\langle\tau\rangle$. The associated basis
$$
 v_{ij}=v_{\delta^{i-1}\rho^{j-1}},
 \quad
 u_{ij}=u_{\delta^{i-1}\rho^{j-1}}
 \quad
 (1\leq i\leq e, 1\leq j\leq f_+)
$$
is ordered lexicographically. Then 
$$
 \varphi(\delta)=\begin{bmatrix}
                  \Delta&             & \\
                        &^t\Delta^{-1}& \\
                        &             &1
                 \end{bmatrix}\in SO_{2n+1}(\Bbb C)
 \;\text{\rm with}\;
 \Delta=\begin{bmatrix}
         \Delta_1&      &            \\
                 &\ddots&            \\
                 &      &\Delta_{f_+}
        \end{bmatrix},
$$
where
$$
 \Delta_j
 =\begin{bmatrix}
   0&0&0&\cdots&0&\alpha_{e,j}\\
   \alpha_{1,j}&0&0&\cdots&0&0\\
    &\alpha_{2,j}&0&\cdots&0&0\\
    & &\ddots&\ddots&\vdots&\vdots\\
    & &      &\ddots& 0    & 0    \\
    & &      &      &\alpha_{e-1,j}&0
  \end{bmatrix}
$$
with 
$\alpha_{i,j}=\widetilde\vartheta\left(
              \alpha_{K/F}(\delta,\delta^{i-1}\rho^{j-1})\right)$. 
The action of $\varphi(\delta)$ on 
\eqref{eq:fixed-vector-of-ad-varphi-unit-li-lie-alg-so(2n+1)} shows
that the space 
$\widehat{\frak g}^{I_F}$ of the 
$\text{\rm Ad}\circ\varphi(I_F)$-fixed vectors in $\widehat{\frak g}$
is
$$
 \widehat{\frak g}^{I_F}
 =\left\{\begin{bmatrix}
          A&  & \\
           &-A& \\
           &  &0
         \end{bmatrix}\biggm|
   A=\begin{bmatrix}
      a_11_e&      &      &    \\
            &a_21_e&      &    \\
            &      &\ddots&    \\
            &      &      &a_{f_+}1_e
     \end{bmatrix}\right\}.
$$
Since
$$
 \rho\cdot\delta^{i-1}\rho^{j-1}
 =\begin{cases}
   \delta^{l(i-1)}\rho^j=\delta^{i^{\prime}-1}\rho-j&:1\leq j<f_+,\\
   \delta^{l(i-1)-a}\tau=\delta^{i^{\prime\prime}-1}\tau&:j=f_+
  \end{cases}
 \qquad
 (1\leq i^{\prime}, i^{\prime\prime}\leq e)
$$
for $1\leq i\leq e$, let $[l]$ and $[l,a]$ be the permutation matrices of
the permutations
$$
 \begin{pmatrix}
  1&2&\cdots&e\\
  1^{\prime}&2^{\prime}&\cdots&e^{\prime}
 \end{pmatrix},
 \qquad
 \begin{pmatrix}
  1&2&\cdots&e\\
  1^{\prime\prime}&2^{\prime\prime}&\cdots&e^{\prime\prime}
 \end{pmatrix}
$$
respectively. Then we have
$$
 \varphi(\rho)=\begin{bmatrix}
                A&B&   0  \\
                C&D&   0  \\
                0&0&(-1)^e
               \end{bmatrix}\in SO_{2n+1}(\Bbb C)
$$
with
$$
 A=\begin{bmatrix}
    0&0&0&\cdots&0&0\\
    P_1&0&0&\cdots&0&0\\
     &P_2&0&\cdots&0&0\\
     & &\ddots&\ddots&\vdots&\vdots\\
     & & &\ddots&0&0\\
     & & & &P_{f_+-1}&0
   \end{bmatrix},
 \quad
 B=\begin{bmatrix}
    & & && && &Q_{f_+}\\
    & & && &&0&       \\
    & & &&\addots&& & \\
    & &0&& && & \\
    &0& && && & \\
   0& & && && & 
   \end{bmatrix},
$$
$$
 C=\begin{bmatrix}
    & & && && &P_{f_+}\\
    & & && &&0&       \\
    & & &&\addots&& & \\
    & &0&& && & \\
    &0& && && & \\
   0& & && && & 
   \end{bmatrix},
 \quad
 D=\begin{bmatrix}
    0&0&0&\cdots&0&0\\
    Q_1&0&0&\cdots&0&0\\
     &Q_2&0&\cdots&0&0\\
     & &\ddots&\ddots&\vdots&\vdots\\
     & & &\ddots&0&0\\
     & & & &Q_{f_+-1}&0
   \end{bmatrix}
$$
where
\begin{align*}
 P_j&=\begin{cases}
      \left[\begin{pmatrix}
       1&2&\cdots&e\\
       1^{\prime}&2^{\prime}&\cdots&e^{\prime}
      \end{pmatrix}\right]
      \begin{bmatrix}
       \alpha_{1j}&           &      &           \\
                  &\alpha_{2j}&      &           \\
                  &           &\ddots&           \\
                  &           &      &\alpha_{ej}
      \end{bmatrix}&:1\leq j<f_+,\\
      \left[\begin{pmatrix}
             1&2&\cdots&e\\
             1^{\prime\prime}&2^{\prime\prime}&\cdots&e^{\prime\prime}
            \end{pmatrix}\right]
      \begin{bmatrix}
       \beta_1&       &      &       \\
              &\beta_2&      &       \\
              &       &\ddots&       \\
              &       &      &\beta_e
      \end{bmatrix}&:j=f_+,
  \end{cases}\\
Q_j&=\begin{cases}
      \left[\begin{pmatrix}
             1&2&\cdots&e\\
             1^{\prime}&2^{\prime}&\cdots&e^{\prime}
            \end{pmatrix}\right]
      \begin{bmatrix}
       \alpha_{1j}^{-1}&                &      &           \\
                       &\alpha_{2j}^{-1}&      &           \\
                       &                &\ddots&           \\
                       &                &      &\alpha_{ej}^{-1}
      \end{bmatrix}&:1\leq j<f_+,\\
   \left[\begin{pmatrix}
          1&2&\cdots&e\\
          1^{\prime\prime}&2^{\prime\prime}&\cdots&e^{\prime\prime}
         \end{pmatrix}\right]
   \begin{bmatrix}
    \beta_1^{-1}&            &      &       \\
                &\beta_2^{-1}&      &       \\
                &            &\ddots&       \\
                &            &      &\beta_e^{-1}
   \end{bmatrix}&:j=f_+
  \end{cases}
\end{align*}
with 
\begin{align*}
 \alpha_{ij}
 &=\widetilde\vartheta\left(
   \alpha_{K/F}(\rho,\delta^{i-1}\rho^{j-1})\right),\\
 \beta_i
 &=\widetilde\vartheta\left(
   \alpha_{K/F}(\rho,\delta^{i-1}\rho^{f_+-1})\cdot
   \alpha_{K/F}(\delta^{i^{\prime\prime}-1},\tau)^{-1}\right).
\end{align*}
Then the adjoint action of $\varphi(\rho)$ on 
$\widehat{\frak g}^{I_F}$ gives
$$
 \det\left(
  1_{f_+}-t\cdot\text{\rm Ad}\circ\varphi(\rho)|_{\widehat{\frak g}^{I_F}}
            \right)
 =1+t^{-f_+}
$$
so that we have 
$L(\varphi,\text{\rm Ad},s)=\left(1+q^{-f_+s}\right)^{-1}$. Finally
the centralizer of $\text{\rm Im}(\varphi)$ in $SO_{2n+1}(\Bbb C)$ is 
$$
 \mathcal{A}_{\varphi}
 =\left\{\begin{bmatrix}
          \pm 1_{2n}&0\\
               0    &1
        \end{bmatrix}\right\}.
$$
\end{proof}

Next we will calculate the Artin conductor of 
$\text{\rm Ad}\circ\varphi$. Note that the complex vector space
$\widehat{\frak g}$ is isomorphic to the space of alternating matrices
$$
 \text{\rm Alt}_{2n+1}(\Bbb C)
 =\{X\in M_{2n+1}(\Bbb C)\mid X+\,^tX=0\}
$$
and $\text{\rm Ad}\circ\varphi$ on $\widehat{\frak g}$ is isomorphic
to ${\bigwedge}^2\varphi$ on $\text{\rm Alt}_{2n+1}(\Bbb C)$. Since 
$\varphi=\varphi_1\oplus\det\varphi_1$ and 
$\varphi_1=\text{\rm Ind}_{K^{\times}}^{W_{K/F}}\widetilde\vartheta$
with $(\det\varphi_1)|_{K^{\times}}=1$, we have
$$
 {\bigwedge}^2\varphi
 =({\bigwedge}^2\varphi_1)\oplus(\varphi_1\otimes\det\varphi_1)
 =({\bigwedge}^2\varphi_1)\oplus\varphi_1.
$$
Then 
$\chi_{\text{\rm Ad}\circ\varphi}
 =\chi_{{\wedge}^2\varphi_1}+\chi_{\varphi_1}$ and 
the character formula
$$
 \chi_{\varphi_1}(g)
 =\begin{cases}
   0&:\sigma\neq 1,\\
   \sum_{\gamma\in\Gamma}\widetilde\vartheta(x^{\gamma})
    &:\sigma=1
  \end{cases}
$$
for $g=(\sigma,x)\in W_{K/F}=\Gamma{\ltimes}_{\alpha_K/F}K^{\times}$ 
gives
\begin{align}
 \chi_{{\wedge}^2\varphi_1}(g)
 &=\frac 12\left\{\chi_{\varphi_1}(g)^2-\chi_{\varphi_1}(g^2)
                 \right\} \nonumber\\
 &=\begin{cases}
    0&:\sigma^2\neq 1,\\
    -\frac 12\sum_{\gamma\in\Gamma}
      \widetilde\vartheta\left(
       \alpha_{K/F}(\sigma,\sigma)^{\gamma}\cdot x^{(1+\sigma)\gamma}
                               \right)
     &:\sigma^2=1, \sigma\neq 1,\\
   \frac 12\sum_{\stackrel{\scriptstyle \alpha,\gamma\in\Gamma}
                          {\gamma\neq 1}}
            \widetilde\vartheta\left(x^{\alpha(1+\gamma)}\right)
     &:\sigma=1.
   \end{cases}
\label{eq:character-of-wedge-square-of-vaphi-1}
\end{align}
Now we have

\begin{prop}\label{prop:artin-conductor-of-adjoint-representation}
The Artin conductor of $\text{\rm Ad}\circ\varphi$ is
$$
 a(\text{\rm Ad}\circ\varphi)=2n^2r.
$$
\end{prop}
\begin{proof}
Let us denote by $K^{(k)}=K_{\varpi_K,k}$ ($k=1,2,\cdots$) the
field of $\varpi_K^k$-th division points of Lubin-Tate theory over
$K$. Then we have an isomorphism 
$$
 \delta_K:1+\frak{p}_K^k\,\tilde{\to}\,
  \text{\rm Gal}(K^{\text{\rm ab}}/K^{(k)}K^{\text{\rm ur}}).
$$
Because the character 
$\widetilde\vartheta:K^{\times}\to\Bbb C^{\times}$ comes from
a character of 
$$
 G_{\beta}(O_F/\frak{p}^r)\subset\left(O_K/\frak{p}_K^{er}\right)^{\times},
$$
$\varphi$ is trivial on 
$\text{\rm Gal}(K^{\text{\rm ab}}/K^{(er)}K^{\text{\rm ur}})$. 
Note that 
$K^{(er)}K^{\text{\rm ur}}=K^{(er)}F^{\text{\rm ur}}$ 
is a finite extension of $F^{\text{\rm ur}}$. If us use the upper
numbering 
$$
 V^s=V_t(K^{(er)}F^{\text{\rm ur}}/F^{\text{\rm ur}})
$$
of the higher ramification group, where $t\mapsto s$
is the inverse of Hasse function whose graph is
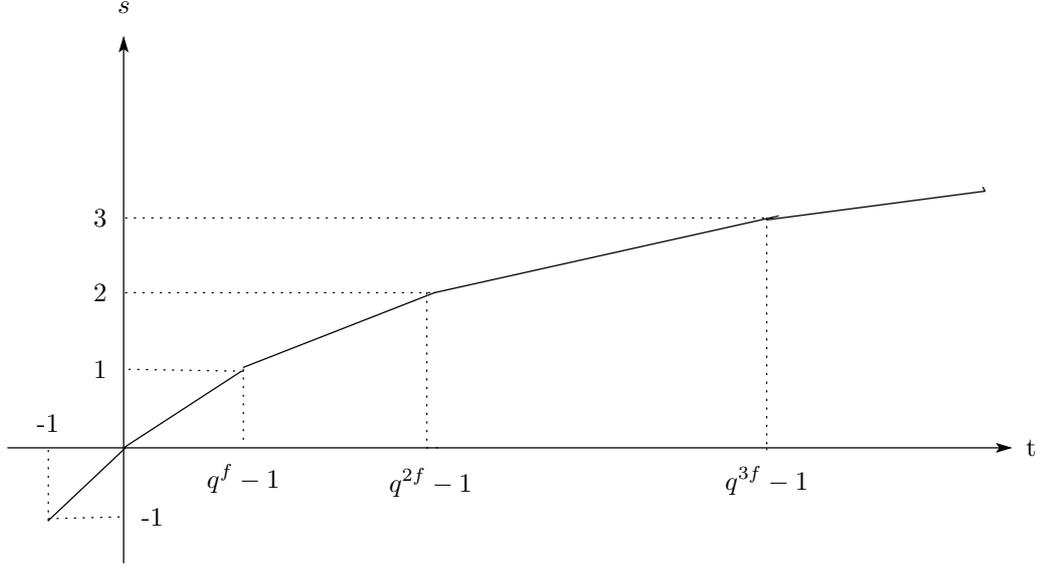
\begin{center}
\unitlength 0.1in
\begin{picture}( 52.8000, 29.6500)( 14.0000,-42.1000)
%
\special{pn 8}%
\special{pa 2000 3610}%
\special{pa 2610 3210}%
\special{fp}%
%
\special{pn 8}%
\special{pa 2620 3190}%
\special{pa 3610 2800}%
\special{fp}%
\special{pa 3610 2800}%
\special{pa 3610 2800}%
\special{fp}%
%
\special{pn 8}%
\special{pa 3610 2800}%
\special{pa 5390 2400}%
\special{fp}%
%
\special{pn 8}%
\special{pa 2010 3610}%
\special{pa 1610 3990}%
\special{fp}%
%
\special{pn 8}%
\special{pa 1610 3990}%
\special{pa 1610 3610}%
\special{dt 0.045}%
%
\special{pn 8}%
\special{pa 1610 3980}%
\special{pa 2010 3970}%
\special{dt 0.045}%
%
\special{pn 8}%
\special{pa 2620 3200}%
\special{pa 2620 3600}%
\special{dt 0.045}%
%
\special{pn 8}%
\special{pa 2620 3210}%
\special{pa 2010 3200}%
\special{dt 0.045}%
%
\special{pn 8}%
\special{pa 2010 2800}%
\special{pa 3620 2800}%
\special{dt 0.045}%
\special{pa 3620 3610}%
\special{pa 3620 3610}%
\special{dt 0.045}%
%
\special{pn 8}%
\special{pa 2010 2410}%
\special{pa 5400 2410}%
\special{dt 0.045}%
\special{pa 5330 2410}%
\special{pa 5330 3610}%
\special{dt 0.045}%
\special{pa 5330 3620}%
\special{pa 5330 3580}%
\special{dt 0.045}%
%
\special{pn 8}%
\special{pa 2000 4210}%
\special{pa 2000 1470}%
\special{fp}%
\special{sh 1}%
\special{pa 2000 1470}%
\special{pa 1980 1538}%
\special{pa 2000 1524}%
\special{pa 2020 1538}%
\special{pa 2000 1470}%
\special{fp}%
\put(18.8000,-32.0000){\makebox(0,0){1}}%
\put(18.8000,-28.0000){\makebox(0,0){2}}%
\put(18.8000,-24.1000){\makebox(0,0){3}}%
\put(21.5000,-39.7000){\makebox(0,0){-1}}%
\put(16.1000,-34.8000){\makebox(0,0){-1}}%
\put(26.2000,-37.7000){\makebox(0,0){$q^f-1$}}%
\put(35.9000,-37.9000){\makebox(0,0){$q^{2f}-1$}}%
\put(53.3000,-37.8000){\makebox(0,0){$q^{3f}-1$}}%
\put(20.0000,-13.1000){\makebox(0,0){$s$}}%
%
\special{pn 8}%
\special{pa 1400 3610}%
\special{pa 6590 3610}%
\special{fp}%
\special{sh 1}%
\special{pa 6590 3610}%
\special{pa 6524 3590}%
\special{pa 6538 3610}%
\special{pa 6524 3630}%
\special{pa 6590 3610}%
\special{fp}%
%
\special{pn 8}%
\special{pa 5330 2420}%
\special{pa 6460 2270}%
\special{fp}%
\special{pa 6460 2270}%
\special{pa 6450 2250}%
\special{fp}%
\put(67.0000,-36.1000){\makebox(0,0){t}}%
%
\special{pn 8}%
\special{pa 3570 2800}%
\special{pa 3570 3610}%
\special{dt 0.045}%
\special{pa 3570 3610}%
\special{pa 3570 3610}%
\special{dt 0.045}%
\end{picture}%

\end{center}
then $\delta_K$ induces the isomorphism
$$
 (1+\frak{p}_K^k)/(1+\frak{p}_K^{er})\,\tilde{\to}\,
 \text{\rm Gal}(K^{(er)}K^{\text{\rm ur}}/K^{(k)}K^{\text{\rm ur}})
 =V^s
$$
for $k-1<s\leq k$ ($k=1,2,\cdots$), and hence, for 
$V_t=V_t(K^{(er)}F^{\text{\rm ur}}/F^{\text{\rm ur}})$, we have
$$
 |V_t|=\begin{cases}
        e\cdot q^{nr}(1-q^{-f})&:t=0,\\
        q^{nr-fk}              &:q^{f(k-1)}-1<t\leq q^{fk}-1.
       \end{cases}
$$
By the definition
$$
 a(\text{\rm Ad}\circ\varphi)
 =\sum_{t=0}^{\infty}
   \left(\dim_{\Bbb C}\widehat{\frak g}
         -\dim_{\Bbb C}\widehat{\frak g}^{V_t}\right)\cdot
   (V_0:V_t)^{-1}.
$$
We have
$$
 \dim_{\Bbb C}\widehat{\frak g}^{V_0}
 =\dim_{\Bbb C}\widehat{\frak g}^{I_F}
 =\begin{cases}
   0&:\text{\rm $K/K_+$ is ramified},\\
   f_+&:\text{\rm $K/K_+$ is unramified}
  \end{cases}
$$
as shown in the proof of Proposition 
\ref{prop:l-factor-of-adjoint-representation}. 
For $q^{f(k-1)}-1<t\leq q^{fk}-1$ with $k>0$, we have
\begin{align*}
 |V_t|\cdot\dim_{\Bbb C}\widehat{\frak g}^{V_t}
 &=\sum_{g\in V_t}\chi_{\text{\rm Ad}\circ\varphi}(g)
  =\frac 12\sum_{\dot x\in V_t}
           \sum_{\stackrel{\scriptstyle \alpha,\gamma\in\Gamma}
                          {\gamma\neq 1}}
            \widetilde\vartheta(x^{\alpha(1+\gamma)})
   +\sum_{\dot x\in V_t}\widetilde\vartheta(x^{\gamma})\\
 &=n\cdot\sum_{\dot x\in V_t}\sum_{\tau\neq\gamma\in\Gamma}
          \widetilde\vartheta(x^{1-\gamma})
   +2n\cdot\sum_{\dot x\in V_t}\widetilde\vartheta(x)
\end{align*}
where $V_t$ is identified with 
$(1+\frak{p}_K^k)/(1+\frak{p}_K^{er})$. 

If $K/K_+$ is unramified, then $\tau\not\in\text{\rm Gal}(K/K_0)$, and
Proposition
\ref{prop:conductor-and-conjugate-triviality-of-tilde-vartheta} gives
$$
 \sum_{\dot x\in V_t}\sum_{\tau\neq\gamma\in\Gamma}
  \widetilde\vartheta(x^{1-\gamma})
 =|V_t|\times\begin{cases}
              2n-1&:k>e(r-1),\\
                e &:k=e(r-1),\\
                1 &:k<e(r-1)
             \end{cases}
$$
and
$$
 \sum_{\dot x\in V_t}\widetilde\vartheta(x)
 =\begin{cases}
   |V_t|&:k>e(r-1),\\
     0  &:k\leq e(r-1).
  \end{cases}
$$
So we have
$$
 \dim_{\Bbb C}\widehat{\frak g}^{V_t}
 =\begin{cases}
   n(2n+1)&:k>e(r-1),\\
   ne     &:k=e(r-1),\\
   n      &:k<e(r-1).
  \end{cases}
$$
Then we have
\begin{align*}
 a(\text{\rm Ad}\circ\varphi)
 &=n(2n+1)-f_+
  +\{n(2n+1)-n\}\cdot e^{-1}\cdot\{e(r-1)-1\}\\
 &\hphantom{=n(2n+1)-\;f_+}
  +\{n(2n+1)-ne\}\cdot e^{-1}\\
 &=2n^2r.
\end{align*}
If $K/K_+$ is ramified, then $K/F$ is totally ramified and we have
$$
 \sum_{\dot x\in V_t}\sum_{\tau\neq\gamma\in\Gamma}
  \widetilde\vartheta(x^{1-\gamma})
 =|V_t|\times\begin{cases}
              2n-1&:k\geq e(r-1),\\
                1 &:k<e(r-1)
             \end{cases}
$$
and
$$
 \sum_{\dot x\in V_t}\widetilde\vartheta(x)
 =\begin{cases}
   |V_t|&:k\geq e(r-1),\\
     0  &:k<e(r-1)
  \end{cases}
$$
by Proposition
\ref{prop:conductor-and-conjugate-triviality-of-tilde-vartheta}. 
The we have
\begin{align*}
 a(\text{\rm Ad}\circ\varphi)
 &=n(2n+1)
  +\{n(2n+1)-n\}\cdot(2n)^{-1}\cdot\{2n(r-1)-1\}\\
 &=2n^2r.
\end{align*}
\end{proof}

Since
$$
 \gamma(\varphi,\text{\rm Ad},\psi,d(x),0)
 =\varepsilon(\varphi,\text{\rm Ad},\psi,d(x))\cdot
  \frac{L(\varphi,\text{\rm Ad},1)}
       {L(\varphi,\text{\rm Ad},0)}
$$
and
$$
 \varepsilon(\varphi,\text{\rm Ad},\psi,d(x))
 =w(\text{\rm Ad}\circ\varphi)\cdot q^{a(\text{\rm Ad}\circ\varphi)/2},
$$
Proposition \ref{prop:l-factor-of-adjoint-representation} and 
Proposition \ref{prop:artin-conductor-of-adjoint-representation} give
the proof of Theorem \ref{th:gamma-factor-of-l-parameter-sp(2n)}.

\subsection{$\gamma$-factor of principal parameter}
\label{subsec:gamma-factor-of-principal-parameter}
Let $\text{\rm Sym}_{2n}$ be the symmetric tensor representation of
$SL_2(\Bbb C)$ on the space $\mathcal{P}_{2n}$ of the complex
coefficient homogeneous polynomials of $X, Y$ of degree $2n$. Then
$$
 \langle f,g\rangle
 =\left.
  f\left(-\frac{\partial}
               {\partial Y},\frac{\partial}
                                 {\partial X}\right)g(X,Y)
  \right|_{(X,Y)=(0,0)}
 \quad
 (f,g\in\mathcal{P}_{2n})
$$
defines a $SL_2(\Bbb C)$-invariant non-degenerate symmetric complex
bilinear form on the complex vector space $\mathcal{P}_{2n}$. 
For the $\Bbb C$-basis 
$\left\{
  v_k=\frac 1{(k-1)!}X^{2n+1-k}Y^{k-1}\right\}_{k=1,2,\cdots,2n+1}$ of
$\mathcal{P}_{2n}$, we have
$$
 \langle v_k,v_l\rangle
 =\begin{cases}
   0&:k+l=2n+2,\\
   (-1)^{k-1}=(-1)^{l-1}&:k+l=2n+2
  \end{cases}
$$
and the identification 
$$
 SO(\mathcal{P}_{2n},\langle\,,\rangle)=SO_{2n+1}(\Bbb C)
 =\{g\in SL_{2n+1}(\Bbb C)\mid gJ_{2n+1}\,^tg=J_{2n+1}\}
$$
where
$$
 J_{2n+1}=\begin{bmatrix}
            &  &       &  &1\\
            &  &       &-1& \\
            &  &\addots&  & \\
            &-1&       &  & \\
           1&  &       &  &
          \end{bmatrix}.
$$
The Lie algebra of $SO_{2n+1}(\Bbb C)$ is
$$
 \frak{so}_{2n+1}(\Bbb C)
=\{X\in\frak{gl}_{2n+1}(\Bbb C)\mid XJ_{2n+1}+J_{2n+1}\,^tX=0\}
$$
and
$$
 d\,\text{\rm Sym}_{2n}\begin{bmatrix}
                        0&1\\
                        0&0
                       \end{bmatrix}=N_0
 =\begin{bmatrix}
   0&1& &      & \\
    &0&1&      & \\
    & &\ddots&\ddots& \\
    & & &0&1\\
    & & &      &0
   \end{bmatrix}\in\widehat{\frak g}
$$
is the nilpotent element in $\frak{so}_{2n+1}(\Bbb C)$
associated with the standard {\it \'epinglage} of the standard root
system of $\frak{so}_{2n+1}(\Bbb C)$. Then 
\begin{equation}
 \varphi_0:W_F\times SL_2(\Bbb C)
           \xrightarrow{\text{\rm proj.}}SL_2(\Bbb C)
           \xrightarrow{\text{\rm Sym}_{2n}}SO_{2n+1}(\Bbb C)
\label{eq:candidate-of-principal-parameter}
\end{equation}
is a representation of Weil-Deligne group with the associated triplet 
$(\rho_0,SO_{2n+1}(\Bbb C),N_0)$ such that
$\rho_0|_{I_F}$ is trivial and
$$
 \rho_0(\widetilde{\text{\rm Fr}})
 =\begin{bmatrix}
    q^{-n}&            &      &           &           \\
                &q^{-(n-1)}&      &           &           \\
                &            &\ddots&           &           \\
                &            &      &q^{(n-1)}&           \\
                &            &      &           &q^{n}
   \end{bmatrix}
 \in SO_{2n+1}(\Bbb C).
$$
Now
\begin{equation}
 \{N_0^{2k-1}\mid k=1,2,\cdots,n\}
\label{eq:basis-of-kernel-of-ad(n0)-for-so(2n+1)}
\end{equation}
is a $\Bbb C$-basis of
$$
 \widehat{\frak g}_{N_0}
 =\{X\in\widehat{\frak g}\mid[X,N_0]=0\}.
$$
The representation matrix of 
$\text{\rm Ad}\circ\rho_0(\widetilde{\text{\rm Fr}})
 \in GL_{\Bbb C}(\widehat{\frak g})$ is
$$
 \begin{bmatrix}
  q^{-1}&      &      &           \\
        &q^{-3}&      &           \\
        &      &\ddots&           \\
        &      &      &q^{-(2n-1)}
 \end{bmatrix}
$$
so that we have
\begin{align*}
 L(\varphi_0,\text{\rm Ad},s)
 &=\det\left(1-q^{-s}\cdot
    \text{\rm Ad}\circ\rho_0(
     \widetilde{\text{\rm Fr}})|_{\widehat{\frak g}_{N_0}}
            \right)^{-1}\\
 &=\prod_{k=1}^n\left(1-q^{-(s+2k-1)}\right)^{-1}.
\end{align*}
On the other hand \cite[p.448]{Gross-Reeder2010} shows 
$$
 \varepsilon(\varphi_0,\text{\rm Ad},\psi,d(x))=q^{n^2}.
$$
Since the symmetric tensor representation $\text{\rm Sym}_{2n}$ is
self-dual, we have
\begin{align}
 \gamma(\varphi_0,\text{\rm Ad},\psi,d(x),0)
 &=\varepsilon(\varphi_0,\text{\rm Ad},\psi,d(x))\cdot
   \frac{L(\varphi_0,\text{\rm Ad},1)}
        {L(\varphi_0,\text{\rm Ad},0)} \nonumber\\
 &=q^{n^2}\cdot\prod_{k=1}^n\frac{1-q^{-(2k-1)}}
                                 {1-q^{-2k}}.
\label{eq:gamma-factor-of-principal-parameter-sp(2n)}
\end{align}

\subsection{Verification of formal degree conjecture}
\label{subse:verification-of-formal-degree-conjecture}
Let $d_{G(F)}$ be the Haar measure on $G(F)$ such that 
$\int_{G(O_F)}d_{G(F)}(x)=1$. Then the 
Euler-Poincar\'e measure $\mu_{G(F)}$ on $G(F)=Sp_{2n}(F)$ is  
(see \cite[p.150, Th.7]{Serre1971})
$$
 d\mu_{G(F)}(x)
 =(-1)^nq^{n^2}\prod_{k=1}^n\left(1-q^{-(2k-1)}\right)\cdot
   d_{G(F)}(x).
$$
Then Theorem \ref{th:supercuspidal-representation-of-sp(2n)} implies
that the formal degree of the supercuspidal representation 
$\pi_{\beta,\theta}
 =\text{\rm ind}_{G(O_F)}^{G(F)}\delta_{\beta,\theta}$ with respect to
 the absolute value of the Euler-Poincar\'e measure on $G(F)$ is
\begin{equation}
 q^{n^2(r-1)}\cdot\prod_{k=1}^n\frac{1-q^{-2k}}
                                    {1-q^{-(2k-1)}}\times
 \begin{cases}
  \frac 12&:\text{\rm $K/K_+$ is ramified},\\
  \frac 1{1+q^{-f_+}}&:\text{\rm $K/K_+$ is unramified}.
 \end{cases}
\label{eq:formal-degree-of-cuspidal-rep-of-sp(2n)-wrt-euler-poincare}
\end{equation}
Since the order of the centralizer $\mathcal{A}_{\varphi}$ of 
$\text{\rm Im}(\varphi)$ in $SO_{2n+1}(\Bbb C)$ is two 
(Proposition \ref{prop:l-factor-of-adjoint-representation}), 
Theorem \ref{th:gamma-factor-of-l-parameter-sp(2n)} and 
\eqref{eq:gamma-factor-of-principal-parameter-sp(2n)} gives the
following

\begin{thm}\label{th:formal-degree-conjecture-for-sp(2n)}
The formal degree of the supercuspidal representation 
$\pi_{\beta,\theta}
 =\text{\rm ind}_{G(O_F)}^{G(F)}\delta_{\beta,\theta}$ with respect to
 the absolute value of the Euker-Poincar\'e measure on $G(F)$ is
$$
 \frac 1{|\mathcal{A}_{\varphi}|}\cdot
 \left|\frac{\gamma(\varphi,\text{\rm Ad},\psi,d(x),0)}
            {\gamma(\varphi_0,\text{\rm Ad},\psi,d(x),0)}\right|.
$$
\end{thm}

Since $\mathcal{A}_{\varphi}$ is a finite abelian group, all the
irreducible representation of $\mathcal{A}_{\varphi}$ is
one-dimensional. So Theorem 
\ref{th:formal-degree-conjecture-for-sp(2n)} says that the formal
degree conjecture is valid if we consider 
\eqref{eq:candidate-of-langlands-arthur-parameter} as the
Arthur-Langlands parameter of the supercuspidal representation
$\pi_{\beta,\theta}$ and 
\eqref{eq:candidate-of-principal-parameter} as the principal parameter
of $G(F)=Sp_{2n}(F)$.

\section{Root number conjecture}
\label{sec:root-number-conjecture}
In this section, we will assume that $K/F$ is a 
tamely ramified Galois extension of degree $2n$ and
put $\Gamma=\text{\rm Gal}(K/F)$, We 
will keep the notations of the preceding sections.

\subsection{Structure of adjoint representation}
\label{subsec:structure-of-adjoint-representation}
We will identify the representations of $W_{K/F}$ with the
representations of $W_F$ which factor through the canonical
surjection 
$$
 W_F\to W_F/\overline{[W_K.W_K]}=W_{K/F}.
$$
We will also regard a representation of $\Gamma$ as the representation
of $W_{K/F}$ via the projection $W_{K/F}\to\Gamma$.

As we have seen in the subsection 
\ref{subsec:gamma-factor-of-adjoint-representation}
\begin{equation}
 \text{\rm $\text{\rm Ad}\circ\varphi$ on 
           $\widehat{\frak g}$}
 ={\bigwedge}^2\varphi
 ={\bigwedge}^2\varphi_1\oplus\varphi_1
\label{eq:decomposition-of-ad-varphi-into-wedge-square-and-varphi-1}
\end{equation}
with 
$\varphi_1
 =\text{\rm Ind}_{K^{\times}}^{W_{K/F}}\widetilde\vartheta$. Now we
 have

\begin{thm}\label{th:structure-of-wedge-square-of-varphi-1}
$$
 {\bigwedge}^2\varphi_1
 =\bigoplus_{\pi(\tau)\neq 1}\pi^{\dim\pi}
  \oplus
  \bigoplus_{\{\gamma\neq\gamma^{-1}\}\subset\Gamma}
   \text{\rm Ind}_{K^{\times}}^{W_{K/F}}
    \widetilde\vartheta_{\gamma}
  \oplus
  \bigoplus_{\stackrel{\scriptstyle 1,\tau\neq\gamma\in\Gamma}
                      {\gamma^2=1}}
   \text{\rm Ind}_{W_{K/K_{\gamma}}}^{W_{K/F}}\chi_{\gamma}.
$$
\end{thm}

Here $\bigoplus_{\pi(\tau)\neq 1}$ denotes the direct sum over the
equivalence classes $\pi$ of the irreducible representations of
$\Gamma$ such that $\pi(\tau)\neq 1$. The direct sum 
$\bigoplus_{\{\gamma\neq\gamma^{-1}\}\subset\Gamma}$ is over the
subsets $\{\gamma,\gamma^{-1}\}\subset\Gamma$ such that
$\gamma^2\neq 1$, and 
$\widetilde\vartheta_{\gamma}(x)
 =\widetilde\vartheta(x^{1+\gamma})$ ($x\in K^{\times}$). 
For a $\gamma\in\Gamma$ of order two, the unitary
character $\chi_{\gamma}$ of $W_{K/K_{\gamma}}$ is defined by 
\begin{align*}
 \chi_{\gamma}:W_{K/K_{\gamma}}
               =W_{K_{\gamma}}/\overline{[W_K,W_K]}
  &\xrightarrow[\;\text{\rm can.}\;\;]{}
    W_{K_{\gamma}}/\overline{[W_{K_{\gamma}},W_{K_{\gamma}}]}\\
  &\xrightarrow[\delta_{K_{\gamma}}^{-1}]{\sim}
   K_{\gamma}^{\times}
  \xrightarrow[(\ast,K/K_{\gamma})\cdot\widetilde\vartheta]{}
  \Bbb C^{\times}
\end{align*}
with the subfield  $F\subset K_{\gamma}\subset K$ such that 
$\text{\rm Gal}(K/K_{\gamma})=\langle\gamma\rangle$ and
$$
 (x,K/K_{\gamma})
 =\begin{cases}
   1&:x\in N_{K/K_{\gamma}}(K^{\times}),\\
  -1&:x\not\in N_{K/K_{\gamma}}(K^{\times}).
  \end{cases}
$$
The rest of this subsection is devoted to the proof of the theorem.

\numsubsubsection{}
\label{subsubsec:character-of-pi-gamma-of-order-two-not-tau}
Take a $\gamma\in\Gamma$ of order two. Note that the group
homomorphism $x\mapsto\left(x,K/K_{\gamma}\right)$ induces the
inverse of the isomorphism
$$
 \text{\rm Gal}(K/K_{\gamma})\,\tilde{\to}\,
 K_{\gamma}^{\times}/N_{K/K_{\gamma}}(K^{\times})
 \quad
 (\sigma\mapsto\alpha_{K/F}(\sigma,\gamma))
$$
if we identify $\text{\rm Gal}(K/K_{\gamma})$ with $\{\pm 1\}$. 
Then the commutative diagram 
$$
\begin{tikzcd}
  W_K/\overline{[W_K,W_K]}
    \arrow[d, equal]\arrow{r}{\text{\rm can.}}
   &W_{K_{\gamma}}/\overline{[W_K,W_K]}
            \arrow{r}{\text{\rm can.}}\arrow{d}{\text{\rm can.}}
    &W_{K_{\gamma}}/W_K\arrow{d}{\text{\rm res.}}[swap]{\wr}
     \\
  W_K/\overline{[W_K,W_K]}\arrow{r}{\text{\rm can.}}
                          \arrow{d}{\delta_K^{-1}}[swap]{\wr}
   &W_{K_{\gamma}}/\overline{[W_{K_{\gamma}},W_{K_{\gamma}}]}
                          \arrow{r}{\text{\rm res.}}
                          \arrow{d}{\delta_{K_{\gamma}}^{-1}}[swap]{\wr}
    &\text{\rm Gal}(K/K_{\gamma})
     \arrow{d}{[\sigma\mapsto\alpha_{K/K_{\gamma}}(\sigma,\gamma)]}[swap]{\wr}
     \\
  K^{\times}\arrow{r}[swap]{N_{K/K_{\gamma}}}
   &K_{\gamma}^{\times}\arrow{r}[swap]{\text{\rm can.}}
    &K_{\gamma}^{\times}/N_{K/K_{\gamma}}(K^{\times})
\end{tikzcd}
$$
implies that we have 
$$
 \chi_{\gamma}(\sigma,x)
 =\text{\rm sign}(\sigma)\cdot
  \widetilde\vartheta\left(
   \alpha_{K/F}(\sigma,\gamma)\cdot x^{1+\gamma}\right)
$$
for 
$(\gamma,x)\in W_{K/K_{\gamma}}
              =\text{\rm Gal}(K/K_{\gamma})
                {\ltimes}_{\alpha_{K/F}}K^{\times}
              \subset W_{K/F}$
where
$$
 \text{\rm sign}(\sigma)
 =\begin{cases}
   1&:\sigma=1,\\
  -1&:\sigma=\gamma.
  \end{cases}
$$
By means of the cocycle relation of $\alpha_{K/F}$, we have 
$$
 \chi_{\gamma}(\alpha^{-1}(\sigma,x)\alpha)
 =\begin{cases}
   \widetilde\vartheta(x^{\alpha(1+\gamma)})&:\sigma=1,\\
  -\widetilde\vartheta\left(
    \alpha_{K/F}(\gamma,\gamma)^{\alpha}x^{\alpha(1+\gamma)}
                      \right)&:\sigma=\gamma
  \end{cases}
$$
for any $(\sigma,x)\in W_{K/K_{\gamma}}$ and $\alpha\in\sigma$. 
Note that the elements of $\Gamma$ of order two are central as shown
by Proposition
\ref{prop:order-two-element-in-tamely-ramified-galois-group}. The the
character of the induced representation 
$\pi_{\gamma}
 =\text{\rm Ind}_{W_{K/K_{\gamma}}}^{W_{K/F}}\chi_{\gamma}$ is
\begin{align}
 \chi_{\pi_{\gamma}}(\sigma,x)
 &=\begin{cases}
    0&:\sigma\not\in\text{\rm Gal}(K/K_{\gamma}),\\
   \sum_{\dot\alpha\in\Gamma/\langle\gamma\rangle}
    \chi_{\gamma}\left(\alpha^{-1}(\sigma,x)\alpha\right)
     &:\sigma\in\text{\rm Gal}(K/K_{\gamma})\\
   \end{cases} \nonumber\\
 &=\begin{cases}
    0&:\sigma\neq 1,\gamma,\\
    \sum_{\dot\alpha\in\Gamma/\langle\gamma\rangle}
     \widetilde\vartheta\left(x^{\alpha(1+\gamma)}\right)
      &:\sigma=1,\\
    -\sum_{\dot\alpha\in\Gamma/\langle\gamma\rangle}
      \widetilde\vartheta\left(
       \alpha_{K/F}(\gamma,\gamma)^{\alpha}
        x^{\alpha(1+\gamma)}\right)
      &:\sigma=\gamma.
    \end{cases}
\label{eq:chracter-of-ind-chi-sigma-of-order-two-element}
\end{align}

\numsubsubsection{}
\label{subsubsec:character-of-partial-regular-rep-of-galois-group}
Since $\tau\in\Gamma$ is a central element, 
the character of the induced representation 
$R_{\tau}=\text{\rm Ind}_{\langle\tau\rangle}^{\Gamma}
           \text{\bf 1}_{\langle\tau\rangle}$
is
$$
 \chi_{R_{\tau}}(\sigma)
 =\begin{cases}
   0&:\sigma\neq 1,\tau,\\
  (\Gamma:\langle\tau\rangle)=n&:\sigma=1,\tau.
  \end{cases}
$$
For an irreducible representation $\pi$ of $\Gamma$, we have 
$\pi(\tau)=\pm\text{\bf 1}$, and we have
$$
 \langle\pi,R_{\tau}\rangle
 =|\Gamma|^{-1}(n\cdot\chi_{\pi}(1)+n\cdot\chi_{\pi}(\tau))
 =\begin{cases}
   \dim\pi&:\pi(\tau)=\text{\bf 1},\\
      0   &:\pi(\tau)=-\text{\bf 1}.
  \end{cases}
$$
Hence $R_{\tau}\oplus\bigoplus_{\pi(\tau)\neq 1}\pi^{\dim\pi}$ 
is the regular representation 
$R_{\Gamma}=\text{\rm Ind}_{\{1\}}^{\Gamma}\text{\bf 1}$, and we have
\begin{equation}
 \left(\chi_{R_{\Gamma}}-\chi_{R_{\tau}}\right)(\sigma)
 =\begin{cases}
   n&:\sigma=1,\\
  -n&:\sigma=\tau,\\
   0&:\sigma\neq 1,\tau.
  \end{cases}
\label{eq:character-of-partial-regular-reppresentation}
\end{equation}

\numsubsubsection{}
\label{subsubsec:character-of-wedge-square-varphi-1}
Recall the character formula 
\eqref{eq:character-of-wedge-square-of-vaphi-1}. 
Since $\widetilde\vartheta(x)=\vartheta(x^{1-\tau})$ 
($x\in K^{\times}$) and the elements of $\Gamma$ of order two are
central (Proposition
\ref{prop:order-two-element-in-tamely-ramified-galois-group}),  we have
$$
 \frac 12\sum_{\alpha\in\Gamma}
          \widetilde\vartheta(x^{\alpha(1+\tau)})
 =\frac 12|\Gamma|=n.
$$
Since
$$
 \sum_{\alpha\in\Gamma}\widetilde\vartheta\left(
                              x^{\alpha(1+\gamma^{-1})}\right)
 =\sum_{\alpha\in\Gamma}\widetilde\vartheta(
                              x^{\alpha(1+\gamma)})
$$
for any $\gamma$, we have
\begin{align*}
 \chi_{{\wedge}^2\varphi_1}(1,x)
 &=\frac 12\sum_{\stackrel{\scriptstyle \alpha,\gamma\in\Gamma}
                          {\gamma\neq 1}}
            \widetilde\vartheta(x^{\alpha(1+\gamma)})\\
 &=n+\sum_{\{\gamma\neq\gamma^{-1}\}\subset\Gamma}
     \sum_{\alpha\in\Gamma}
      \widetilde\vartheta(x^{\alpha(1+\gamma)})
    +\sum_{\stackrel{\scriptstyle 1,\tau\neq\gamma\in\Gamma}
                    {\gamma^2=1}}
     \sum_{\dot\alpha\in\Gamma/\langle\gamma\rangle}
      \widetilde\vartheta(x^{\alpha(1+\gamma)}).
\end{align*}
Take a $\sigma\in\Gamma$ of order two. Then the cocycle relation of
$\alpha_{K/F}$ gives 
$\alpha_{K/F}(\sigma,\sigma)^{\sigma}=\alpha_{K/F}(\sigma,\sigma)$. Since
$\sigma$ is a central element of $\Gamma$, we have 
$$
 \alpha_{K/F}(\sigma,\sigma)^{\sigma\alpha}
 =\alpha_{K/F}(\sigma,\sigma)^{\alpha}
$$
for any $\alpha\in\Gamma$. If $\sigma=\tau$, we have
$$
 \chi_{{\wedge}^2\varphi_1}(\tau,x)
 =-\frac 12\sum_{\alpha\in\Gamma}
            \widetilde\vartheta(x^{\alpha(1+\tau)})
 =-\frac 12|\Gamma|
 =-n.
$$
If $\sigma\neq\tau$, then we have
$$
 \chi_{{\wedge}^2\varphi_1}(\sigma,x)
 =-\sum_{\dot\alpha\in\Gamma/\langle\sigma\rangle}
    \widetilde\vartheta\left(
     \alpha_{K/F}(\sigma,\sigma)^{\alpha}\cdot
      x^{\alpha(1+\sigma)}\right).
$$
Since the character of the induced representation 
$\rho_{\gamma}
 =\text{\rm Ind}_{K^{\times}}^{W_{K/F}}\widetilde\vartheta_{\gamma}$
 for $\gamma\in\Gamma$ such that $\gamma^2\neq 1$ is
$$
 \chi_{\rho_{\gamma}}(\sigma,x)
 =\begin{cases}
   0&:\sigma\neq 1,\\
   \sum_{\alpha\in\Gamma}\widetilde\vartheta(x^{\alpha(1+\gamma)})
    &:\sigma=1,
  \end{cases}
$$
the formulae 
\eqref{eq:chracter-of-ind-chi-sigma-of-order-two-element} and 
\eqref{eq:character-of-partial-regular-reppresentation} gives
$$
 \chi_{{\wedge}^2\varphi_1}
 =\chi_{R_{\Gamma}}-\chi_{R_{\tau}}
 +\sum_{\{\gamma\neq\gamma^{-1}\}\subset\Gamma}\chi_{\sigma_{\gamma}}
 +\sum_{\stackrel{\scriptstyle 1,\tau\neq\gamma\in\Gamma}
                 {\gamma^2=1}}\chi_{\pi_{\gamma}}
$$
which complete the proof of Theorem 
\ref{th:structure-of-wedge-square-of-varphi-1}.

\subsection{Root number of adjoint representation}
\label{subsec:root-number-of-adjoint-representation}
By the decomposition
\ref{eq:decomposition-of-ad-varphi-into-wedge-square-and-varphi-1} and 
Theorem \ref{th:structure-of-wedge-square-of-varphi-1}, the adjoint
representation $\text{\rm Ad}\circ\varphi$ of the Weil group $W_F$ on
$\widehat{\frak g}$ is written as a direct sum of representations
induced from abelian characters. Using this decomposition, we can
calculate the $\varepsilon$-factor of the adjoint representation. The
result is

\begin{thm}\label{th:varepsilon-factor-of-adjoint-representation}
With respect to a additive character $\psi$ of $F$ such that
$$
 \{x\in F\mid\psi(xO_F)=1\}=O_F
$$
and the Haar measure $d(x)$ on $F$ such that 
$\int_{O_F}d(x)=1$, we have
$$
 \varepsilon(\varphi,\text{\rm Ad},\psi,d(x))
 =w(\text{\rm Ad}\circ\varphi)\cdot q^{n^2r}
$$
with the root number
$$
 w(\text{\rm Ad}\circ\varphi)
 =\vartheta(-1)\times
  \begin{cases}
   (-1)^{\frac{q-1}{2n}\cdot\frac{n(n+1)}2}
       &:\text{\rm $K/K_+$ is ramified},\\
   1   &:\text{\rm $K/K_+$ is unramified and $|H|=2$},\\
   -(-1)^{\frac{q^{f_+}-1}2}
       &:\text{\rm $K/K_+$ is unramified and $|H|=4$}.
  \end{cases}
$$
Here $H=\{\gamma\in\Gamma\mid\gamma^2=1\}$ whose structure is given in
Proposition 
\ref{prop:order-two-element-in-tamely-ramified-galois-group}.
\end{thm}

Note that if $K/K_+$ is ramified, then $K/F$ is totally ramified and
hence $2n=(K:F)$ divides $q-1$. 

The rest of this devoted to the proof of the theorem. 

\numsubsubsection{}
\label{subsubsec:decomposition-of-varepsilon-wrt-pi1-pi2-pi3}
To begin with
$$
 \varepsilon(\varphi,\text{\rm Ad},\psi,d(x))
 =\varepsilon(\text{\rm Ad}\circ\varphi,\psi,d(x))
$$
by the definition. Define the
additive character $\psi_F$ of $F$ by
$$
 \psi_F:F\xrightarrow{T_{F/\Bbb Q_p}}
        \Bbb Q_p\xrightarrow{\text{\rm canonical}}
        \Bbb Q_p/\Bbb Z_p\,\tilde{\to}\,
        \Bbb Q/\Bbb Z\xrightarrow{\exp(2\pi\sqrt{-1}\ast)}
        \Bbb C^{\times}.
$$
Then 
$$
 \{x\in F\mid\psi_F(xO_F)=1\}=\mathcal{D}(F/\Bbb Q_p)^{-1}
                             =\frak{p}_F^{-d(F)}
$$
and $\psi_K=\psi_F\circ T_{K/F}$. Let $d_F(x)$ be the Haar measure on 
$F$ such that
$$
 \int_{O_F}d_F(x)=q^{-d(F)}.
$$
Then 
$$
 \varepsilon(\text{\rm Ad}\circ\varphi,\psi,d(x))
 =q^{-n(2n+1)\cdot d(F)/2}
  \varepsilon(\text{\rm Ad}\circ\varphi,\psi_F,d_F(x)).
$$
Put 
$$
 \varepsilon(\ast,\psi_F)=\varepsilon(\ast,\psi_F,d_F(X)),
 \quad
 \lambda(K/F,\psi_F)
 =\lambda(K/F,\psi_F,d_F(x),d_K(x))
$$ 
for the sake of simplicity. By
\eqref{eq:decomposition-of-ad-varphi-into-wedge-square-and-varphi-1}
and Theorem \ref{th:structure-of-wedge-square-of-varphi-1}, we have 
$$
 \text{\rm Ad}\circ\varphi=\Pi_1\oplus\Pi_2\oplus\Pi_3
$$
with 
$\Pi_1=\bigoplus_{\pi(\tau)\neq 1}\pi^{\dim\pi}$ and 
$$
 \Pi_2
 =\text{\rm Ind}_{K^{\times}}^{W_{K/F}}\widetilde\vartheta
   \oplus
   \bigoplus_{\{\gamma\neq\gamma^{-1}\}\subset\Gamma}
   \text{\rm Ind}_{K^{\times}}^{W_{K/F}}
                          \widetilde\vartheta_{\gamma},
 \qquad
 \Pi_3
 =\bigoplus_{\stackrel{\scriptstyle 1,\tau\neq\gamma\in\Gamma}
                       {\gamma^2=1}}
   \text{\rm Ind}_{W_{K/K_{\gamma}}}^{W_{K/F}}\chi_{\gamma},
$$
were $\Pi_3$ appears only if $|H|=4$. Note also that we can change
in $\Pi_2$, the definition of $\widetilde\vartheta_{\gamma}$ to 
$\widetilde\vartheta_{\gamma}(x)=\widetilde\vartheta(x^{1-\gamma})$
for $\gamma\in\Gamma$ such that $\gamma\neq\gamma^{-1}$ by replacing
$\gamma$ with $\tau\gamma$. 
Now we have
$$
 \varepsilon(\text{\rm Ad}\circ\varphi,\psi_F)
 =\varepsilon(\Pi_1,\psi_F)\cdot
  \varepsilon(\Pi_2,\psi_F)\cdot
  \varepsilon(\Pi_3,\psi_F).
$$
Since 
$\Pi_1=\text{\rm Ind}_{W_K}^{W_F}\text{\bf 1}_{W_K}
      -\text{\rm Ind}_{W_{K_+}}^{W_F}\text{\bf 1}_{W_{K_+}}$, we have
\begin{align*}
 &\varepsilon(\Pi_1,\psi_F)
 =\varepsilon\left(
    \text{\rm Ind}_{W_K}^{W_F}\text{\bf 1}_{W_K},\psi_F\right)\cdot
   \varepsilon\left(
    \text{\rm Ind}_{W_{K_+}}^{W_F}\text{\bf 1}_{W_{K_+}}
                   \psi_F\right)^{-1}\\
 &=\lambda(K/F,\psi_F)\varepsilon(\text{\bf 1}_{W_K},\psi_K)
   \cdot
   \lambda(K_+/F,\psi_F)^{-1}
    \varepsilon(\text{\bf 1}_{W_{K_+}},\psi_{K_+})^{-1}
\end{align*}
and 
$$
 \varepsilon(\text{\bf 1}_{W_K},\psi_K)
 =q_K^{d(K)/2}
 =q^{f(e\cdot d(F)+e-1)/2}
$$
where $q_K=|O_K/\frak{p}_K|=q^f$. Since $K/F$ is tamely ramified, we
have $d(K)=e\cdot d(F)+e-1$. Similarly we have
$$
 \varepsilon(\text{\bf 1}_{W_{K_+}},\psi_{K_+})
 =q^{f_+(e_+\cdot d(F)+e_+-1)/2}.
$$
On the other hand, we have
$$
 \varepsilon(\Pi_2,\psi_F)
 =\varepsilon(\widetilde\vartheta,\psi_K)
  \prod_{\{\gamma\neq\gamma^{-1}\}\subset\Gamma}
  \varepsilon(\widetilde\vartheta_{\gamma},\psi_K)
  \times
  \begin{cases}
   \lambda(K/F,\psi_F)^n&:|H|=2,\\
   \lambda(K/F,\psi_F)^{n-1}&:|H|=4.
  \end{cases}
$$
Now we have
$$
 \varepsilon(\widetilde\vartheta,\psi_K)
 =G_{\psi_K}\left(
   \widetilde\vartheta^{-1},\varpi_K^{-(d(K)+f(\widetilde\vartheta))}
            \right)
  \cdot\widetilde\vartheta(\varpi_K)^{d(K)+f(\widetilde\vartheta)}
  \cdot q_K^{(d(K)+f(\widetilde\vartheta))/2}.
$$
Since $\widetilde\vartheta|_{K_+^{\times}}=1$ and 
$K=K_+(\beta)$, Theorem 3 of \cite{Frohlich-Queyrut1973} says that
$$
 G_{\psi_K}\left(
   \widetilde\vartheta^{-1},\varpi_K^{-(d(K)+f(\widetilde\vartheta))}
            \right)
  \cdot\widetilde\vartheta(\varpi_K)^{d(K)+f(\widetilde\vartheta)}
 =\widetilde\vartheta(\beta)=\vartheta(-1).
$$
So we have
$$
 \varepsilon(\widetilde\vartheta,\psi_K)
 =\vartheta(-1)\cdot q^{f(e\cdot d(F)+e-1+f(\widetilde\vartheta))/2}.
$$
Similarly we have
$$
 \varepsilon(\widetilde\vartheta_{\gamma},\psi_K)
 =q^{f(e\cdot d(F)+e-1+f(\widetilde\vartheta_{\gamma}))/2}
$$
for $\gamma\in\Gamma$ such that $\gamma^2\neq 1$, since 
$\widetilde\vartheta_{\gamma}(\beta)=\vartheta((-1)^{1-\gamma})=1$. 

\numsubsubsection{}
\label{subsubsec:root-number-ramified}
Assume that $K/K_+$ is ramified. Then $K/F$ is totally ramified and 
$H=\{1,\tau=\delta^n\}$. Since $e=2n$ is even, we have
$$
 \lambda(K/F,\psi_F)
 =(-1)^{\frac{q-1}
             {2n}\cdot\frac{n(n+1)}2}
  G_{\psi_F}(\left(\frac{\ast}F\right),\varpi_F^{-(d(F)+1)})
$$
by Proposition
\ref{prop:lambda-factor-of-tamely-ramified-galois-ext}. Similarly we
have
$$
 \lambda(K_+/F,\psi_F)
 =\begin{cases}
   (-1)^{\frac{q-1}n\cdot\frac{n(n+2)}8}
   G_{\psi_F}(\left(\frac{\ast}F\right),\varpi_F^{-(d(F)+1)})
     &:\text{\rm $n$ is even},\\
   1 &:\text{\rm $n$ is odd}.
  \end{cases}
$$
So we have
$$
 \varepsilon(\Pi_1,\psi_F)
 =q^{(n\cdot d(F)+n)/2}\times
  \begin{cases}
   (-1)^{\frac{q-1}4}&:\text{\rm $n$ is even},\\
   (-1)^{\frac{q-1}2\cdot\frac{n+1}2}
   G_{\psi_F}(\left(\frac{\ast}F\right),\varpi_F^{-(d(F)+1)})
                     &:\text{\rm $n$ is odd}.
  \end{cases}
$$
By Proposition
\ref{prop:conductor-and-conjugate-triviality-of-tilde-vartheta}, we
have
$$
 f(\widetilde\vartheta)
 =\text{\rm Min}\{0<k\in\Bbb Z\mid
        \widetilde\vartheta(1+\frak{p}_K^k)=1\}
 =2n(r-1)
$$
and
$$
 f(\widetilde\vartheta_{\gamma})
 =\text{\rm Min}\{0<k\in\Bbb Z\mid
   \widetilde\vartheta_{\gamma}(1+\frak{p}_K^k)=1\}
 =2n(r-1)
$$
for $\gamma\in\Gamma$ such that $\gamma^2\neq 1$. Then we have
$$
 \varepsilon(\Pi_2,\psi_F)
 =\vartheta(-1)\cdot q^{n^2d(F)+n^2r-n/2}\lambda(K/F,\psi_F)^n
$$
and
$$
 \lambda(K/F,\psi_F)^n
 =(-1)^{\frac{q-1}2\cdot\frac{n(n+1)}2}
  G_{\psi_F}(\left(\frac{\ast}F\right),\varpi_F^{-(d(F)+1)})^n.
$$
Since
$$
 G_{\psi_F}(\left(\frac{\ast}F\right),\varpi_F^{-(d(F)+1)})^2
 =\left(\frac{-1}F\right)
 =(-1)^{\frac{q-1}2},
$$
we have
$$
 \lambda(K/F,\psi_F)^n=(-1)^{\frac{q-1}s\cdot\frac n2}\cdot
                       (-1)^{\frac{q-1}s\cdot\frac n2}=1
$$
if $n$ is even, and
\begin{align*}
 \lambda(K/F,\psi_F)^n
 &=(-1)^{\frac{q-1}2\cdot\frac{n+1}2}\cdot
   (-1)^{\frac{q-1}2\cdot\frac{n-1}2}
   G_{\psi_F}(\left(\frac{\ast}F\right),\varpi_F^{-(d(F)+1)})\\
 &=G_{\psi_F}(\left(\frac{\ast}F\right),-\varpi_F^{-(d(F)+1)})
\end{align*}
if $n$ is odd. So we finally get
\begin{align*}
 \varepsilon(\text{\rm Ad}\circ\varphi,\psi_F)
 &=\varepsilon(\Pi_1,\psi_F)\cdot\varepsilon(\Pi_2,\psi_F)\\
 &=\vartheta(-1)\cdot q^{n(2n+1)d(F)/2+n^2r}\times
   \begin{cases}
    (-1)^{\frac{q-1}4} &:\text{\rm $n$ is even},\\
    (-1)^{\frac{q-1}2\cdot\frac{n+1}2}
                       &:\text{\rm $n$ is odd}.
   \end{cases}
\end{align*}

\numsubsubsection{}
\label{subsubsec:root-number-unramified-unique-order-two-element}
Assume that $K/K_+$ is unramified and $|H|=2$. In this case,
Proposition
\ref{prop:order-two-element-in-tamely-ramified-galois-group} shows
that $e=e_+$ is odd, and $H=\{1,\tau\}$. 
Then $f=2f_+$ is even, since $ef=2n$. 
By Proposition \ref{prop:lambda-factor-of-tamely-ramified-galois-ext},
we have
$$
 \lambda(K/F,\psi_F)=(-1)^{(f-1)d(F)}=(-1)^{d(F)},
 \quad
 \lambda(K_+/F,\psi_F)=(-1)^{(f_+-1)d(F)}.
$$
So we have
$$
 \varepsilon(\Pi_1,\psi_F,d_F(x))
 =(-1)^{f_+\cdot d(F)}q^{(n\cdot d(F)+n-f_+)/2}.
$$
By Proposition
\ref{prop:conductor-and-conjugate-triviality-of-tilde-vartheta}, we
have
$$
 f(\widetilde\vartheta)
 =\text{\rm Min}\{0<k\in\Bbb Z\mid
    \widetilde\vartheta(1+\frak{p}_K^k)=1\}
 =e(r-1)+1,
$$
and 
\begin{align*}
 f(\widetilde\vartheta_{\gamma})
 &=\text{\rm Min}\{0<k\in\Bbb Z\mid
    \widetilde\vartheta_{\gamma}(1+\frak{p}_K^k)=1\}\\
 &=\begin{cases}
    e(r-1)&:\gamma\in\{\delta^{\pm 1},\delta^{\pm 2},\cdots,
                       \delta^{\pm\frac{e-1}2}\},\\
    e(r-1)+1&:\text{\rm otherwise}
   \end{cases}
\end{align*}
for a $\gamma\in\Gamma$ such that $\gamma\neq\gamma^{-1}$. 
So we have
$$
 \varepsilon(\Pi_2,\psi_F,d_F(x))
 =\vartheta(-1)\cdot(-1)^{n\cdot d(F)}
  q^{n^2(d(F)+r)-)n-f_+)/2}.
$$
Then finally we have
\begin{align*}
 \varepsilon(\text{\rm Ad}\circ\varphi,\psi_F)
 &=\varepsilon(\Pi_1,\psi_F)\cdot\varepsilon(\Pi_2,\psi_F)\\
 &=\vartheta(-1)\cdot q^{n(2n+1)d(F)/2+n^2r}.
\end{align*}

\numsubsubsection{}
\label{subsubsec:root-number-unramified-three-order-two-element}
Assume that $K/K_+$ is unramified and $|H|=4$. In this case, 
Proposition
\ref{prop:order-two-element-in-tamely-ramified-galois-group} shows
that $e=e_+, f=2f_+$ and $m$ are all even, and 
$$
 H=\{1,\tau,\delta^{\prime}=\delta^{\frac e2},
            \tau^{\prime}=\delta^{\prime}\tau\}.
$$
Put
$$
 E=K_+\cap K_{\tau^{\prime}}
  =K_{\tau^{\prime}}\cap K_{\delta^{\prime}}
  =K_{\delta^{\prime}}\cap K_+.
$$
Then $K/K_+, K/K_{\tau^{\prime}}$ and $K_{\delta^{\prime}}/E$ are
unramified quadratic extension, on the other hand 
$K_+/E, K_{\tau^{\prime}}/E$ and $K/K_{\delta^{\prime}}$ are ramified
quadratic extension. $K_0\subset K_{\delta^{\prime}}\subset K$ and 
$E_0=E\cap K_0$ is the maximal unramified subextension of $E/F$.

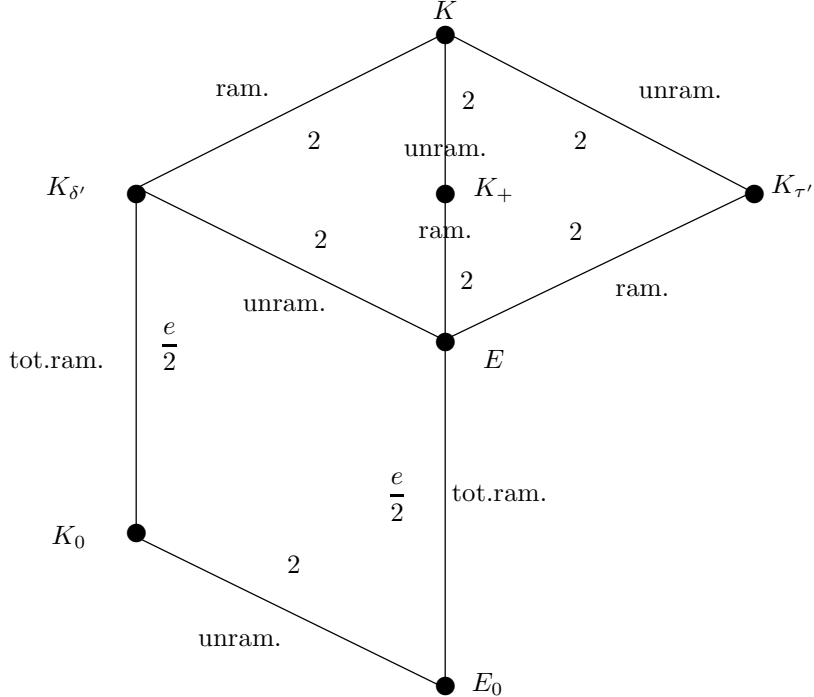
\unitlength 0.1in
\begin{picture}( 39.2000, 36.0500)( 16.4000,-38.1000)
%
\put(36,-4){\circle*{1}}
%
\put(36,-20){\circle*{1}}
%
\put(36,-38){\circle*{1}}
%
\put(36,-12.3){\circle*{1}}
%
\put(52,-12.3){\circle*{1}}
%
\put(20,-12.3){\circle*{1}}
%
\put(20,-30){\circle*{1}}
%
\special{pn 8}%
\special{pa 3610 390}%
\special{pa 5190 1220}%
\special{fp}%
\special{pa 5190 1220}%
\special{pa 3600 1990}%
\special{fp}%
\special{pa 3600 400}%
\special{pa 3600 3790}%
\special{fp}%
%
\special{pn 8}%
\special{pa 3600 390}%
\special{pa 2000 1200}%
\special{fp}%
\special{pa 2010 1200}%
\special{pa 3590 1990}%
\special{fp}%
\special{pa 2000 1210}%
\special{pa 2000 3000}%
\special{fp}%
\special{pa 2000 3020}%
\special{pa 3600 3810}%
\special{fp}%
\special{pa 3600 3810}%
\special{pa 3600 3810}%
\special{fp}%
\put(36.0000,-2.7000){\makebox(0,0){$K$}}%
\put(38.5000,-12.1000){\makebox(0,0){$K_+$}}%
\put(54.0000,-11.9000){\makebox(0,0){$K_{\tau^{\prime}}$}}%
\put(16.4000,-12.0000){\makebox(0,0){$K_{\delta^{\prime}}$}}%
\put(38.5000,-20.9000){\makebox(0,0){$E$}}%
\put(16.5000,-30.2000){\makebox(0,0){$K_0$}}%
\put(38.2000,-37.9000){\makebox(0,0){$E_0$}}%
\put(38.8000,-27.9000){\makebox(0,0){tot.ram.}}%
\put(44.8000,-17.6000){\makebox(0,0)[lb]{ram.}}%
\put(46.0000,-7.3000){\makebox(0,0)[lb]{unram.}}%
\put(24.1000,-7.2000){\makebox(0,0)[lb]{ram.}}%
\put(25.5000,-18.4000){\makebox(0,0)[lb]{unram.}}%
\put(13.4000,-21.4000){\makebox(0,0)[lb]{tot.ram.}}%
\put(23.2000,-35.9000){\makebox(0,0)[lb]{unram.}}%
\put(27.8000,-32.1000){\makebox(0,0)[lb]{$2$}}%
\put(21.2000,-21.5000){\makebox(0,0)[lb]{$\frac e2$}}%
\put(33.5000,-28.1000){\makebox(0,0){$\frac e2$}}%
\put(29.2000,-15.1000){\makebox(0,0)[lb]{$2$}}%
\put(42.4000,-14.7000){\makebox(0,0)[lb]{$2$}}%
\put(36.0000,-10.0000){\makebox(0,0){unram.}}%
\put(36.0000,-14.3000){\makebox(0,0){ram.}}%
\put(29.2000,-9.5000){\makebox(0,0){$2$}}%
\put(43.0000,-9.5000){\makebox(0,0){$2$}}%
\put(37.2000,-7.4000){\makebox(0,0){$2$}}%
\put(37.1000,-16.8000){\makebox(0,0){$2$}}%
\end{picture}%

By Proposition
\ref{prop:lambda-factor-of-tame-galois-with-unramified-order-two-ele},
we have
$$
 \lambda(K/F,\psi_F)=-(-1)^{\frac{q^{f_+}-1}2}.
$$
By Proposition
\ref{prop:conductor-and-conjugate-triviality-of-tilde-vartheta}, we
have
$$
 f(\widetilde\vartheta)
 =\text{\rm Min}\{0<k\in\Bbb Z\mid
   \widetilde\vartheta(1+\frak{p}_K^k)=1\}=e(r-1)+1,
$$
and
\begin{align*}
 f(\widetilde\vartheta_{\gamma})
 &=\text{\rm Min}\{0<k\in\Bbb Z\mid
    \widetilde\vartheta_{\gamma}(1+\frak{p}_K^k)=1\}\\
 &=\begin{cases}
    e(r-1)&:\gamma\in\{\delta^{\pm 1},\delta^{\pm 2},\cdots,
                       \delta^{\pm\frac{e-1}2}\},\\
    e(r-1)+1&:\text{\rm otherwise}
   \end{cases}
\end{align*}
for a $\gamma\in\Gamma$ such that $\gamma\neq\gamma^{-1}$. 
So we have
\begin{equation}
 \varepsilon(\Pi_1,\psi_F)\cdot\varepsilon(\Pi_2,\psi_F)\\
 =\vartheta(-1)\cdot q^{n(2n-1)d(F)/2+n(n-1)r+f_+/2}\cdot
   \lambda(K_+/F,\psi_F)^{-1}.
\label{eq:varepsilon-pi1-varepsilon-pi2-in-three-order-two-element}
\end{equation}

\begin{prop}\label{prop:varepsilon-factor-of-pi3}
\begin{align*}
 \varepsilon(\Pi_3,\psi_F)
 =q^{n\cdot d(F)+nr-f_+/2}
    &\cdot\lambda(K_{\tau^{\prime}}/F,\psi_F)\\
    &\times\begin{cases}
            -1&:\text{\rm $e/2$ is even},\\
            (-1)^{d(F)+\frac{q^{f_+}-1}2}
              &:\text{\rm $e/2$ is odd}.
           \end{cases}
\end{align*}
\end{prop}
\begin{proof}
We have
$$
 \varepsilon(\Pi_3,\psi_F)
 =\prod_{\gamma\in\{\delta^{\prime}, \tau^{\prime}\}}
   \lambda(K_{\gamma}/F,\psi_F)\cdot
   \varepsilon(\widetilde\chi_{\gamma},\psi_{K_{\gamma}})
$$
where 
$\widetilde\chi_{\gamma}(x)
 =(x,K/K_{\gamma})\cdot\widetilde\vartheta(x)$ 
($x\in K_{\gamma}^{\times}$). 

\vspace{2mm}

\noindent
1) The case $\gamma=\tau^{\prime}$. Since $K/K_{\gamma}$ is unramified,
we have
$$
 (x,K/K_{\gamma})=(-1)^{\text{\rm ord}_{K_{\gamma}}(x)}
 \quad
 (x\in K_{\gamma}^{\times})
$$
and $N_{K/K_{\gamma}}(1+\frak{p}_K^k)=1+\frak{p}_{K_{\gamma}}^k$ 
($0<k\in\Bbb Z$). Then we have
$$
 f(\chi_{\gamma})
 =\text{\rm Min}\{0<k\in\Bbb Z\mid
   \chi_{\gamma}(1+\frak{p}_{K_{\gamma}}^k)=1\}
 =e(r-1)
$$
because $\chi_{\gamma}(1+\frak{p}_{K_{\gamma}}^k)=1$ if and only if 
$\widetilde\vartheta(x^{1-\delta^{\prime}})=1$ for all 
$x\in 1+\frak{p}_K^k$ which is equivalent to $k\geq e(r-1)$ by 
Proposition
\ref{prop:conductor-and-conjugate-triviality-of-tilde-vartheta}. 
Since $K_{\gamma}/E$ is ramified quadratic extension, we have 
$K_{\gamma}=E(\sqrt{\varpi_E})$ where $\varpi_E$ is a prime element of
$E$. Then $\widetilde\chi_{\gamma}|_{E^{\times}}=1$ and 
$$
 \widetilde\chi_{\gamma}(\sqrt{\varpi_E})
 =(\sqrt{\varpi_E},K/K_{\gamma})\cdot
  \widetilde\vartheta(\sqrt{\varpi_E})
 =-\vartheta(-1),
$$
hence we have
$$
 G_{\psi_{K_{\gamma}}}(\widetilde\chi_{\gamma}^{-1},
  -\varpi_{K_{\gamma}}^{-(d(K_{\gamma})+f(\widetilde\chi_{\gamma}))})
 \cdot
 \widetilde\chi_{\gamma}(\varpi_{K_{\gamma}})
             ^{d(K_{\gamma})+f(\widetilde\chi_{\gamma})}
 =-\vartheta(-1)
$$
by Theorem 3 of \cite{Frohlich-Queyrut1973}. Then we have
$$
 \varepsilon(\widetilde\chi_{\gamma},\psi_{K_{\gamma}})
 =-\vartheta(-1)\cdot q^{(n\cdot d(F)+nr-f_+)/2}.
$$

\vspace{3mm}

\noindent
2) The case $\gamma=\delta^{\prime}$. In this case $K/K_{\gamma}$ is
ramified quadratic extension. Then we have
$$
 f(\widetilde\chi_{\gamma})
 =\text{\rm Min}\{0<k\in\Bbb Z\mid
   \widetilde\chi_{\gamma}(1+\frak{p}_{K_{\gamma}}^k)=1\}
 =\frac e2\cdot (r-1)+1
$$
because $\widetilde\chi_{\gamma}(1+\frak{p}_{K_{\gamma}}^k)=1$ if and
only if $\widetilde\vartheta(x^{1-\tau^{\prime}})=1$ for all 
$x\in 1+\frak{p}_K^{2k}$ which is equivalent to 
$k\geq \frac e2\cdot (r-1)+1$. There exists a prime element 
$\varpi_{K_{\gamma}}$ of $K_{\gamma}$ such that 
$K=K_{\gamma}(\sqrt{\varpi_{K_{\gamma}}})$, and we have
$$
 (\varpi_{K_{\gamma}},K/K_{\gamma})
 =(-1,K/K_{\gamma})
 =(-1)^{\frac{q^f-1}2}
 =1
$$
since $f=2f_+$ is even. On the other hand $K_{\gamma}/E$ is
unramified quadratic extension, and we have 
$$
 (\varepsilon,K/K_{\gamma})=(\varepsilon,K_{\gamma}/E)=1
$$
for all $\varepsilon\in O_E^{\times}$. Hence
$\widetilde\chi_{\gamma}|_{E^{\times}}=1$. If we put 
$K_{\gamma}=E(\sqrt{\varepsilon})$ with 
$\varepsilon\in O_E^{\times}$, then we have
$$
 G_{\psi_{K_{\gamma}}}(\widetilde\chi_{\gamma}^{-1},
  -\varpi_{K_{\gamma}}^{-(d(K_{\gamma})+f(\widetilde\chi_{\gamma}))})
 = \widetilde\chi_{\gamma}(\sqrt\varepsilon)
 =(\sqrt\varepsilon,K/K_{\gamma})\cdot\vartheta(-1)
$$
by Theorem 3 of \cite{Frohlich-Queyrut1973}. It is shown in the proof
of Proposition \ref{prop:explicit-value-of-c(-1)} that 
$$
 (\sqrt{\varepsilon},K/K_{\gamma})
 =-(-1)^{\frac{q^{f_+}-1}2}.
$$
So we have
$$
 \varepsilon(\widetilde\chi_{\gamma},\psi_{K_{\gamma}})
 =-(-1)^{\frac{q^{f_+}-1}2}\vartheta(-1)\cdot 
  q^{(n\cdot d(F)+nr)/2}.
$$
Since $K_{\gamma}/E$ is unramified quadratic extension, we have
$$
 \lambda(K_{\gamma}/F,\psi_F)
 =\begin{cases}
   -(-1)^{\frac{q^{f_+}-1}2}&:\text{\rm $e/2$ is even},\\
   (-1)^{d(F)}&:\text{\rm $e/2$ is odd}.
  \end{cases}
$$
by Proposition
\ref{prop:lambda-factor-of-tame-galois-with-unramified-order-two-ele}. 
\end{proof}

\begin{prop}\label{prop:lambda-k-tau-prime-times-lambda-k-+-inverse}
$$
 \lambda(K_{\tau^{\prime}}/F,\psi_F)\cdot
 \lambda(K_+/F,\psi_F)^{-1}
 =\begin{cases}
   1&:\text{\rm $e/2$ is even},\\
  (-1)^{d(F)+1}&:\text{\rm $e/2$ is odd}.
  \end{cases}
$$
\end{prop}
\begin{proof}
Since $K_{\delta^{\prime}}/E$ is an unramified quadratic extension, put 
$K_{\delta^{\prime}}=E(\sqrt{\varepsilon})$ with 
$\varepsilon\in O_E^{\times}$. Then 
${\sqrt\varepsilon}^{\tau^{\prime}}=-\sqrt\varepsilon$. Since $K_+/E$
is a ramified quadratic extension, we have $K_+=E(\varpi_{K_+})$ with 
a prime element $\varpi_{K_+}$ of $K_+$ such that 
$\varpi_{K_+}^2\in E$. Then 
${\varpi_{K_+}}^{\tau^{\prime}}=-\varpi_{K_+}$, and hence 
$\varpi_{K_{\tau^{\prime}}}=\sqrt\varepsilon\cdot\varpi_{K_+}$ is a
prime element of $K_{\tau^{\prime}}$ such that 
$K_{\tau^{\prime}}=E(\varpi_{K_{\tau^{\prime}}})$ and 
${\varpi_{K_{\tau^{\prime}}}}^2\in E$. Then 
$$
 \varpi_+=N_{K_+/E_0}(\varpi_{K_+})
 \;\;\text{\rm and}\;\;
 \varpi_{\tau^{\prime}}
 =N_{K_{\tau^{\prime}}/E_0}(\varpi_{K_{\tau^{\prime}}})
$$
are prime elements of $E_0$, since $K_+/E_0$ and
$K_{\tau^{\prime}}/E_0$ are totally ramified extension. On the other
hand, we have
$$
 N_{K_{\tau^{\prime}}/E}(\varpi_{K_{\tau^{\prime}}})
 =-{\varpi_{K_{\tau^{\prime}}}}^2
 =-\varepsilon\cdot{\varpi_{K_+}}^2
 =\varepsilon\cdot n_{K_+/E}(\varpi_{K_+}),
$$
and hence 
$\varpi_{\tau^{\prime}}=N_{E/E_0}(\varepsilon)\cdot\varpi_+$. 
Now we have
\begin{align*}
 &\lambda(K_{\tau^{\prime}}/F,\psi_F)\cdot
  \lambda(K_+/F,\psi_F)^{-1}\\
 =&G_{\psi_{E_0}}(\left(\frac{\ast}
                             {E_0}\right),
                 -\varpi_{\tau^{\prime}}^{-(d(E_0)+1)})\cdot
   G_{\psi_{E_0}}(\left(\frac{\ast}
                             {E_0}\right),
                 -\varpi_+^{-(d(E_0)+1)})^{-1}\\
 =&\left(\frac{N_{E/E_0}(\varepsilon)}
              {E_0}\right)^{d(F)+1}
\end{align*}
by Proposition
\ref{prop:lambda-factor-of-tamely-ramified-galois-ext}. Since
$K_+/E_0$ is a tamely totally ramified extension, and hence a cyclic
extension, let $E_0\subset M\subset K_+$ be the intermediate field
such that $(M:E_0)=2$. Then
$$
 \left(\frac{N_{E/E_0}(\varepsilon)}
              {E_0}\right)
 =(N_{E/E_0}(\varepsilon),M/E_0).
$$
If $e/2$ is even, then $M\subset E$ because $(E:E_0)=e/2$, and hence
$$
 \left(\frac{N_{E/E_0}(\varepsilon)}
              {E_0}\right)
 =(N_{M/E_0}\left(N_{E/M}(\varepsilon)\right),M/E_0)=1.
$$
Assume that $e/2$ is odd. Since $K_+/E$ is a ramified quadratic
extension and $\varepsilon\in O_E^{\times}$ is not square, we have
$$
 (\varepsilon,K_+/E)=\left(\frac{\varepsilon}E\right)=-1.
$$
On the other hand, we have
$$
 (\varepsilon,K_+/E)
 =\left(N_{E/E_0}(\varepsilon),K_+/E_0\right)
 \in\text{\rm Gal}(K_+/E)\subset\text{\rm Gal}(K_+/E_0)
$$
and $\left(N_{E/E_0}(\varepsilon),K_+/E_0\right)$ is mapped to 
$\left(N_{E/E_0}(\varepsilon),M/E_0\right)$ by the restriction mapping
$$
 \text{\rm Gal}(K_+/E_0)\to\text{\rm Gal}(M/E_0).
$$
$M\not\subset E$ Since $(E:E_0)=e/2$ is odd, hence $K_+=ME$ and 
$M\cap E=E_0$. Then the restriction mapping gives the isomorphism
$$
 \text{\rm Gal}(K_+/E)\,\tilde{\to}\,
 \text{\rm Gal}(M/E_0),
$$
hence we have
$$
 \left(\frac{N_{E/E_0}(\varepsilon)}
            {E_0}\right)
 =\left(N_{E/E_0},M/E_0\right)
 =\left(\varepsilon,K_+/E\right)=-1.
$$
\end{proof}

\eqref{eq:varepsilon-pi1-varepsilon-pi2-in-three-order-two-element}
and Proposition \ref{prop:varepsilon-factor-of-pi3} combined with 
Proposition \ref{prop:lambda-k-tau-prime-times-lambda-k-+-inverse}
gives
\begin{align*}
 \varepsilon(\text{\rm Ad}\circ\varphi,\psi_F)
 &=\varepsilon(\Pi_1,\psi_F)\cdot
   \varepsilon(\Pi_2,\psi_F)\cdot
   \varepsilon(\Pi_3,\psi_F)\\
 &=\vartheta(-1)\cdot w^{n(2n+1)\cdot d(F)/2+n^2r}\times
   \begin{cases}
    -1&:\text{\rm $e/2$ is even},\\
    -(-1)^{\frac{q^{f_+}-1}2}&:\text{\rm $e/2$ is odd}.
   \end{cases}
\end{align*}
Since $K_+/F$ is a tamely ramified extension such that $e(K_+/F)=e$
and $f(K_+/F)=f_+$, and $e$ is even, $e/2$ divides
$(q^{f_+}-1)/2$. Hence $(-1)^{\frac{q^{f_+}-1}2}=1$ if $e/2$ is
even. The proof of the formula of Theorem
\ref{th:varepsilon-factor-of-adjoint-representation} is completed.

\subsection{Verification of root number conjecture}
\label{subsec:verification-of-root-number-onjecture}
Let $D$ be the maximal torus of $Sp_{2n}$ consisting of the diagonal
matrices. The group $X^{\vee}(D)$ of the one-parameter subgroups of
$D$ is identified with $\Bbb Z^n$ by $m\mapsto u_m$ where
$$
 u_m(t)=\begin{bmatrix}
         t^m&0\\
          0 &t^{-m}
        \end{bmatrix}
 \;\;\text{\rm with}\;\;
 t^m=\begin{bmatrix}
      t^{m_1}&      &       \\
             &\ddots&       \\
             &      &t^{m_n}
     \end{bmatrix}\in GL_n,
$$
or we will denote by $u_m=\sum_{i=1}^nm_i\cdot u_i$. Then the set of
the co-roots of $SP_{2n}$ with respect to $D$ is
$$
 \Phi^{\vee}(D)
 =\{\pm(u_i\pm u_j), \pm u_k\mid 1\leq i<j\leq n, 1\leq k\leq n\}.
$$
Now we have
\begin{align*}
 2\cdot\rho
 &=\sum_{1\leq i<j\leq n}(u_i-u_j)+\sum_{1\leq i<j\leq n}(u_i+u_j)
   +\sum_{k=1}^n2u_k\\
 &=2\sum_{i=1}^n\{2(n-i)+1\}\cdot u_i.
\end{align*}
So the special central element is
$$
 \epsilon=2\cdot\rho(-1)=-1_{2n}\in Sp_{2n}(F).
$$
If we recall 
$$
 \pi_{\beta,\theta}
 =\text{\rm ind}_{G(O_F)}^{G(F)}\delta_{\beta,\theta}
 \;\;\text{\rm with}\;\;
 \delta_{\beta,\theta}
 =\text{\rm Ind}_{G(O_F/\frak{p}_F^r;\beta)}^{G(O_F/\frak{p}_F^r)}
   \sigma_{\beta,\theta}
$$
and the construction of $\sigma_{\beta,\theta}$, we have
$$
 \pi_{\beta,\theta}(\epsilon)=\delta_{\beta,\theta}(\epsilon)
 =\sigma_{\beta,\theta}(\epsilon)
 =\theta(-1).
$$
Since $\vartheta=c\cdot\theta$, Theorem
\ref{th:varepsilon-factor-of-adjoint-representation} and Proposition 
\ref{prop:explicit-value-of-c(-1)} show that
$$
 w(\text{\rm Ad}\circ\varphi)
 =\theta(-1)\times\begin{cases}
                   1&:\text{\rm $K/K_+$ is unramified},\\
                  (-1)^{\frac{q-1}
                             {2n}\cdot\frac{n(n-1)}2}
                    &:\text{\rm $K/K_+$ is ramified}.
                  \end{cases}
$$
So we have proved the following theorem.

\begin{thm}\label{th:root-number-conjecture-for-sp(2n)}
If $K/F$ is not totally ramified or $K/F$ is totally ramified and
$$
 \frac{q-1}2\cdot(n-1)\equiv 0\npmod{4}, 
$$
then we have 
$w(\text{\rm Ad}\circ\varphi)=\pi_{\beta,\theta}(\epsilon)$.
\end{thm}

This theorem says that the root number conjecture is valid if we
consider $\varphi$ as the Langlands parameter of the supercuspidal
representation $\pi_{\beta,\theta}$ under the required conditions.

\section{The case of $Sp_4(F)$}
\label{sec:case-of-sp(4)}
In this section, let us assume that $K/F$ is a quintic Galois
extension, and consider a candidate of the Langlands parameter of the
supercuspidal representation $\pi_{\beta,\theta}$ of $Sp_4(F)$
different from the parameter considered in the subsection 
\ref{subsec:l-parameter-associated-with-character-of-tame-elliptic-tori}.
Note that $\Gamma=\text{\rm Gal}(K/F)$ is a cyclic group if and only
if $K/F$ is unramified or totally ramified. 

The proofs are omitted because they are 
quite similar to those of the preceding sections.

\subsection{Another candidate for the Langlands parameter}
\label{subsec:another-candidate-for-langlands-parameter}
The character $\theta$ of $U_{K/K_+}$ which parametrizes the
supercuspidal representation $\pi_{\beta,\theta}$ defines the
character $\widetilde\theta$ of $K^{\times}$ by 
$\widetilde\theta(x)=\theta(x^{1-\tau})$. Then the representation
space $V_{\theta}$ of the induced representation 
$\text{\rm Ind}_{K^{\times}}^{W_{K/F}}\widetilde\theta$ has
$W_{K/F}$-quasi invariant anti-symmetric form
$$
 D_{\nu}(\varphi,\psi)
 =\sum_{\gamma\in\text{\rm Gal(K/F)}}\nu(\gamma)\cdot
   \widetilde\theta\left(\alpha_{K/F}(\gamma,\tau)\right)^{-1}
    \varphi(\gamma)\psi(\gamma\tau)
$$
where $\nu$ is a character of $G\text{\rm Gal}(K/F)$ such that
$\nu(\tau)=-1$ (c.f. appendix
\ref{sec:symmetric-or-anti-symmetric-form-on-induced-rep-of-weil-groip}). 
Let us identify $GSp(V_{\theta},D_{\nu})$ with $GSp_4(\Bbb C)$ by
means of the symplectic basis 
$\{u_{\rho},v_{\rho}\}_{\dot\rho\in\Gamma/\langle\tau\rangle}$. Then
we have a group homomorphism
\begin{equation}
 \varphi:W_F\xrightarrow{\text{\rm can.}}
         W_{K/F}
          \xrightarrow{
           \text{\rm Ind}_{K^{\times}}^{W_{K/F}}\widetilde\theta}
         GSp_4(\Bbb C)
          \xrightarrow{(\ast)}
         SO_5(\Bbb C)
\label{eq:representation-of-weil-group-to-so(5)-sp(4)-case}
\end{equation}
where $(\ast)$ is the accidental surjection. The admissible
representation of the Weil-Deligne group $W_F\times SL_2(\Bbb C)$ to
$SO_5(\Bbb C)$ corresponding to the triple 
$(\varphi,SO_5(\Bbb C),0)$ as explained in appendix 
\ref{subsec:weil-deligne-group} is
\begin{equation}
 W_F\times SL_2(\Bbb C)\xrightarrow{\text{\rm proj.}}
 W_F\xrightarrow{\varphi}
 SO_5(\Bbb C)
\label{eq:another-parameter-sp(4)-case}
\end{equation}
which is also denoted by $\varphi$.

\subsection{Formal degree conjecture}
\label{subsec:formal-degree-conjecture-for-another-langlands-parametr}
By writing down the parameter 
\eqref{eq:another-parameter-sp(4)-case} explicitly as in the
subsection 
\ref{subsec:gamma-factor-of-adjoint-representation}, 
we have
\begin{equation}
 Z_{SO_5(\Bbb C)}(\text{\rm Im}\varphi)\simeq
 =\begin{cases}
   \Bbb Z/2\Bbb Z&:\text{\rm $K/F$ is unramified or totally ramified},\\
   \Bbb Z/2\Bbb Z\times\Bbb Z/2\Bbb Z&:\text{\rm otherwise}
  \end{cases}
\label{eq:order-of-centralizer-sp(4)-case}
\end{equation}
and
\begin{equation}
 L(\varphi,\text{\rm Ad},s)
 =\begin{cases}
   1                 &:\text{\rm $K/K_+$ is ramified,}\\
   \frac 1{1+q^{-f_+s}} &:\text{\rm $K/K_+$ is unramified.}
  \end{cases}
\label{eq:l-factor-sp(4)-case}
\end{equation}
The Artin conductor of $\text{\rm Ad}\circ\varphi$ is
\begin{equation}
 a(\text{\rm Ad}\circ\varphi)=8r.
\label{eq:artin-conductor-sp(4)O-case}
\end{equation}
Then
\begin{equation}
 \frac 1{\left|Z_{SO_5(\Bbb C)}(\text{\rm Im}\varphi)\right|}
 \cdot\left|\frac{\gamma(\varphi,\text{\rm Ad},0)}
                 {\gamma(\varphi_0,\text{\rm Ad},0)}\right|
\label{eq:formal-degree-conjecture-sp(4)-case}
\end{equation}
gives the formal degree of the supercuspidal representation 
$\pi_{\beta,\theta}$ given by 
\eqref{eq:formal-degree-of-cuspidal-rep-of-sp(2n)-wrt-euler-poincare} 
if $K/F$ is unramified or totally ramified. If $K/F$ is ramified not
totally ramified, this
is not the case, that is, the order of the centralizer 
$Z_{SO_5(\Bbb C)}(\text{\rm Im}\varphi)$ is twice as big as required, 
in other words, the image $\text{\rm Im}(\varphi)$ of the parameter is
too small.

\subsection{The root number conjecture}
\label{subsec:root-number-conjecture-for-another-langlands-parameter}
Since the parameter \eqref{eq:another-parameter-sp(4)-case} failed the
formal degree conjecture if $K/F$ is ramified not totally ramified, we
will consider in this subsection, the root number conjecture in the
case of $K/F$ being unramified or totally ramified. 

In this case $K/F$ is a cyclic extension. So we put 
$\text{\rm Gal}(K/F)=\langle\rho\rangle$ so that $\tau=\rho^2$. Then
the representation 
\eqref{eq:representation-of-weil-group-to-so(5)-sp(4)-case} has a
decomposition $\varphi=\varphi_1\oplus\det\varphi_1$ with
$$
 \varphi_1:W_F\xrightarrow{\text{\rm can.}}W_{K/F}
   \xrightarrow{
    \text{\rm Ind}_{K^{\times}}^{W_{K/F}}\widetilde\theta_{\rho}}
   O_4(\Bbb C)
$$
where $\widetilde\theta_{\rho}(x)=\widetilde\theta(x^{1-\rho})$ 
($x\in K^{\times}$). Then we have
$$
 \text{\rm Ad}\circ\varphi
 =\bigoplus_{\stackrel{\chi\in\widehat\Gamma}
                      {\chi(\tau)\neq 1}}\chi\oplus
   \text{\rm Ind}_{K^{\times}}^{W_{K/F}}\widetilde\theta^2\oplus
   \text{\rm Ind}_{K^{\times}}^{W_{K/F}}\widetilde\theta_{\rho}.
$$
The epsilon factor with respect to the additive character and the Haar
measure normalized as in Theorem
\ref{th:varepsilon-factor-of-adjoint-representation} is
$$
 \varepsilon(\varphi,\text{\rm Ad},\psi,d(x))
 =w(\text{\rm Ad}\circ\varphi)\cdot q^{4r}
$$
with the root number
$$
 w(\text{\rm Ad}\circ\varphi)
 =\begin{cases}
   1&:\text{\rm $K/F$ is unramified,}\\
   (-1)^{\frac{q-1}4}&:\text{\rm $K/F$ is totally ramified.}
  \end{cases}
$$
This means that the root number conjecture is valid if and only if
$$
 \theta(-1)
 =\begin{cases}
   1&:\text{\rm $K/F$ is unramified,}\\
   (-1)^{\frac{q-1}4}&:\text{\rm $K/F$ is totally ramified.}
  \end{cases}
$$
In other words, the parameter \eqref{eq:another-parameter-sp(4)-case}
is note the Langlands parameter of the supercuspidal representation 
$\pi_{\beta,\theta}$ in general.

\appendix
\section{Local factors}
\label{sec:local-factors}
Fix an algebraic closure $F^{\text{\rm alg}}$ of $F$ in which we will
take every algebraic extensions of $F$. Put
$$
 \nu_F(x)=(F(x):F)^{-1}\text{\rm ord}_F(N_{F(x)/F}(x))
 \;\text{\rm for $\forall x\in F^{\text{\rm alg}}$}
$$
and
$$
 O_K=\{x\in F^{\text{\rm alg}}\mid\nu_F(x)\geq 0\},
 \quad
 \frak{p}_K=\{x\in F^{\text{\rm alg}}\mid\nu(F)(x)>0\}.
$$
Then $\Bbb K=O_K/\frak{p}_K$ is an algebraic extension of 
$\Bbb F=O_F/\frak{p}_F$. If $K/F$ is a finite extension, fix a
generator $\varpi_K\in O_K$ of $\frak{p}_K$.

\subsection{Weil group}
\label{subsec:weil-group}
Let $F^{\text{\rm ur}}$ be the maximal unramified extension of $F$ and 
$\text{\rm Fr}\in\text{\rm Gal}(F^{\text{\rm ur}}/F)$ the
inverse of the Frobenius automorphism of $F^{\text{\rm ur}}$ over
$F$. The the Weil group $W_F$ of $F$ is 
$$
 W_F=\left\{\sigma\in\text{\rm Gal}(F^{\text{\rm alg}}/F)\mid
              \sigma|_{F^{\text{\rm ur}}}\in
               \langle\text{\rm Fr}\rangle\right\}    
$$
The
group $W_F$ is a locally compact group with respect to the topology
such that $I_F=\text{\rm Gal}(F^{\text{\rm alg}}/F^{\text{\rm ur}})$
is an open compact subgroup of $W_F$. 

Let $F^{\text{\rm ab}}$ be the maximal abelian extension of $F$ in
$F^{\text{\rm alg}}$. Then
$$
 \overline{[W_F,W_F]}
 =\text{\rm Gal}(F^{\text{\rm alg}}/F^{\text{\rm ab}})
$$
and
$$
 W_F/\overline{[W_F,W_F]}\xrightarrow[\text{\rm res.}]{\sim}
 \{\sigma\in\text{\rm Gal}(F^{\text{\rm ab}}/F)\mid
    \sigma|_{F^{\text{\rm ur}}}\in\langle\text{\rm Fr}\rangle\}.
$$
So, by the local class field theory, there exists a topological group
isomorphism 
$$
 \delta_F:F^{\times}\,\tilde{\to}\,W_F/\overline{[W_F,W_F]}
$$
such that $\delta_F(\varpi)|_{F^{\text{\rm ur}}}=\text{\rm Fr}$. 
Fix a $\widetilde{\text{\rm Fr}}\in\text{\rm Gal}(F^{\text{\rm alg}}/F)$
such that 
$\widetilde{\text{\rm Fr}}|_{F^{\text{\rm ab}}}
 =\delta_F(\varpi)$. Then
$$
 W_F=\langle\widetilde{\text{\rm Fr}}\rangle\ltimes
      \text{\rm Gal}(F^{\text{\rm alg}}/F^{\text{\rm ur}}).
$$
Let $K/F$ be a finite extension in $F^{\text{\rm alg}}$. Then 
$K^{\text{\rm ur}}=K\cdot F^{\text{\rm ur}}$ and 
$$
 W_K=\{\sigma\in\text{\rm Gal}(F^{\text{\rm alg}}/K)\mid
        \sigma|_{F^{\text{\rm ur}}}\in\langle\text{\rm Fr}^f\rangle
          \}
    =\{\sigma\in W_F\mid\sigma|_K=1\},
$$
where $f=(\Bbb K:\Bbb F)$, is a closed subgroup of $W_F$. 
If further $K/F$ is a Galois extension, then 
$\overline{[W_K,W_K]}\triangleleft W_F$ and
$$
 W_{K/F}=W_F/\overline{[W_K,W_K]}
        =\{\sigma\in\text{\rm Gal}(K^{\text{\rm ab}}/F)\mid
            \sigma|_{F^{\text{\rm ur}}}\in\langle\text{\rm Fr}\rangle
             \}
$$
is called the relative Weil group of $K/F$. Then we have a exact
sequence 
$$
 1\to K{\times}\xrightarrow{\delta_K}W_{K/F}
               \xrightarrow{\text{\rm res.}}\text{\rm Gal}(K/F)
               \to 1
$$
which is the group extension associated with the fundamental calss 
$$
 [\alpha_{K/F}]\in H^2(\text{\rm Gal}(K/F),K^{\times}),
$$
that is, we can identify 
$W_{K/F}=\text{\rm Gal}(K/F)\times K^{\times}$ with the group
operation
$$
 (\sigma,x)\cdot(\tau,y)
 =(\sigma\tau,\alpha_{K/F}(\sigma,\tau)\cdot xy).
$$
Let $K_0=K\cap F^{\text{\rm ur}}$ be the maximal unramified
subextension of $K/F$. Then the fundamental calss can be chosen 
so that $\alpha_{K/F}(\sigma,\tau)\in O_K^{\times}$ for all 
$\sigma, \tau\in\text{\rm Gal}(K/K_0)$, and the image $I_{K/F}$ of 
$I_F=\text{\rm Gal}(F^{\text{\rm alg}}/F^{\text{\rm ur}})\subset W_F$
under the canonical surjection $W_F\to W_{K/F}$ is identified with 
$\text{\rm Gal}(K/K_0)\times O_K^{\times}$.

\subsection{Artin conductor of representations of Weil group}
\label{subsec:artin-conductor}
Let $(\Phi,V)$ be a finite dimensional continuous complex
representation of the Weil group $W_F$. Since 
$I_F\cap\text{\rm Ker}(\Phi)$ is an open subgroup of 
$I_F=\text{\rm Gal}(F^{\text{\rm alg}}/F^{\text{\rm ur}})$, there
exists a finite Galois extension $K/F^{\text{\rm ur}}$ such that 
$text{\rm Gal}(F^{\text{\rm alg}}/K)\subset\text{\rm Ker}(\Phi)$. Let 
\begin{align*}
 V_k=&V_k(K/F^{\text{\rm ur}})\\
 =&\left\{\sigma\in\text{\rm Gal}(K/F^{\text{\rm ur}})\mid
           x^{\sigma}\equiv x\npmod{\frak{p}_K^{k+1}}
           \,\text{\rm for}\,\forall x\in O_K\right\}
\end{align*}
be the $k$-th ramification group of $K/F^{\text{\rm ur}}$ put
$$
 \widetilde V_k
 =\left[\text{\rm Gal}(F^{\text{\rm alg}}/F^{\text{\rm ur}})
         \xrightarrow{\text{\rm res.}}
          \text{\rm Gal}(K/F^{\text{\rm ur}})\right]^{-1}V_k
$$
for $k=0,1,2,3,\cdots$. So $\widetilde V_0=I_F$. The Artin conductor 
$a(\Phi)=a(V)$ is defined by 
$$
 a(\Phi)=a(V)
 =\sum_{k=0}^{\infty}\dim_{\Bbb C}(V/V^{\Phi(\widetilde V_k)})
                     \cdot
                     |V_0/V_k|^{-1}
$$
where 
$$
 V^{\Phi(\widetilde V_k)}
 =\{v\in V\mid \Phi(\widetilde V_k)v=v\}
 \quad
 (k=0,1,2,3,\cdots).
$$

\subsection{$\varepsilon$-factor of representations of Weil group}
\label{subsec:epsilon-factor}
Fix a continuous unitary character $\psi:F\to\Bbb C^{\times}$ of the
additive group $F$ and a Haar measure $d(x)$ of $F$.

Langlands and Deligne \cite{Deligne1973} show that, for every
finite dimensional continuous complex representation $(\Phi,V)$ of
$W_F$, there exists a complex constant 
$$
 \varepsilon(\Phi,\psi,d(x))=\varepsilon(V,\psi,d(x))
$$
which satisfies the following relations:
\begin{enumerate}
\item an exact sequence 
$$
 1\to V^{\prime}\to V\to V^{\prime\prime}\to 1
$$
implies
$$
 \varepsilon(V,\psi,d(x))
 =\varepsilon(V^{\prime},\psi,d(x))\cdot
  \varepsilon(V^{\prime\prime},\psi,d(x)),
$$
\item for a positive real number $r$
$$
 \varepsilon(\Phi,\psi,r\cdot d(x))
 =r^{\dim\Phi}\cdot\varepsilon(\Phi,\psi,d(x)),
$$
\item for any finite extension $K/F$ and a finite dimensional
  continuous complex representation $\phi$ of $W_K$, we have
$$
 \varepsilon\left(\text{\rm Ind}_{W_K}^{W_F}\phi,\psi,d(x)\right)
 =\varepsilon\left(\phi,\psi\circ T_{K/F},d^{(K)}(x)\right)\cdot
  \lambda(K/F,\psi)^{\dim\phi}
$$
where $d^{(K)}(x)$ is a Haar measure of $K$ and
$$
 \lambda(K/F,\psi)=\lambda(K/F,\psi,d(x),d^{(K)}(x))
 =\frac{\varepsilon\left(\text{\rm Ind}_{W_K}^{W_F}\text{\bf 1}_K,
                         \psi,d(x)\right)}
       {\varepsilon\left(\text{\bf 1}_K,\psi\circ T_{K/F},d^{(K)}(x)
                          \right)},
$$
\item if $\dim\Phi=1$, then $\Phi$ factors through 
      $W_F/\overline{[W_F,W_F]}$ and put
$$
 \chi:F^{\times}\xrightarrow{\delta_F}
      W_F/\overline{[W_F,W_F]}\xrightarrow{\Phi}
      \Bbb C~{\times}.
$$
Then we have
$$
 \varepsilon(\Phi,\psi,d(x))=\varepsilon(\chi,\psi,d(x))
$$
where the right hand side is the $\varepsilon$-factor of Tate 
\cite{Tate1979}.
\end{enumerate}

By the definition of $\lambda(K/F,\psi)$, we have the following chain
rule for the finite extensions:

\begin{prop}\label{prop:chain-relation-of-lambda-factor}
For finite extensions $F\subset K\subset L$, we have
$$
 \lambda(L/F,\psi)
 =\lambda(L/K,\psi\circ T_{K/F})\cdot
  \lambda(K/F,\psi)^{(L:K)}.
$$
\end{prop}

If the Haar measure $d(x)$ of $F$ is normalized so that the Fourier
transform
$$
 \widehat\varphi(y)=\int_F\varphi(x)\cdot\psi(-xy)d(x)
$$
has inverse transform
$$
 \varphi(x)=\int_F\widehat\varphi(y)\cdot\psi(xy)d(y),
$$
in other words
$$
 \int_{O_F}d(x)=q^{-n(\psi)/2}\;\;\text{\rm with}\;\;
 \{x\in F\mid \psi(xO_F)=1\}=\frak{p}_F^{-n(\psi)},
$$
then the explicit value of the $\varepsilon$-factor
$\varepsilon(\chi,\psi,d(x))$ is  
\begin{enumerate}
\item if $\chi|_{O_F^{\times}}=1$, then 
\begin{equation}
 \varepsilon(\chi,\psi,d(x))
 =\chi(\varpi)^{n(\psi)}\cdot q^{n(\psi)/2},
\label{eq:explicit-value-of-epsilon-factor-trivial}
\end{equation}
\item if $\chi|_{O_F^{\times}}\neq 1$, then 
\begin{equation}
 \varepsilon(\chi,\psi,d(x))
 =G_{\psi}(\chi^{-1},-\varpi^{-(n(\psi)+f(\chi))})\cdot
  \chi(\varpi)^{n(\psi)+f(\chi)}\cdot q^{-(n(\psi)+f(\chi))/2}
\label{eq:explicit-value-of-epsilon-facot-non-trivial}
\end{equation}
where 
$f(\chi)=\text{\rm Min}\{0<n\in\Bbb Z\mid\chi(1+\frak{p}_F^n)=1\}$ and
$$
 G_{\psi}(\chi^{-1},\varpi^{-(n(\psi)+f(\chi))})
 =q^{-n/2}\sum_{\dot t\in(O_F/\frak{p}_F^{f(\chi)})^{\times}}
           \chi(t)^{-1}\psi\left(-\varpi^{-(n(\psi)+f(\chi))}t\right)
$$
is the Gauss sum.
\end{enumerate}

\begin{rem}\label{remark:normalization-of-gauss-sum}
The definition of the Gauss sum is normalized so that 
$$
 \left|G_{\psi}(\chi^{-1},-\varpi^{-(n(\psi)+f(\chi))})\right|=1.
$$
\end{rem}

We have

\begin{prop}\label{prop:scaling-o-epsilon-factor}
\begin{enumerate}
\item Put $\psi_a(x)=\psi(ax)$ for $a\in F^{\times}$. Then
$$
 \varepsilon(\Phi,\psi_a,d(x))
 =\det\Phi(a)\cdot|a|_F^{-\dim\Phi}\cdot
  \varepsilon(\Phi,\psi,d(x))
$$
where
$$
 \det\Phi:F^{\times}\xrightarrow{\delta_F}
          W_F/\overline{[W_F,W_F]}\xrightarrow{\det\circ\Phi}
          \Bbb C^{\times}.
$$
\item For any $s\in\Bbb C$ 
\begin{align*}
 \varepsilon(\Phi,\psi,d(x),s)
 &=\varepsilon(\Phi\otimes|\cdot|_F^s,\psi,d(x))\\
 &=\varepsilon(\Phi,\psi,d(x))\cdot
   q^{-s(n(\psi)\cdot\dim\Phi+a(\Phi))}.
\end{align*}
\end{enumerate}
\end{prop}

\begin{prop}\label{prop:epsilon-factor-wrt-normaized-unramified-character}
If $n(\psi)=0$ and the Haar measure $d(x)$ is normalized so that 
$$
 \int_{O_F}d(x)=1,
$$
then
$$
 \varepsilon(\Phi,\psi,d(x))=w(\Phi)\cdot q^{a(\Phi)/2}
                            =w(V)\cdot q^{a(V)/2}
$$
with $w(\Phi)\in\Bbb C$ of absolute value one.
\end{prop}

When $K/F$ is a finite tamely ramified Galois extension, the
maximal unramified subextension $K_0=K\cap F^{\text{\rm ur}}$ is a
cyclic extension of $F$ and $K/K_0$ is also cyclic extension. So, by
means of Proposition \ref{prop:chain-relation-of-lambda-factor}, we
can give the explicit value of $\lambda(K/F,\psi)$. 

Let $\psi_F:F\to\Bbb C^{\times}$ be a continuous unitary character
such that
$$
 \{x\in F\mid\psi_F(xOF)=1\}=\mathcal{D}(F/\Bbb Q_p)^{-1}
 =\frak{p}_F^{-d(F)}
$$
and the Haar measure $d_F(x)$ on $F$ is normalized so that
$$
 \int_{O_F}d_F(x)=q^{-d(F)}.
$$
Let $K/F$ be a tamely ramified finite Galois extension, and put
$\psi_K=\psi_F\circ T_{K/F}$. Put
$$
 e=e(K/F)=(K:K_0),
 \quad
 f=f(K/F)=(K_0:F)
$$
where $K_0=K\cap F^{\text{\rm ur}}$ is the maximal unramified
subextension of $K/F$. Let
$$
 \left(\frac{\varepsilon}{K_0}\right)
 =\begin{cases}
   1&:\varepsilon\equiv\text{\rm square}\npmod{\frak{p}_{K_0}},\\
  -1&:\text{\rm otherwise}
  \end{cases}
 \qquad
 (\varepsilon\in O_{K_0}^{\times})
$$
be the Legendre symbol of $K_0$. Then we have

\begin{prop}\label{prop:lambda-factor-of-tamely-ramified-galois-ext}
\begin{align*}
 \lambda(K/F,\psi_F)
 &=\lambda(K/F,\psi_F,d_F(x),d_K(x))\\
 &=\begin{cases}
    (-1)^{\frac{q^f-1}e\cdot\frac{e(e+2)}8}\cdot
    G_{\psi_{K_0}}(\left(\frac{\ast}{K_0}\right),
                   \varpi_0^{-(d(K_0)+1)})
      &:e=\text{\rm even},\\
    (-1)^{(f-1)d(F)}
      &:e=\text{\rm odd}
   \end{cases}
\end{align*}
where $\varpi_0$ is a prime element of $K_0$ such that 
$\varpi_0\in N_{K/K_0}(K^{\times})$.
\end{prop}

\begin{prop}
\label{prop:lambda-factor-of-tame-galois-with-unramified-order-two-ele}
If there exists an intermediate field $F\subset E\subset K$ such that
$K/E$ is unramified quadratic extension, then $f=2f_+$ is even and
$$
 \lambda(K/F,\psi_F)
 =\begin{cases}
   -(-1)^{\frac{q^{f_+}-1}2}&:\text{\rm $e$ is even},\\
   (-1)^{d(F)}&:\text{\rm $e$ is odd}.
  \end{cases}
$$
\end{prop}

\subsection{$\gamma$-factors of admissible representations of Weil group}
\label{subsec:admissible-representation-of-weil-group}

\begin{dfn}\label{def:admisible-representation-of-weil-group}
The pair $(\Phi,V)$ is called an admissible representation of $W_F$ if
\begin{enumerate}
\item $V$ is a finite dimensional complex vector space and $\Phi$ is a
  group homomorphism of $W_F$ to $GL_{\Bbb C}(V)$,
\item $\text{\rm Ker}(\Phi)$ is an open subgroup of $W_F$,
\item $\Phi(\widetilde{\text{\rm Fr}})\in GL_{\Bbb C}(V)$ is
  semisimple.
\end{enumerate}
\end{dfn}

Let $(\Phi,V)$ be an admissible representation of $W_F$. Since 
$I_F=\text{\rm Gal}(F^{\text{\rm alg}}/F^{\text{\rm ur}})$ is a normal
subgroup of $W_F$, $\Phi(\widetilde{\text{\rm Fr}})\in GL_{\Bbb C}(V)$
keeps 
$$
 V^{I_F}=\{v\in V\mid\Phi(\sigma)v=v\;\forall\sigma\in I_F\}
$$
stable. Then the $L$-factor of $(\Phi,V)$ is defined by 
$$
 L(\Phi,s)=L(V,s)
 =\det\left(1-q^{-s}\cdot\Phi(\widetilde{\text{\rm Fr}})|_{V^{I_F}}
              \right)^{-1}.
$$
Since $\Phi:W_F\to GL_{\Bbb C}(V)$ is continuous group homomorphism,
we have the $\varepsilon$-factor $\varepsilon(\Phi,\psi,d(x),s)$ of
$\Phi$. Then the $\gamma$-factor of $(\Phi,V)$ is defined by
$$
 \gamma(\Phi,\psi,d(x),s)=\gamma(V,\psi,d(x),s)
 =\varepsilon(\Phi,\psi,d(x),s)\cdot
  \frac{L(\Phi\sphat,1-s)}
       {L(\Phi,s)}
$$
where $\Phi\sphat$ is the dual representation of $\Phi$.

\subsection{Symmetric tensor representation of $SL_2(\Bbb C)$}
\label{subsec:symmetric-tensor-representation-fo-sl(2)}
The complex special linear group $SL_2(\Bbb C)$ acts on the polynomial
ring $\Bbb C[X,Y]$ of two variables $X,Y$ by
$$
 g\cdot \varphi(X,Y)=\varphi((X,Y)g)
 \qquad
 (g\in SL_2(\Bbb C), \varphi(X,Y)\in\Bbb C[X,Y]).
$$
Let
$$
 \mathcal{P}_n
 =\langle X^n,X^{n-1}Y,\cdots,XY^{n-1},Y^n\rangle_{\Bbb C}
$$
be the subspace of $\Bbb C[X,Y]$ consisting of the homogeneous
polynomials of degree $n$. The action of $SL_2(\Bbb C)$ on
$\mathcal{P}_n$ defines the symmetric tensor representation 
$\text{\rm Sym}_n$ of degree $n+1$. The complex vector space
$\mathcal{P}_n$ has a non-degenerate bilinear form defined by
$$
 \langle\varphi,\psi\rangle
 =\left.
  \varphi\left(-\frac{\partial}
                     {\partial Y},\frac{\partial}
                                       {\partial X}\right)
  \psi(X,Y)\right|_{(X,Y)=(0,0)}\in\Bbb C
$$
for $\varphi,\psi\in\mathcal{P}_n$. This bilinear form is 
$SL_2(\Bbb C)$-invariant
$$
 \langle\text{\rm Sym}_n(g)\varphi,\text{\rm Sym}_n(g)\psi\rangle
 =\langle\varphi,\psi\rangle
 \quad
 (g\in SL_2(\Bbb C), \varphi,\psi\in\mathcal{P}_n)
$$
and
$$
 \langle\psi,\varphi\rangle=(-1)^n\langle\varphi,\psi\rangle
 \quad
 (\varphi,\psi\in\mathcal{P}_n).
$$
So we have group homomorphisms
$$
 \text{\rm Sym}_n:SL_2(\Bbb C)\to SO(\mathcal{P}_n)
 \;\;\text{\rm if $n$ is even}
$$
and
$$
 \text{\rm Sym}_n:SL_2(\Bbb C)\to Sp(\mathcal{P}_n)
 \;\;\text{\rm if $n$ is odd}.
$$

\subsection{Admissible representations of Weil-Deligne group}
\label{subsec:weil-deligne-group}
Fix a complex Lie group $\mathcal{G}$ such that the connected
component $\mathcal{G}^o$ is a reductive complex algebraic linear
group. Then the $\mathcal{G}^o$-conjugacy class of the group
homomorphisms 
$$
 \varphi:W_F\times SL_2(\Bbb C)\to\mathcal{G}
$$
such that
\begin{enumerate}
\item $I_F\cap\text{\rm Ker}(\varphi)$ is an open subgroup of $I_F$,
\item $\varphi(\widetilde{\text{\rm Fr}})\in\mathcal{G}$ is semi-simple,
\item $\varphi|_{SL_2(\Bbb C)}:SL_2(\Bbb C)\to\mathcal{G}^o$ is a
  morphism of complex linear algebraic group
\end{enumerate}
corresponds bijectively the equivalence classes of the triples 
$(\rho,\mathcal{G},N)$ where $N\in\text{\rm Lie}(\mathcal{G})$ is a
nilpotent element and 
$$
 \rho:W_F\to\mathcal{G}
$$
is a group homomorphism such that
\begin{enumerate}
\item $\rho|_{I_F}:I_F\to\mathcal{G}$ is continuous,
\item $\rho(\widetilde{\text{\rm Fr}})\in\mathcal{G}$ is semi-simple,
\item $\rho(g)N=|g|_F\cdot N$ for $\forall g\in W_F$ where
$$
 |\cdot|_F:W_F\xrightarrow{\text{\rm can.}}
           W_F/\overline{[W_F,W_F]}\xrightarrow{\text{\rm l.c.f.t.}}
           F^{\times}\xrightarrow{q^{-\text{\rm ord}_F(\cdot)}}
           \Bbb Q^{\times}
$$
\end{enumerate}
by the relations
$$
 \rho|_{I_F}=\varphi|_{I_F},
 \quad
 \rho(\widetilde{\text{\rm Fr}})
 =\varphi(\widetilde{\text{\rm Fr}})\cdot
  \varphi\begin{pmatrix}
          q^{-1/2}&0\\
          0&q^{1/2}
         \end{pmatrix},
 \quad
 N=d\varphi\begin{pmatrix}
            0&1\\
            0&0
           \end{pmatrix}
$$
(see \cite[Prop.2.2]{Gross-Reeder2010}). Here two triples 
$(\rho,\mathcal{G},N)$ and $(\rho^{\prime},\mathcal{G},N^{\prime})$ is
 equivalent if there exists a $g\in\mathcal{G}$ such that 
$\rho^{\prime}=g\rho g^{-1}$ and $N^{\prime}=\text{\rm Ad}(g)N$. 

The couple
 $(\varphi,\mathcal{G})$ or the triple $(\rho,\mathcal{G},N)$ is
 called an admissible representation of the Weil-Deligne group. 

Let $(r.V)$ be a continuous finite dimensional complex representation
of $\mathcal{G}$ which is algebraic on $\mathcal{G}^o$. Then the
$L$-factor associated with $(\varphi,\mathcal{G})$ and $(r,V)$ 
is defined by
$$
 L(\varphi,r,s)
 =\det\left(
   1-q^{-s}r\circ\rho(\widetilde{\text{\rm Fr}})|_{V_N^{I_F}}
             \right)^{-1},
$$
where $V_N=\{v\in V\mid dr(N)v=0\}$ and 
$$
 V_N^{I_F}
 =\{v\in V_N\mid r\circ\rho(\sigma)v=v\;\forall\sigma\in I_F\}.
$$
The $\varepsilon$-actor is defined by
$$
 \varepsilon(\varphi,r,\psi,d(x),s)
 =\varepsilon(r\circ\rho,\psi,d(x),s)\cdot
  \det\left(-q^{-s}r\circ\rho(\widetilde{\text{\r Fr}})
             |_{V^{I_F}/V_N^{I_F}}\right)
$$
where $\varepsilon(r\circ\rho,\psi,d(x),s)$ is the $\varepsilon$-factor
of the representation $(r\circ\rho,V)$ of $W_F$ defined in the
subsection \ref{subsec:admissible-representation-of-weil-group}. 
Finally the $\gamma$-factor is defined by 
$$
 \gamma(\varphi,r,\psi,d(x),s)
 =\varepsilon(\varphi,r,\psi,d(x),s)\cdot
   \frac{L(\varphi,r^{\vee},1-s)}
        {L(\varphi,r,s)}
$$
where $r^{\vee}$ is the dual representation of $r$. 

Let $\text{\rm Sym}_n$ be the symmetric tensor representation of
$SL_2(\Bbb C)$ of degree $n+1$. Then the $W_F\times SL_2(\Bbb
C)$-module $V$ has a decomposition
$$
 V=\bigoplus_{n=0}^{\infty}V_n\otimes\text{\rm Sym}_n
$$
where $V_n$ is a $W_F$-module. Then we have
$$
 V_N^{I_F}
 =\bigoplus_{n=0}^{\infty}V_n^{I_F}\otimes\text{\rm Sym}_{n,N}
$$
where $\text{\rm Sym}_{n,N}$ is the highest part of 
$\text{\rm Sym}_n$. 
Since $r\circ\rho(\widetilde{\text{\rm Fr}})$ act on 
$V_n\otimes\text{\rm Sym}_{n,N}$ by 
$q^{-n/2}r\circ\varphi(\widetilde{\text{\rm Fr}})$, we have
$$
 L(\varphi,r,s)
 =\prod_{n=0}^{\infty}\det\left(
   1-q^{-(s+n/2)}r\circ\varphi(\widetilde{\text{\rm Fr}})|_{V_n^{I_F}}
                                 \right)^{-1}.
$$
If the Haar measure $d(x)$ on the additive group $F$ and the additive
character $\psi:F\to\Bbb C^{\times}$ are normalized so
that $\int_{O_F}d(x)=1$ and
$$
 \{x\in F\mid\psi(xO_F)=1\}=O_F,
$$
then we have
$$
 \varepsilon(\varphi,r,\psi,d(x),s)
 =w(\varphi,r)\cdot q^{a(\varphi,r)(1/2-s)}
$$
where
$$
 w(\varphi,r)=\prod_{n=0}^{\infty}w(V_n)^{n+1}\cdot
              \prod_{n=1}^{\infty}\det\left(
               -\varphi(\widetilde{\text{\rm Fr}})|_{V_n^{I_F}}
                                              \right)^n
$$
and
$$
 a(\varphi,r)=\sum_{n=0}^{\infty}(n+1)a(V_n)
             +\sum_{n=1}^{\infty}n\cdot\dim V_n^{I_F}.
$$
If $\varphi|_{SL_2(\Bbb C)}=1$, then $V_n=0$ for all $n>0$ and we have
$$
 w(\varphi,r)=w(r\circ\varphi)=w(r\circ\rho),
 \quad
 a(\varphi,r)=a(r\circ\varphi)=a(r\circ\rho).
$$

\section{Symmetric or anti-symmetric forms on induced
            representations of Weil group}
\label{sec:symmetric-or-anti-symmetric-form-on-induced-rep-of-weil-groip}

Let $K/F$ be a finite Galois extension of even degree. We will assume that 
the elements of $\Gamma=\text{\rm Gal}(K/F)$ of order two are central 
\footnote{This is the case if $K/F$ is tamely ramified extension. 
See Proposition
\ref{prop:order-two-element-in-tamely-ramified-galois-group}.}. 
Fix an element $\tau\in\Gamma$ of order two. 
Let $K_+$ be the  intermediate field of $K/F$ such that 
$\text{\rm Gal}(K/K^+)=\langle\tau\rangle$, and put
$$
 U_{K/K_+}=\{\varepsilon\in O_K^{\times}\mid
              N_{K/K_+}(\varepsilon)=1\}.
$$
Take a continuous unitary character 
$\vartheta:U_{K/K_+}\to\Bbb C^{\times}$ and put 
$\widetilde\vartheta(x)=\vartheta(x^{1-\tau})$ ($x\in K^{\times}$). 
The representation space 
$V_{\vartheta}
 =\text{\rm Ind}_{K^{\times}}^{W_{K/F}}\widetilde\vartheta$ 
is the complex vector space of the $\Bbb C$-valued functions $v$ on
$\Gamma$ on which 
$(\sigma,x)
 \in W_{K/F}=\Gamma{\ltimes}_{\alpha_{K/F}}K^{\times}$ acts by
$$
 (x\cdot v)(\gamma)
 =\widetilde\vartheta(x^{\gamma})\cdot v(\gamma),
 \quad
 (\sigma\cdot v)(\gamma)
 =\widetilde\vartheta(\alpha_{K/F}(\sigma,\sigma^{-1}\gamma))
    \cdot v(\sigma^{-1}\gamma)
$$
with the fundamental class 
$[\alpha_{K/F}]\in H^2(\Gamma,K^{\times})$. 
The character $\chi_{\vartheta}$ of $V_{\vartheta}$ is
$$
 \chi_{\vartheta}(\sigma,x)
 =\begin{cases}
   0&:\sigma\neq 1,\\
   \sum_{\gamma\in\Gamma}\widetilde\vartheta(x^{\gamma})
    &:\sigma=1
  \end{cases}
$$
for $(\sigma,x)\in W_{K/F}$, which is self-conjugate, that is 
$\overline\chi_{\vartheta}=\chi_{\vartheta}$. 

Let $\nu:W_{K/F}\to\Bbb C^{\times}$ be a continuous group
homomorphism. We will look for the $\nu$-invariant $\nu$-symmetric
bilinear form on $V_{\vartheta}$, that is, the non-zero complex
bilinear form $B$ on $V_{\vartheta}$ such that
\begin{enumerate}
\item $B(g\cdot u,g\cdot v)=\nu(g)\cdot B(u,v)$ for all 
      $g\in W_{K/F}$,
\item $B(v,u)=\nu(\tau)\cdot B(u,v)$ for all $u,v\in V_{\vartheta}$.
\end{enumerate}
Note that, in this case, we have $\nu(\tau)=\pm 1$. 

If $\nu|_{K^{\times}}=1$, then 
$$
 B_{\nu}(u,v)
 =\sum_{\gamma\in\Gamma}\nu(\gamma)\cdot
   \widetilde\vartheta\left(
    \alpha_{K/F}(\gamma,\tau)\right)^{-1}\cdot
     u(\gamma)v(\gamma\tau)
 \quad
 (u,v\in V_{\vartheta})
$$
is a non-degenerate $\nu$-invariant $\nu$-symmetric 
bilinear form on $V_{\vartheta}$. For a $\rho\in\Gamma$, define 
$w_{\rho}\in V_{\vartheta}$ by
$$
 w_{\rho}(\gamma)=\begin{cases}
                   1&:\gamma=\rho,\\
                   0&:\gamma\neq\rho
                  \end{cases}
$$
and $u_{\rho}, v_{\rho}\in V_{\vartheta}$ by
$$
 u_{\rho}=\nu(\rho)^{-1}w_{\rho},
 \qquad
 v_{\rho}
 =\widetilde\vartheta\left(\alpha_{K/F}(\rho,\tau)\right)\cdot
   w_{\rho\tau}.
$$
If we fix a complete system of representatives $\mathcal{S}$ of
$\Gamma/\langle\tau\rangle$, 
then $\{u_{\rho},v_{\rho}\}_{\dot\rho\in\mathcal{S}}$ is
a $\Bbb C$-basis of $V_{\vartheta}$ such that
$$
 B_{\nu}(u_{\rho},u_{\rho^{\prime}})
 =B_{\nu}(v_{\rho},v_{\rho^{\prime}})=0,
 \quad
 B_{\nu}(u_{\rho},v_{\rho^{\prime}})
 =\begin{cases}
   1&:\rho=\rho^{\prime},\\
   0&:\rho\neq\rho^{\prime}.
  \end{cases}
$$

\begin{prop}\label{prop:existence-of-nu-invarinat-nu-symmetric-from}
Assume that
\begin{enumerate}
\item $\nu$ is of finite order,
\item $\{\sigma\in\Gamma\mid
         \widetilde\vartheta(x^{\sigma})=\widetilde\vartheta(x)\;
         \forall x\in 1+\frak{p}_K\}=\{1\}$.
\end{enumerate}
Then $V_{\vartheta}$ has $\nu$-invariant $\nu$-symmetric bilinear form
if and only if $\nu|_{K^{\times}}=1$. In this case, the form is a
constant multiple of $B_{\nu}$.
\end{prop}
\begin{proof}
Due to the second assumption and Remark
\ref{remark:condition-for-irreducibility-of-induced-rep}, the induced
representation 
$V_{\vartheta}
 =\text{\rm Ind}_{K^{\times}}^{W_{K/F}}\widetilde\vartheta$ is
 irreducible. 
Since $\nu$ is of finite order, we can choose positive integers $s,t$
such that $M=\langle\varpi_K^s\rangle\times(1+\frak{p}_K^t)$ is a 
$\Gamma$-subgroup of $K^{\times}$ on which $\vartheta$ and $\nu$ are
trivial. Then the induced representation 
$\text{\rm Ind}_{K^{\times}}^{W_{K/F}}\vartheta$ and the character
$\nu$ factor through the canonical morphism
$$
 W_{K/F}\to G=\Gamma{\ltimes}_{\alpha_{K/F}}K^{\times}/M.
$$
So we will consider them on the finite group $G$. The it is well-known
that
$$
 \dim_{\Bbb C}\text{\rm Hom}_G
  (\overline\nu,
   \text{\rm Hom}_{\Bbb C}(V_{\vartheta},V_{\vartheta}^{\ast}))
 =\begin{cases}
   1&:\nu\cdot\chi_{\vartheta}=\chi_{\vartheta},\\
   0&:\text{\rm otherwise},
  \end{cases}
$$
where $V_{\vartheta}^{\ast}$ is the dual representation of
$V_{\vartheta}$. Since 
$T\in\text{\rm Hom}_{\Bbb C}(V_{\vartheta},V_{\vartheta}^{\ast})$
gives a complex bilinear form
$$
 B_T(u,v)=\langle u,Tv\rangle
 \qquad
 (u,v\in V_{\vartheta})
$$
with the canonical pairing 
$\langle\,,\rangle:V_{\vartheta}\times V_{\vartheta}^{\ast}
 \to\Bbb C$, and
$$
 \text{\rm Hom}_{\Bbb C}(V_{\vartheta},V_{\vartheta}^{\ast})
 =\text{\rm Sym}(V_{\vartheta},V_{\vartheta}^{\ast})\oplus
  \text{\rm Alt}(V_{\vartheta},V_{\vartheta}^{\ast}),
$$
there exists $\nu$-invarinat $\nu$-symmetric bilinear form on
$V_{\vartheta}$ if and only if
$\nu\cdot\chi_{\vartheta}=\chi_{\vartheta}$, and in this case
\begin{align*}
 \dim_{\Bbb C}\text{\rm Hom}_G(\overline\nu,
    \text{\rm Sym}(V_{\vartheta},V_{\vartheta}^{\ast}))
 &=\frac 12\left\{1+\frac 1{|G|}\sum_{g\in G}
                    \nu(g)\chi_{\vartheta}(g^2)\right\},\\
 \dim_{\Bbb C}\text{\rm Hom}_G(\overline\nu,
    \text{\rm Alt}(V_{\vartheta},V_{\vartheta}^{\ast}))
 &=\frac 12\left\{1-\frac 1{|G|}\sum_{g\in G}
                    \nu(g)\chi_{\vartheta}(g^2)\right\},
\end{align*}
that is
$$
 \frac 1{|G|}\sum_{g\in G}\nu(g)\chi_{\vartheta}(g^2)
 =\nu(\tau).
$$
Let us assume $\nu\cdot\chi_{\vartheta}=\chi_{\vartheta}$. Then the
prime element $\varpi_K$ of $K$ can be chosen so that 
$\nu(\varpi_K)=1$. In fact 
there exists a prime element $\varpi_K$of $K$ such that 
$\varpi_K^{\tau}=\pm\varpi_K$. Then
$$
 \widetilde\vartheta(\varpi_K^{\gamma})
 =\vartheta(\varpi_K^{\gamma(1-\tau)})
 =\vartheta(\varpi_K^{(1-\tau)\gamma}
 =\vartheta(\pm 1)
$$
for all $\gamma\in\Gamma$, and hence
$$
 \chi_{\vartheta}(1,\varpi_K)
 =\sum_{\gamma\in\Gamma}\widetilde\vartheta(\varpi_K^{\gamma})
 =|\Gamma|\cdot\vartheta(\pm 1)\neq 0.
$$
Then $\nu\cdot\chi_{\vartheta}=\chi_{\vartheta}$ implies 
$\nu(\varpi_K)=1$. 

Note also that $\nu(x^{\gamma})=\nu(x)$ for all $x\in K^{\times}$ and 
$\gamma\in\Gamma$, since 
$(\gamma,1)^{-1}(1,x)(\gamma,1)=(1,x^{\gamma})$. 

Since
$$
 \chi_{\vartheta}((\sigma,x)^2)
 =\begin{cases}
   0&:\sigma^2\neq 1,\\
   \sum_{\gamma\in\Gamma}\widetilde\vartheta\left(
    x^{(1+\sigma)\gamma}\alpha_{K/F}(\sigma,\sigma)^{\gamma}
                                            \right)
    &:\sigma^2=1,
  \end{cases}
$$
we have
\begin{align*}
 \sum_{g\in G}\nu(g)\chi_{\vartheta}(g^2)
 &=\sum_{\stackrel{\scriptstyle \sigma,\gamma\in\Gamma}
                  {\sigma^2=1}}
   \sum_{\dot x\in K^{\times}/M}
    \nu(\sigma,x)\cdot
    \widetilde\vartheta\left(
           x^{(1+\sigma)\gamma}\alpha_{K/F}(\sigma,\sigma)^{\gamma}
                       \right)\\
 &=\sum_{\stackrel{\scriptstyle \sigma,\gamma\in\Gamma}
                  {\sigma^2=1}}
   \sum_{\dot x\in K^{\times}/M}
    \nu(\sigma,x^{\gamma^{-1}})\cdot
    \widetilde\vartheta\left(
           x^{(1+\sigma)}\alpha_{K/F}(\sigma,\sigma)^{\gamma}
                       \right)\\
 &=\sum_{\stackrel{\scriptstyle \sigma,\gamma\in\Gamma}
                  {\sigma^2=1}}
    \nu(\sigma)\cdot
    \widetilde\vartheta\left(\alpha_{K/F}(\sigma,\sigma)^{\gamma}\right)
   \sum_{\dot x\in K^{\times}/M}
    \nu(x)\cdot\widetilde\vartheta(x^{1+\sigma}).
\end{align*}
Since $\nu(\varpi_K)=1$ and $\varpi_K^{1-\tau}=\pm 1$, we have
$$
 \widetilde\vartheta(\varpi_K^{1+\sigma})
 =\vartheta\left(\varpi_K^{(1+\sigma)(1-\tau)}\right)
 =\vartheta\left((\pm 1)^{1+\sigma}\right)=1
$$
for all $\sigma\in\Gamma$, we have
$$
 \sum_{\dot x\in K^{\times}/M}\nu(x)\cdot
                   \widetilde\vartheta(x^{1+\sigma})
 =s\sum_{\dot x\in(O_K/\frak{p}_K^t)^{\times}}
    \nu(x)\widetilde\vartheta(x^{1+\sigma}).
$$
If $\nu(x)\widetilde\vartheta(x^{1+\sigma})=1$ for all 
$x\in O_K^{\times}$, then we have
$$
 1=\nu(x^{\tau})\widetilde\vartheta(x^{\tau(1+\sigma)})
  =\nu(x)\widetilde\vartheta(x^{1+\sigma})^{-1},
$$
and hence
$$
 \widetilde\vartheta(x^{2\sigma})
 =\widetilde\vartheta(x^{-2})
 =\widetilde\vartheta(x^{2\tau})
$$
for all $x\in O_K^{\times}$. Since $x\mapsto x^2$ gives a surjection 
of $1+\frak{p}_K$ onto $1+\frak{p}_K$, we have 
$\widetilde\vartheta(x^{\sigma})=\widetilde\vartheta(x^{\tau})$ for
all $x\in 1+\frak{p}_K$, and hence $\sigma=\tau$. Since
$$
 \alpha_{K/F}(\tau,\tau)^{\tau}\alpha_{K/F}(\tau,1)^{-1}
 \alpha_{K/F}(1,\tau)\alpha_{K/F}(\tau,\tau)^{-1}
 =1
$$
and $\alpha_{K/F}(1,\tau)=\alpha_{K/F}(\tau,1)=1$, we have
$$
 \widetilde\vartheta\left(\alpha_{K/F}(\tau,\tau)^{\gamma}\right)
 =\vartheta\left(\alpha_{K/F}(\tau,\tau)^{\gamma(1-\tau)}\right)
 =1
$$
fro all $\gamma\in\Gamma$. Then we have
\begin{align*}
 |G|^{-1}\sum_{g\in G}\nu(g)\chi_{\vartheta}(g^2)
 &=\left|(O_K/\frak{p}_K^t)^{\times}\right|^{-1}\nu(\tau)
   \sum_{\dot x\in(O_K/\frak{p}_K^t)^{\times}}\nu(x)\\
 &=\begin{cases}
    0&:\nu|_{O_K^{\times}}\neq 1,\\
    \nu(\tau)&:\nu|_{O_K^{\times}}=1.
   \end{cases}
\end{align*}
This completes the proof.
\end{proof}

Sendai 980-0845, Japan\\
Miyagi University of Education\\
Department of Mathematics
\end{document}